\documentclass[reqno]{amsart}

\usepackage{tikz-cd}
\usetikzlibrary{knots}
\usetikzlibrary{positioning,arrows.meta}
\tikzset{
  dynkin node/.style={circle,draw,minimum size=6.2mm,inner sep=0pt,font=\small},
  dynkin sub/.style={dynkin node,fill=black!12},
  dynkin edge/.style={line width=.55pt}
}

\usepackage{fullpage,dsfont}

\usepackage{zref-clever}

\usepackage{hyperref} 

\usepackage{epigraph}

\usepackage{parskip}
\makeatletter 
\def\thm@space@setup{%
 \thm@preskip=\parskip \thm@postskip=0pt
}
\def\th@remark{%
  \thm@headfont{\itshape}%
  \normalfont 
  \thm@preskip\parskip \thm@postskip=0pt
}
\makeatother

\usepackage[nobysame,alphabetic,initials,msc-links]{amsrefs}

\usepackage[final]{pdfpages}

\usepackage{multirow}

\usepackage{array}

\usepackage{tikz-cd}
\usetikzlibrary{knots}
\usepackage{spath3}

\usepackage{float}

\DefineSimpleKey{bib}{how}
\DefineSimpleKey{bib}{mrclass}
\DefineSimpleKey{bib}{mrnumber}
\DefineSimpleKey{bib}{fjournal}
\DefineSimpleKey{bib}{mrreviewer}

\renewcommand{\PrintDOI}[1]{%
  \href{http://dx.doi.org/#1}{{\tt DOI:#1}}%
}
\renewcommand{\eprint}[1]{#1}
\BibSpec{book}{%
    +{}  {\PrintPrimary}                {transition}
    +{.} { \PrintDate}                  {date}
    +{.} { \textit}                     {title}
    +{.} { }                            {part}
    +{:} { \textit}                     {subtitle}
    +{,} { \PrintEdition}               {edition}
    +{}  { \PrintEditorsB}              {editor}
    +{,} { \PrintTranslatorsC}          {translator}
    +{,} { \PrintContributions}         {contribution}
    +{,} { }                            {series}
    +{,} { \voltext}                    {volume}
    +{,} { }                            {publisher}
    +{,} { }                            {organization}
    +{,} { }                            {address}
    +{,} { }                            {status}
    +{,} { \PrintDOI}                   {doi}
    +{,} { \PrintISBNs}                 {isbn}
    +{}  { \parenthesize}               {language}
    +{}  { \PrintTranslation}           {translation}
    +{;} { \PrintReprint}               {reprint}
    +{.} { }                            {note}
    +{.} {}                             {transition}
    +{}  {\SentenceSpace \PrintReviews} {review}
}
\BibSpec{article}{%
    +{}  {\PrintAuthors}                {author}
    +{,} { \textit}                     {title}
    +{.} { }                            {part}
    +{:} { \textit}                     {subtitle}
    +{,} { \PrintContributions}         {contribution}
    +{.} { \PrintPartials}              {partial}
    +{,} { }                            {journal}
    +{}  { \textbf}                     {volume}
    +{}  { \PrintDatePV}                {date}
    +{,} { \issuetext}                  {number}
    +{,} { \eprintpages}                {pages}
    +{,} { }                            {status}
    +{,} { \PrintDOI}                   {doi}
    +{,} { \eprint}        {eprint}
    +{}  { \parenthesize}               {language}
    +{}  { \PrintTranslation}           {translation}
    +{;} { \PrintReprint}               {reprint}
    +{.} { }                            {note}
    +{.} {}                             {transition}
    +{}  {\SentenceSpace \PrintReviews} {review}
}
\BibSpec{collection.article}{%
    +{}  {\PrintAuthors}                {author}
    +{,} { \textit}                     {title}
    +{.} { }                            {part}
    +{:} { \textit}                     {subtitle}
    +{,} { \PrintContributions}         {contribution}
    +{,} { \PrintConference}            {conference}
    +{}  {\PrintBook}                   {book}
    +{,} { }                            {booktitle}
    +{,} { \PrintDateB}                 {date}
    +{,} { pp.~}                        {pages}
    +{,} { }                            {publisher}
    +{,} { }                            {organization}
    +{,} { }                            {address}
    +{,} { }                            {status}
    +{,} { \PrintDOI}                   {doi}
    +{,} { \eprint}        {eprint}
    +{}  { \parenthesize}               {language}
    +{}  { \PrintTranslation}           {translation}
    +{;} { \PrintReprint}               {reprint}
    +{.} { }                            {note}
    +{.} {}                             {transition}
    +{}  {\SentenceSpace \PrintReviews} {review}
}
\BibSpec{misc}{%
  +{}{\PrintAuthors}  {author}
  +{,}{ \textit}      {title}
  +{.}{ }             {how}
  +{}{ \parenthesize} {date}
  +{,} { available at \eprint}        {eprint}
  +{,}{ available at \url}{url}
  +{,}{ }             {note}
  +{.}{}              {transition}
}

\usepackage{amssymb, amsfonts, amsxtra, amsmath}
\usepackage{mathrsfs}
\usepackage{mathdots}
\usepackage{wasysym}
\usepackage[all]{xy}
\usepackage{bbm}
\usepackage{calc}
\usepackage{accents}

\numberwithin{equation}{section}

\newtheorem{Theorem}{Theorem}[section]

\newtheorem*{Theorem*}{Theorem}
\newtheorem{Def}[Theorem]{Definition}

\newtheorem{Lem}[Theorem]{Lemma}

\newtheorem{Prop}[Theorem]{Proposition}
\newtheorem{Cor}[Theorem]{Corollary}

\newtheorem{Rem}[Theorem]{Remark}

\newtheorem{Exa}[Theorem]{Example}

\newtheorem{maintheorem}{Theorem}

\newcommand\bp{\begin{proof}}
\newcommand\ep{\end{proof}}

\mathchardef\mhyph="2D

\DeclareMathOperator{\diag}{\mathrm{diag}}

\DeclareMathOperator{\Ad}{\mathrm{Ad}}

\DeclareMathOperator{\End}{\mathrm{End}}

\DeclareMathOperator{\id}{\mathrm{id}}

\DeclareMathOperator{\Tr}{\mathrm{Tr}}

\DeclareMathOperator{\Imm}{\mathrm{Im}}

\DeclareMathOperator{\Prod}{\prod}

\def\longra{\longrightarrow}
\def\longla{\longleftarrow}

\newcommand{\msD}{\mathscr{D}}

\newcommand{\mfb}{\mathfrak{b}}

\newcommand{\mfd}{\mathfrak{d}}

\newcommand{\mfg}{\mathfrak{g}}

\newcommand{\mfh}{\mathfrak{h}}

\newcommand{\mfm}{\mathfrak{m}}
\newcommand{\mfn}{\mathfrak{n}}

\newcommand{\mfsl}{\mathfrak{sl}}
\newcommand{\mfso}{\mathfrak{so}}
\newcommand{\mfsp}{\mathfrak{sp}}

\newcommand{\Bc}{\mathcal{B}}

\newcommand{\mcE}{\mathcal{E}}

\newcommand{\mcH}{\mathcal{H}}
\newcommand{\Hsp}{\mathcal{H}}
\newcommand{\Hc}{\mcH}
\newcommand{\Gsp}{\mathcal{G}}

\newcommand{\mcK}{\mathcal{K}}
\newcommand{\Kc}{\mcK}

\newcommand{\mcO}{\mathcal{O}}

\newcommand{\mbX}{\mathbf{X}}

\newcommand{\C}{\mathbb{C}}

\newcommand{\N}{\mathbb{N}}

\newcommand{\R}{\mathbb{R}}

\newcommand{\Z}{\mathbb{Z}}

\newcommand{\opp}{\mathrm{op}}

\newcommand{\sgn}{\mathrm{sgn}}

\newcommand{\Sb}{\mathbf{S}}
\newcommand{\Tb}{\mathbf{T}}

\newcommand{\br}{\mathrm{br}}

\newcommand{\eb}{\mathbf{e}}
\newcommand{\fb}{\mathbf{f}}

\newcommand{\Hb}{\mathbf{H}}

\newcommand{\Xb}{\mathbf{X}}

\newcommand{\Yb}{\mathbf{Y}}
\newcommand{\Zb}{\mathbf{Z}}
\newcommand{\Pb}{\mathbf{P}}

\newcommand{\mba}{\mathbf{a}}
\newcommand{\mbb}{\mathbf{b}}
\newcommand{\mbi}{\mathbf{i}}
\newcommand{\mbj}{\mathbf{j}}
\newcommand{\Ab}{\mathbf{A}}
\newcommand{\Bb}{\mathbf{B}}
\newcommand{\Cb}{\mathbf{C}}
\newcommand{\Db}{\mathbf{D}}
\newcommand{\Eb}{\mathbf{E}}
\newcommand{\Fb}{\mathbf{F}}
\newcommand{\Gb}{\mathbf{G}}
\newcommand{\Kb}{\mathbf{K}}
\newcommand{\Qb}{\mathbf{Q}}
\newcommand{\qb}{\mathbf{q}}
\newcommand{\Lb}{\mathbf{L}}
\newcommand{\kb}{\mathbf{k}}
\newcommand{\lb}{\mathbf{l}}
\newcommand{\hb}{\mathbf{h}}

\newcommand{\vb}{\mathbf{v}}

\newcommand{\xb}{\mathbf{x}}
\newcommand{\pb}{\mathbf{p}}
\newcommand{\aff}{\,\eta\,}

\newcommand{\wA}{\widetilde{A}}
\newcommand{\wB}{\widetilde{B}}
\newcommand{\wC}{\widetilde{C}}
\newcommand{\wD}{\widetilde{D}}

\newcommand{\wZ}{\widetilde{Z}}
\newcommand{\wbA}{\widetilde{\Ab}}
\newcommand{\wbB}{\widetilde{\Bb}}
\newcommand{\wbC}{\widetilde{\Cb}}
\newcommand{\wbD}{\widetilde{\Db}}
\newcommand{\wbX}{\widetilde{\Xb}}
\newcommand{\wbY}{\widetilde{\Yb}}
\newcommand{\wbZ}{\widetilde{\Zb}}

\newcommand{\blank}{\textrm{---}}

\title{Haar measure for totally positive split real semisimple quantum groups of classical type}

\author{K.\ De Commer}
\address{Vrije Universiteit Brussel}
\email{kenny.de.commer@vub.be}

\author{G.\ Janssens}
\address{Vrije Universiteit Brussel}
\email{geoffrey.janssens@vub.be}

\author{W.\ Malfait}
\address{Vrije Universiteit Brussel}
\email{wannes.malfait@vub.be}

\begin{document}

\begin{abstract}
Let $\hbar$ be a non-zero real number and put $q= e^{-\pi i \hbar}$. Let $G$ be a semisimple split real Lie group of classical type, with Lie algebra $\mfg$ and associated Borel subalgebra $\mfb \subseteq \mfg$. We show that there exists a locally compact quantum group $(L_{\hbar}(B),\Delta)$ which integrates the Borel part $U_q(\mfb)$ of the Drinfeld--Jimbo quantized enveloping algebra $U_q(\mfg)$, viewed as a Hopf $*$-algebra. We further compute the associated Haar weights and modular data for $(L_{\hbar}(B),\Delta)$. Up to an excess Cartan part, the Drinfeld double of  $(L_{\hbar}(B),\Delta)$ can be seen as an analytic quantization of the function algebra of the totally positive semigroup $G^+\subseteq G$.
\end{abstract}

\maketitle


\tableofcontents

\section*{Introduction}

In this paper, we prove the following theorem: 

\begin{maintheorem}\label{TheoA}
Let $\mfg_{\C} \in\{\mfsl_N,\mfso_N,\mfsp_N\}$ be a classical complex simple Lie algebra, with Borel subalgebra $\mfb_{\C}$ and associated set of simple roots labeled by a set $I$.  Let $\hbar\in \R^{\times}$ and put $q = e^{-\pi i \hbar}$. 

There exists a locally compact quantum group $(L_{\hbar}(B),\Delta)$, with the von Neumann algebra $L_{\hbar}(B)$ generated by affiliated positive (unbounded) operators $\{\Eb_i,\Kb_i\mid i\in I\}$, such that the following hold: 
\begin{enumerate}
\item The operators $\Eb_i,\Kb_i$ satisfy 
\begin{equation}\label{EqCoprodForm}
\Delta(\Kb_i) = \Kb_i\otimes \Kb_i,\qquad \Delta(\Eb_i) = \Eb_i \otimes 1 \boxplus \Kb_i \otimes \Eb_i,\qquad \forall i\in I,
\end{equation}
where $\boxplus$ is the \emph{form sum} of positive operators. 
\item When $q$ is not a root of unity, there exists a concrete faithful normal $*$-representation $L_{\hbar}(B) \subseteq \Bc(\Hsp)$ on a Hilbert space $\Hsp$, and a dense subspace $\msD \subseteq \Hsp$ in the common domain of, and stable under, the generators $\Eb_i,\Kb_i$, such that 
\[
\pi: U_q(\mfb_{\C}) \rightarrow \End(\msD),\qquad K_i \mapsto  (\Kb_i)_{\mid \msD},\quad E_i \mapsto  (\Eb_i)_{\mid \msD}
\]
makes $\msD$ into a faithful $U_q(\mfb_{\C})$-module, with $U_q(\mfb_{\C}) \subseteq U_q(\mfg_{\C})$ the Borel part of the Drinfeld--Jimbo quantized enveloping algebra of $\mfg_{\C}$ (see e.g.\ \cite{Jan96} for definitions).
\end{enumerate}
\end{maintheorem} 

In other words, the above theorem provides an integrated version of $(U_q(\mfb_{\C}),\Delta)$ with its split real structure when $|q|=1$. We stress however that the integration is sensitive to a precise choice of logarithm of $q$, and that the integrated version in fact exists for arbitrary  non-zero values of $\hbar$, even when the associated $q$ is a root of unity.  

Below, we give some more details and context on the ingredients of this theorem. 

\subsection*{Locally compact quantum groups}

Locally compact quantum groups were introduced in \cites{KV00,KV03}, and consist of a von Neumann algebra $M$ with coproduct $\Delta$ and invariant weights $\varphi,\psi$ (see Definition \ref{DefLCQGFormal}). They are analytic, operator algebraic versions of Hopf algebras. For example, if $G$ is a locally compact group, one can take $M = L^{\infty}(G)$, with $\Delta$ dual to group multiplication and $\varphi,\psi$ corresponding to integration with respect to the left and right Haar measure. As another example, one can take the group von Neumann algebra $M= L(G)$ of $G$ with the diagonal coproduct and $\varphi = \psi$ equal to the \emph{Plancherel measure}. 

Locally compact quantum groups offer a vast generalisation of the classical theory of locally compact groups and their representation theory. For example, the fact that they admit a generalized \emph{Pontryagin duality} makes that they can serve, in the least, as a tool to unify various constructions in the classical setting, such as group actions, Fell bundles, crossed product operator algebras and their Takesaki--Takai duality. Their interest goes however well beyond this: already the subclass of \emph{compact quantum groups} \cites{Wor87,DK94,Wor98} and their duals, the \emph{discrete quantum groups} \cites{ER94,VD96}, offer the proper global framework to study (the compact forms of) \emph{Drinfeld--Jimbo quantum groups} in the operator algebraic framework \cite{LS91} (see also below), whereas their associated complexifications can be treated via the \emph{Drinfel'd double construction} (see \cite{VY20} for an overview).

Although the latter `complex quantum groups' offer genuine (and very interesting) examples of `truly quantum' locally compact quantum groups beyond the compact and/or discrete setting, they are somehow still `of algebraic type', in the sense that they are \emph{$*$-algebraic quantum groups}  \cite{VD98}. This entails that some of the associated structure, which is known to exist by the general formalism, trivializes. For example, the scaling constant of such quantum groups is always equal to $1$ \cite{DCDV10}. 

In \cite{DCSSV25}, the first examples of (higher-rank) locally compact quantum groups were constructed which were genuine $q$-deformations `in the continuous (non-algebraic) regime'. Namely, it was shown there that $SL_N^+(\R)$, the semigroup of totally positive $N$-by-$N$-matrices of determinant $1$ (see below), `quantizes' into  a locally compact quantum group. The main difficulty here is to establish the existence of the Haar weights for the associated quantum Borel subgroup. It is the latter result that we generalize beyond type $A$ in this paper. Below, we give some more brief context and history, referring to \cite{DCSSV25} for a more extensive discussion (especially concerning results in the rank $1$ case).

\subsection*{Totally positive part of split real Lie groups}

Put
\[
SL_N^+(\R) = \{x \in SL_N(\R)\mid \textrm{ all minors of }x\textrm{ are positive}\}.
\]
Then $SL_N^+(\R)$ is a semigroup by the Cauchy--Binet formula, but does not have inverses as these would introduce minus signs. Work of Lusztig showed that an analogous semigroup $G^+ \subseteq G$ exists for \emph{any} split real semisimple Lie group $G$ \cite{Lus94a}, with a lot of the structural results for Lie groups having analogues for this `totally positive part'. Here we will (abusively) interpret `positive' as meaning $\geq0$, so that we can make use of units, as well as of associated positive and negative Borel submonoids $B,B_- \subseteq G^+$. 

The main result of \cite{DCSSV25} shows that, despite $SL_N^+(\R)$ only being a semigroup, it admits an associated quantization $(L_{\hbar}^{\infty}(SL_N^+(\R)),\Delta)$ as a locally compact quantum \emph{group}! Intuitively, the reason for this miracle is the fact that antipodes for locally compact quantum groups, as for Hopf algebras, need not be involutive:  in the analytic setting they however admit a polar decomposition into a unitary and an absolute value part. It is this extra freedom that somehow allows signs, appearing when taking inverses, to be distributed over these two components in such a way that the positivity requirements for generators are not perturbed. What is remarkable is that this construction does not at all depend on any choice, but is simply a corollary of the general theory! This precise phenomena was already (implicitly or explicitly) understood before, see e.g.\ \cites{WZ02,BT03,Ip13a,FI14}.

The main result of this paper shows that the above quantization can also be performed for the other split real simple Lie groups of classical type: there exist locally compact quantum groups
\[
(L_{\hbar}(SO_{n,n}^+(\R)),\Delta),\quad (L_{\hbar}(SO_{n,n+1}^+(\R)),\Delta),\quad (L_{\hbar}(Sp_n^+(\R)),\Delta),\qquad \hbar\in \R^{\times},n \in \N.  
\]

We do not put this here as a formal theorem however, for the following reasons: 
\begin{itemize}
\item First of all, we do not provide any justification of the fact that we obtain a genuine quantization of $G^+$ on the analytic level. Rather, the quantization is to be interpreted in the sense that some associated algebraic structure admits a `$q$ tends to $1$'-limit, and that under this quantization certain particular positive functions on $G^+$, e.g.\ minors in the case of $SL^+_N(\R)$, get quantized into positive operators (similar to the formulation of Theorem \ref{TheoA}). The consideration of the classical limit would be very interesting, but is outside of the scope of this paper (see \cite{Ip13b} for some considerations in the rank $1$ case). 
\item Secondly, we only provide a quantization of the negative Borel part $B_-$ of $G^+$. Indeed, one can then use the Drinfeld double construction, which is available in the analytic setup \cite{BV05}, to construct the whole quantum group associated to $G^+$ (see e.g\ \cite{DCSSV25}*{Section 8}). Since this latter step would not add anything essential at the moment, we do not consider it in this paper. 
\item Finally, to quantize $B_-$ would mean to find an operator algebraic analogue $L_{\hbar}^{\infty}(B_-)$ of $\mcO_q(B_-)$, the Hopf $*$-algebra quantizing the $*$-algebra of regular functions on $B_-$. However, in the quantum setting one may interpret the above function algebra as a quantized enveloping algebra, due to \emph{Drinfeld duality} \cite{ES98}*{Section 10.2}: as Hopf algebras one has
\[
(\mcO_q(B_-),\Delta) \cong (U_q(\mfb_{\C}),\Delta),
\]
with $\mfb_{\C}$ the complexified Lie algebra of $B \subseteq G$. The positivity is carried along this isomorphism by asking that, in particular, the standard generators of $U_q(\mfb_{\C})$ are implemented by positive operators. Our main construction is then to find a locally compact quantum group `integrating' $U_q(\mfb_{\C})$ with the above real structure, and this is precisely the locally compact quantum group appearing in Theorem \ref{TheoA}. Following this line of reasoning, one may then as well interpret $(L_{\hbar}^{\infty}(B_-),\Delta)  = (L_{\hbar}(B),\Delta)$.  
\end{itemize}

\subsection*{Cluster algebras}

The theory of total positivity, as well as its quantization and dual theory of \emph{totally positive representations} (see e.g.\ \cites{FG06,FI14,SS17,Ip20}), is intimately connected to the theory of \emph{(quantum) cluster algebras} \cites{FZ02,BZ05,FG09}. Although this `cluster structure' is indispensable to probe the finer structure of our locally compact quantum groups, we will downplay their use in this paper, mainly to keep the length of the article within bounds. In fact, except for one key point (the notion of `full exchangeability'), the main result of this article will only rely on the starting point for quantum cluster algebra theory: the existence of an initial seed, i.e.\ the embeddability of the algebra of interest into a quantum torus.

\subsection*{Novel results in this paper}

Here we list some of the novel results and approaches in this paper, and compare it to what was known and done in the literature: 

\begin{itemize}
\item First of all, our construction does not put any restriction on the deformation parameter $\hbar$, which can be an arbitrary element of $\R^{\times}$. (Previous considerations sometimes excluded those $\hbar$ leading to $q= e^{-\pi i \hbar}$ being a root of unity.) 
\item Secondly, we are able to put our construction into a general framework: based on the notion of \emph{Feigin representation}, we associate a von Neumann algebraic torus to an arbitrary word over the index set of \emph{any} real Cartan matrix (i.e.\ the matrix does not need to be defined over the integers), together with certain positive operators affiliated to this quantum torus, the \emph{simple root operators}.  
 \item We isolate two key properties for the above quantum torus: 
\begin{itemize}
\item \emph{Integrability}: the simple root operators generate the quantum torus.
\item \emph{Full exchangeability}: there exists a collection of words which 
\begin{enumerate}
\item contains our initial word, 
\item is sufficiently large, in the sense that each index of the Cartan matrix is the beginning letter of at least one word in this collection, and 
\item consists of words which are all \emph{equivalent}, in the sense that the associated quantum tori can be identified in a way that preserves the simple root operators.
\end{enumerate}
\end{itemize}
We show that these two properties are sufficient to construct a locally compact quantum group.
\item We show that the above two properties are satisfied for particular words over the index set of a classical Cartan matrix. Some exceptional types could also be handled by our techniques, but integrability is not clear in the cases $F_4$ and $E_8$, while full exchangeability would require some more technical work in type $G_2$. 
\item A lot of work has been done on positive representation theory for $U_q(\mfg)$ in work of I. Ip \cites{Ip13a,Ip15a,Ip15b,Ip18a,Ip18b,Ip19,Ip20}. We borrow heavily from this work in terms of intuition. In particular, the notion of full exchangeability is a direct abstraction of certain techniques used in these papers. However, some of the results obtained in these works are insufficiently motivated from the analytic perspective, and transfer algebraic computations to the analytic realm without proper justification. Part of the aim of this paper is to put some of the computations appearing in these papers on a sound analytic footing, and to elucidate the key properties which allow for the locally compact quantum group structure to be present. In any case, no description of Haar measures for general type appear in these papers, and as such it was not yet established in this generality that these totally positive split real quantum groups fit in the framework of locally compact quantum groups.  
\item The notion of full exchangeability is directly motivated by the connection between Heisenberg algebras and multiplicative unitaries (and hence the pentagon identity), see Section \ref{SecQBvNBi}. The latter connection is well-known both in the algebraic setting (via the \emph{canonical} or \emph{Galois} map) and analytic setting, see e.g. \cites{BS93,Ka96,VD98,VVD03,MRW14}. Our particular implementation of this connection is directly inspired by the form for the pairing of dual quantum Borel algebras appearing in \cite{Ber96}, as well as by the computations in \cite{Ip18a}*{Section 7}. 
\item Finally, the theory of quantized enveloping algebras for $|q|=1$ is closely intertwined with the notion of the \emph{modular double quantum group} \cites{Fad00,KLS02,BT03}, which is related to a Langlands-like duality between the quantum groups at parameters $\hbar$ and $\hbar^{-1}$. Again we do not comment in greater detail on this connection within our paper, as it is not directly connected to or needed for our main goal. The duality can however be perceived from the concrete formulas, e.g.\ the scaling constant which appears in Section \ref{SecScaling}.
\end{itemize}

\subsection*{Outline of the paper}

In the \emph{first section} we present a more detailed version of Theorem \ref{TheoA}, and introduce precise definitions for the notions of `integrability' and `full exchangeability'. 

In the \emph{second section} we set up the proper formalism to consider skew-commuting positive operators in the operator algebraic framework, and recall in this context the \emph{quantum exponential function} \cite{Wor00} and its main properties, which is an indispensable tool in this framework.

In the \emph{third section} we briefly recall how one can associate a von Neumann algebra to a finite-dimensional vector space with a skew-symmetric form, and recall the main properties. This sets up the more flexible framework in which to embed the constructions of the first section. 

In the \emph{fourth section} we prove that any fully exchangeable word over a Cartan datum leads to a von Neumann algebra with coproduct, with the coproduct implemented by a \emph{fundamental unitary}. Under the extra assumption of \emph{integrability} of the word, we then show in the \emph{fifth section} that this von Neumann algebra with coproduct gives a locally compact quantum group. 

In the \emph{sixth section} we show that the longest word associated to a classical Cartan datum is fully exchangeable, and in the \emph{seventh section} we show that this word is integrable. Together with the results of the fifth section, this will prove our main theorem. 

In the \emph{eighth section} we provide some more information on the modular structure of the locally compact quantum groups constructed from a classical Cartan datum. 

In the \emph{ninth section}, we revisit the fundamental unitary constructed in the fourth section, and formulate conditions under which it becomes a \emph{multiplicative} unitary. We then provide some examples.

In the \emph{tenth section} we discuss an outlook. 

In \emph{Appendix} \ref{SecLemm} we prove a distributivity result for the form sum with respect to the star product, which is useful in the general context of the theory. In \emph{Appendix} \ref{SecLemmTypeB} we prove a type $B$ version of the Kashaev pentagon equation, which, as far as we know, had not been rigourously derived yet in the operator algebraic framework.

\subsection*{Preliminaries and conventions}

If $x$ is an invertible element of an algebra, we denote $\Ad_x(y) = xyx^{-1}$.

For the basic theory of von Neumann algebras and their weight theory, we refer to \cites{Tak02,Tak03} (see also the appendices to \cite{DCSSV25}).  We denote by $\odot$ the algebraic tensor product between vector spaces, by $\otimes$ the completed tensor product between Hilbert spaces, and by  $\overline{\otimes}$ the spatial tensor product between von Neumann algebras. We denote by $M_*$ the predual of a von Neumann algebra $M$, and by $S_*(M)$ its set of normal states.

We denote the domain of an unbounded operator $\Tb$ on a Hilbert space $\Hsp$ as $\msD(\Tb) \subseteq \Hsp$.   If $J$ is a finite set, $\{\Hsp_a\}_{a\in J}$ a collection of Hilbert spaces and $\Tb$ a closed unbounded operator on $\Hsp_a$ for some $a\in J$, we write
\[
\Tb_a \cong \left(\underset{b\neq a}{\otimes} 1\right) \otimes \Tb
\]
as a closed unbounded operator on $\otimes_{b\in J}\Hsp_b \cong (\otimes_{b\in J,b\neq a}\Hsp_b)\otimes \Hsp_a$, where the right hand side denotes the closure of the operator $\left(\underset{b\neq a}{\otimes} 1\right) \odot \Tb$ on $(\otimes_{b\in J,b\neq a}\Hsp_b)\odot \msD(\Tb)$.

By default, a \emph{positive operator}  (on a Hilbert space) means a (possibly unbounded) self-adjoint operator with positive spectrum and zero kernel. We call a pair $\Ab,\Bb$ of positive operators on a Hilbert space $\Hsp$ \emph{form-summable} if $\msD(\Ab^{1/2})\cap \msD(\Bb^{1/2})$ is dense in $\Hsp$. We then define their  \emph{form sum} $\Ab\boxplus \Bb$ as the unique positive operator such that 
\[
\msD((\Ab\boxplus \Bb)^{1/2}) = \msD(\Ab^{1/2})\cap \msD(\Bb^{1/2}),
\]
\[
\|(\Ab\boxplus \Bb)^{1/2}\xi\|^2 = \|\Ab^{1/2}\xi\|^2+ \|\Bb^{1/2}\xi\|^2,\qquad \forall \xi\in \msD(\Ab^{1/2})\cap \msD(\Bb^{1/2})  
\]
(see e.g.\ \cite{RS81}*{Section VIII.6}, \cite{Tak03}*{Section A.9} or \cite{DCSSV25}*{Section F.3}). Note that $\boxplus$ is a commutative and associative operation: 
\begin{equation}
\Ab\boxplus \Bb = \Bb\boxplus \Ab,\qquad (\Ab\boxplus \Bb)\boxplus \Cb = \Ab\boxplus (\Bb\boxplus \Cb),
\end{equation}
where the latter equation pre-supposes that $\msD(\Ab^{1/2})\cap \msD(\Bb^{1/2})\cap \msD(\Cb^{1/2})$ is dense in $\Hsp$. In the algebraic analogue, $\boxplus$ corresponds to the usual sum. In fact, since one always has $\msD(\Ab)\subseteq \msD(\Ab^{1/2})$, we see that $\Ab$ and $\Bb$ are form-summable as soon as $\msD(\Ab)\cap \msD(\Bb)$ is dense in $\Hsp$. In this case $\Ab \boxplus \Bb$ restricts to $\Ab+\Bb$ on $\msD(\Ab)\cap \msD(\Bb)$. However, $\msD(\Ab)\cap \msD(\Bb)$ can fail to be a core for $\Ab\boxplus \Bb$.

If $\Tb$ is a closed unbounded operator on a Hilbert space $\Hsp$, and  $M \subseteq \Bc(\Hsp)$ a von Neumann algebra, we write 
\[
\Tb \aff M
\]
if $\Tb$ is \emph{affiliated} with $M$. We recall that this is equivalent with 
\begin{equation}\label{EqCharAff}
u\Tb u^* = \Tb\qquad \textrm{for all unitaries }u \in M',\qquad M' = \{y \in \Bc(\Hsp)\mid xy = yx\textrm{ for all }x\in M\}. 
\end{equation}
If $\Tb$ is selfadjoint, we also have 
\begin{equation}
\Tb \aff M\quad \iff \quad f(\Tb) \in M,\qquad \forall \textrm{ bounded Borel functions }f: \R\rightarrow \C,
\end{equation}
and if $\Tb$ is positive we also have 
\begin{equation}\label{EqAffOnePar}
\Tb \aff M\quad \iff \quad \Tb^{it} \in M,\qquad \forall t\in \R.
\end{equation}

We  note that if form-summable positive operators $\Ab,\Bb$ are affiliated with a von Neumann algebra $M$, their form sum is again affiliated to $M$ (using \eqref{EqCharAff} and the definition of the form sum): 
\begin{equation}\label{EqAffForm}
\Ab \aff M,\Bb\aff M \quad \Rightarrow \quad \Ab\boxplus \Bb \aff M.
\end{equation}
In fact, for any unitary $u$ and form-summable positive operators $\Ab,\Bb$ on a Hilbert space, it is easily seen by the definition that $u\Ab u^*$ and $u\Bb u^*$ are again form-summable, with
\begin{equation}\label{EqPropConjFormSum}
u(\Ab\boxplus \Bb)u^* = u\Ab u^*\boxplus u \Bb u^*.
\end{equation}
A similar argument (using \cite{DCSSV25}*{Lemma F.1}) gives that if $\Ab,\Bb$ are form-summable and affiliated with a von Neumann algebra $M \subseteq \Bc(\Hsp)$, then for any normal unital $*$-homomorphism 
\[
\Phi: M \rightarrow N
\]
into another von Neumann algebra $N \subseteq \Bc(\Gsp)$, it will hold that $\Phi(\Ab)$ and $\Phi(\Bb)$ are form-summable, with 
\[
\Phi(\Ab\boxplus \Bb) = \Phi(\Ab) \boxplus \Phi(\Bb). 
\]

If $\{\Tb\}$ is a collection of selfadjoint operators on a Hilbert space $\Hsp$, we denote 
\[
\{\Tb\}' := \{x\in \Bc(\Hsp)\mid x\textrm{ strongly commutes with all }\Tb\},
\]
where we recall that $x$ strongly commutes with $\Tb$ if $\Tb \aff \{x,x^*\}'$.  Then $\{\Tb\}'$ and hence $\{\Tb\}''$ is a von Neumann algebra, and all $\Tb$ are affiliated with $\{\Tb\}''$. If $M$ is a von Neumann algebra and $\{\Tb\}$ is a collection of selfadjoint operators affiliated with $M$, we say that $\{\Tb\}$ \emph{generates} $M$ if $M = \{\Tb\}''$. 

If $H = (H,+)$ is a locally compact abelian group, we denote 
\begin{equation}\label{EqDefGenGroupvN}
\lambda_h: L^2(H) \rightarrow L^2(H),\quad (\lambda_hf)(g) = f(g-h),\qquad h,g\in H,f\in L^2(H) 
\end{equation}
for the regular unitary $H$-representation, and 
\[
L(H) = \{\lambda_h \mid h \in H\}'' \subseteq \Bc(L^2(H))
\]
for its group von Neumann algebra. More generally, if $M\subseteq \Bc(\Hsp)$ is a von Neumann algebra with a continuous $H$-action $H \overset{\gamma}{\curvearrowright}M$, we define the \emph{cross product von Neumann algebra}
\begin{equation}\label{EqCanConstCross}
H \ltimes M  = H\ltimes_{\gamma} M = (\{\lambda_h \otimes 1\mid h\in H\} \cup \{\pi_{\ltimes} (m) \mid m \in M \})'' \subseteq \Bc(L^2(H) \otimes \Hsp),
\end{equation}
where under the natural identification $L^2(H)\otimes \Hsp\cong L^2(H,\Hsp)$ we put
\begin{equation}
(\pi_{\ltimes}(m)f)(g) = \gamma_{-g}(m)f(g),\qquad f\in L^2(H,\Hsp),g\in H,m\in M.
\end{equation}
In the following we identify 
\begin{equation}\label{EqCrossIncl}
L(H),M \subseteq  H\ltimes M,\qquad \lambda_h \cong \lambda_h \otimes 1,\quad m \cong \pi_{\ltimes}(m),\qquad h \in H,m\in M. 
\end{equation}
One can find $M$ back from $H\ltimes M$ using \emph{slice maps}: 
\begin{equation}\label{EqSliceCrossed}
M = \sigma\textrm{-weak closure of }\{(\omega \otimes \id)x\mid x\in H\ltimes M,\omega \in \Bc(L^2(H))_*\},
\end{equation}
cf.\ \cite{Tak03}*{Theorem X.2.3.(iii)}.

\textbf{Acknowledgments}: KDC thanks his coauthors from \cite{DCSSV25} for the many discussions on quantum cluster algebras. We also thank Simon Lentner, Gus Schrader and Christian Voigt specifically for some brief email discussion on the content of Remark \ref{RemScalingConstant}. The research of WM was supported by Fonds voor Wetenschappelijk Onderzoek (Flanders), via an FWO Aspirant fellowship, grant 1104925N. The research of GJ was partially supported by Fonds voor Wetenschappelijk Onderzoek (Flanders), via an FWO senior postdoc fellowship, grant 12V0722N and via the Fonds de la Recherche Scientifique -- FNRS (grant 1.B.239.22). 

\textbf{AI use statement}: No LLMs were used for this manuscript nor any of the results it contains.

\section{Statement of the main theorem}\label{SecMainTheo}

\begin{Def}
A (real) \emph{Cartan datum} $I = (I,\widehat{A})$ consists of a finite set $I$ and a symmetric  matrix $\widehat{A} \in M_I(\R)$ with strictly positive diagonal elements:
\[
\widehat{A} = (\langle i,j\rangle)_{i,j} \in M_I(\R),\qquad 
\langle i,i\rangle > 0,\qquad \forall i \in I.
\]

We then write 
\[
D = \diag(d_i)_{i\in I},\qquad d_i = \frac{1}{2}\langle i,i\rangle \in (0,\infty),\qquad 
A = (a_{ij})_{i,j\in I},\qquad a_{ij} = 2\frac{\langle i,j\rangle}{\langle i,i\rangle}\in \R,
\]
so in particular $\widehat{A} = DA$. We call $A$ the associated \emph{Cartan matrix}, and $d_i$ the \emph{associated multiplicities}.

We call $I$ \emph{non-degenerate} if $\det(\widehat{A})\neq 0$.

We call $I$ \emph{of (finite) Lie type} if $\widehat{A}$ is positive-definite and $A \in M_I(\Z)$ with $a_{ij}<0$ for all $i\neq j$. In this case, $A$ can be put into blocks which correspond to the irreducible Cartan matrices of types $A$ through $G$.  

We call $I$ \emph{classical} if $I$ is of Lie type, and none of the $E,F,G$-types appear in its block decomposition. 
\end{Def}

Below, the Cartan datum $I$ will be kept fixed. We then write 
\begin{equation}\label{EqCart1}
\mfh = \mfh(I) = \R[I] = \{\sum_{i\in I} c_i i \mid c_i \in \R\} = \textrm{ free real vector space spanned by }I,
\end{equation}
and we extend $\langle \blank,\blank\rangle$ bilinearly to $\mfh$. We write
\begin{equation}\label{EqCart2}
\mfh' = \mfh(I)' =  \textrm{ real vector space dual of }\mfh,
\end{equation}
and we also write by  $\langle \blank,\blank\rangle$ the natural pairing between $\mfh'$ and $\mfh$. Note that we have a natural map 
\begin{equation}\label{EqMapToDual}
\alpha: \mfh \rightarrow\mfh',\qquad v \mapsto \alpha_v = \langle v, \blank\rangle.
\end{equation}

\begin{Def}
We write $m_{ij} \in \{1,2,\ldots\}\cup \{\infty\}$ for the order of $r_ir_j$, where 
\begin{equation}\label{EqElRefl}
r_i: \mfh\rightarrow \mfh,\quad v \mapsto v - 2 \frac{\langle i,v\rangle}{\langle i,i\rangle}i.
\end{equation}
We denote by $(W(I),S)$ the \emph{Coxeter system} $(W(I),S)$ with 
\[
S = \{s_i \mid i \in I\},\qquad W(I) = \langle s_i \in S \mid s_i^2 = e, (s_is_j)^{m_{ij}} = e \textrm{ for all }i\neq j\textrm{ and }m_{ij}<\infty\rangle,
\]
and we denote by $\ell$ its associated length function.
\end{Def}
We recall that $W(I)$ is finite if $I$ is of Lie type,  and coincides then with the group generated by the $r_i$. We denote in this case 
\[
w_0 := \textrm{ (unique) longest element of }W(I).
\]

\begin{Def}
A \emph{(formal) word} $\mbi$ in $I$ consists of a finite totally ordered set 
\begin{equation}\label{EqTotal1}
J = J(\mbi) = \{a\} = \{a_1<a_2<\dotsb<a_N\}
\end{equation}
together with a map 
\begin{equation}\label{EqTotal2}
\mbi: J \rightarrow I,\qquad a \mapsto \mbi(a) = i_a.
\end{equation}
We refer to the cardinality $N$ of $J$ as the \emph{length} $\ell(\mbi)$ of the word $\mbi$.  

We call $\mbi$ \emph{surjective} if \eqref{EqTotal2} is surjective.
\end{Def}
We write the \emph{level sets} of a word as 
\begin{equation}
J_i = \{a\in J \mid i_a = i\},\qquad i\in I,
\end{equation}
and equip them with the restriction of the total order. 

Whenever convenient, we identify $J = \{1,2,\ldots, N\}$. Then a word is indeed just a string of elements in $I$: 
\[
\mbi =  (i_1,i_2,\ldots,i_N).
\]
However, having an abstract labeling for $J$ is sometimes useful for disambiguation purposes. Below, we fix $J$ as the standard notation for the set underlying the word $\mbi$, when $\mbi$ is clear from the context.

\begin{Rem}
One can usually assume surjectivity of a word by restricting the alphabet $I$ to the image of $\mbi$.
\end{Rem} 

\begin{Def} 
If $\mbi = (i_1,i_2,\ldots,i_N)$ is a word in $I$, we refer to 
\begin{equation}\label{EqWordWeyl}
w_{\mbi} = s_{i_1}\ldots s_{i_N} \in W(I)
\end{equation}
as the associated Coxeter group element. We then call $\mbi$ \emph{reduced} if $\ell(\mbi) = \ell(w_{\mbi})$.
\end{Def}

Fix now $\hbar \in \R^{\times}$, and consider $L^2(\R)$ with its two positive operators $\kb,\eb$ such that  
\begin{equation}\label{EqGenke}
(\kb^{it}f)(x) = e^{i t x}f(x),\qquad (\eb^{it}f)(x) = f(x-\pi \hbar t),\qquad f\in L^2(\R),x,t\in \R, 
\end{equation}
i.e.\ $\kb,\eb$ are exponentials of (suitable multiples of) the position and momentum operator. Note that 
\[
\kb^{it}\eb^{is} = e^{\pi i \hbar st}\eb^{is}\kb^{it},\qquad s,t\in \R.
\]
This is the analytic implementation of the algebraic commutation relation
\[
ke = qek,\qquad q = e^{-\pi i \hbar}. 
\]

\begin{Def}\label{DefNilpvN}
To any word $\mbi$ we associate the Hilbert space 
\[
\Hsp(\mbi) = \otimes_{J} L^2(\R).
\]
We consider on $\Hsp(\mbi)$ the positive operators 
\begin{equation}\label{EqXVar}
\Xb_a = \Xb_a(\mbi) = \left(\underset{b<a}{\otimes} \kb^{\langle i_a,i_b\rangle}\right) \otimes \eb \otimes \left(\underset{b>a}{\otimes}1\right) = \left(\prod_{b<a} \kb_b^{\langle i_a,i_b\rangle}\right)\eb_a, \qquad a \in J,
\end{equation}
where in the last expression we take the closure of the (strongly commuting) product. We denote the von Neumann algebra generated by the $\Xb_a$ as 
\[
L_{\hbar}(\mfn(\mbi)) = \{\Xb_a\mid a\in J\}'' \subseteq \Bc(\Hsp(\mbi)). 
\]

We further define the \emph{simple root operators} associated to $\mbi$ to be the positive operators
\begin{equation}\label{EqGenQUE}
\Eb_i = \Eb_i(\mbi):=  \boxplus_{a \in J_i} \Xb_a,\qquad i\in I. 
\end{equation}
We denote the von Neumann algebra generated by the $\Eb_i$ as
\[
L_{\hbar}(N(\mbi)) = \{\Eb_i\mid i\in I\}'' \subseteq \Bc(\Hsp(\mbi)),
\]
and refer to it as the \emph{quantum nilpotent von Neumann algebra} associated to $\mbi$.
\end{Def}

Note that 
\begin{equation}\label{EqDenseSubsp}
\mcE = \{p(x)e^{-x^2/2} \mid p \textrm{ polynomial}\}
\end{equation}
is a dense subspace of $L^2(\R)$, and in the domain of any (complex) power of $\kb$ or $\eb$. Hence $\odot_J\mcE$ lies in the domain of all $\Xb_a$, and so the $\Eb_i$ in \eqref{EqGenQUE} are well-defined. Also note that the $\Eb_i$ are affiliated to $L_{\hbar}(\mfn(\mbi))$ by \eqref{EqAffForm}, hence 
\begin{equation}\label{EqInclBorelInf}
L_{\hbar}(N(\mbi)) \subseteq L_{\hbar}(\mfn(\mbi)).
\end{equation}

Further note that the von Neumann algebra $L_{\hbar}(\mfn(\mbi))$ has a natural continuous action 
\begin{equation}\label{EqActGamma}
\mfh' \overset{\gamma}{\curvearrowright} L_{\hbar}(\mfn(\mbi)),
\end{equation}
uniquely determined through
\begin{equation}
\gamma_{\omega}(\Xb_a) = e^{\pi \hbar \langle \omega,i_a\rangle}\Xb_a,\qquad a\in J,\omega \in \mfh'.
\end{equation}
Indeed, if we define on $\Hsp(\mbi)$ the (strongly commuting) positive operators
\begin{equation}\label{EqXVarCart}
\kb_{\omega} = \underset{b\in J}{\otimes} \kb^{\langle \omega,i_b\rangle} = \prod_{b\in J} \kb_b^{\langle \omega,i_b \rangle},\qquad \omega \in \mfh',
\end{equation} 
then $\gamma_{\omega}$ is determined through
\begin{equation}
\gamma_{t\omega}(x) = \Ad_{\kb_{\omega}^{it}}(x),\qquad x\in L_{\hbar}(\mfn(\mbi)), \omega \in \mfh', t\in \R. 
\end{equation}
From this description, we also see that the action $\gamma$ restricts to $L_{\hbar}(N(\mbi))$, with 
\begin{equation}
\gamma_{\omega}(\Eb_i) = e^{\pi \hbar \langle \omega,i\rangle}\Eb_i,\qquad i\in I,\omega \in \mfh'.
\end{equation}

We now introduce two fundamental notions for a word over a Cartan datum.

\begin{Def}
We call a word $\mbi$ \emph{integrable} if equality holds in \eqref{EqInclBorelInf}, i.e.\ when 
\begin{equation}\label{EqEqBorelInf}
L_{\hbar}(N(\mbi)) = L_{\hbar}(\mfn(\mbi)).
\end{equation}
\end{Def}

\begin{Def}\label{DefFullExch}
We call two words $\mbi,\mbi'$ \emph{equivalent} if there exists an $\mfh'$-equivariant $*$-isomorphism 
\begin{equation}\label{EqPhiEquivWord}
\Phi: L_{\hbar}(\mfn(\mbi)) \rightarrow L_{\hbar}(\mfn(\mbi')) \qquad \textrm{with}\qquad 
\Phi(\Eb_i(\mbi)) = \Eb_i(\mbi'),\qquad \forall i\in I.
\end{equation}

We  call a word $\mbi$ \emph{fully exchangeable} if for each $i \in I$ there exists an equivalent word $\mbi'$ with $i_1' = i$. 
\end{Def}

\begin{Theorem}\label{TheoClassicalIntFull}
Let $I$ be a classical Cartan datum, and $\mbi$ a word over $I$. 
\begin{enumerate}
\item If $\mbi$ is reduced, then $\mbi$ is integrable.
\item If $w_{\mbi} = w_0$, then $\mbi$ is fully exchangeable.
\end{enumerate} 
\end{Theorem} 

We  prove this theorem by showing these properties separately in respectively Theorem \ref{TheoExchangeable} (full exchangeability) and Theorem \ref{TheoIntegrable} (integrability).

We now proceed to show that to any integrable and fully exchangeable word, one can associate a locally compact quantum group. We first recall the definition of the latter. 

\begin{Def}\label{DefLCQGFormal}
A \emph{von Neumann bialgebra} is a von Neumann algebra $M$ equipped with a faithful, normal unital $*$-homomorphism
\[
\Delta: M \rightarrow M\bar{\otimes} M
\]
satisfying the \emph{coassociativity} condition
\begin{equation}\label{EqCoass}
(\Delta\otimes \id)\Delta = (\id\otimes \Delta)\Delta. 
\end{equation}

It is called a \emph{locally compact quantum group} (in the von Neumann algebraic setting \cite{KV03}) if there exist nsf (= normal, semi-finite, faithful) weights 
\[
\varphi,\psi: M^+ \rightarrow [0,\infty],
\]
referred to as \emph{Haar weights}, satisfying respectively the \emph{left invariance} and the \emph{right invariance} conditions with respect to $\Delta$: 
\begin{equation}\label{EqLeftInvariance}
\varphi((\omega\otimes \id)\Delta(x)) = \varphi(x),\qquad \forall x\in M^+,\omega \in S_*(M),
\end{equation}
\begin{equation}\label{EqRightInvariance}
\psi((\id\otimes \omega)\Delta(x)) = \psi(x),\qquad \forall x\in M^+,\omega \in S_*(M).
\end{equation}
\end{Def}

Note now that the action \eqref{EqActGamma} leads to an action $\mfh\curvearrowright L_{\hbar}(\mfn(\mbi))$, which we will denote by the same symbol: 
\[
\mfh \overset{\gamma}{\curvearrowright} L_{\hbar}(\mfn(\mbi)),\qquad \gamma_v = \gamma_{\alpha_v},\qquad v\in \mfh.
\]

\begin{Def}\label{DefBorelQA}
If $\mbi$ is a word in $I$, we define  the \emph{quantum Borel von Neumann algebra} associated to $\mbi$ to be the cross product von Neumann algebra 
\[
L_{\hbar}(B(\mbi)) := \mfh \ltimes_{\gamma} L_{\hbar}(N(\mbi)).
\]
\end{Def}

Note that, by construction,
\[
L_{\hbar}(N(\mbi))\subseteq L_{\hbar}(B(\mbi)). 
\]
Using the notation and conventions given by  \eqref{EqCrossIncl}, we define 
\begin{equation}\label{EqNewDefK}
\Kb_v \aff L(\mfh) \subseteq L_{\hbar}(B(\mbi)),\qquad \Kb_v^{it} = \lambda_{tv},\qquad t \in \R,v\in \mfh.
\end{equation}

\begin{Theorem}\label{TheoMain}
If $\mbi$ is a fully exchangeable  word, there exists a unique von Neumann bialgebra structure on $L_{\hbar}(B(\mbi))$ such that 
\[
\Delta(\Kb_i) = \Kb_i \otimes \Kb_i,\qquad \Delta(\Eb_i) = \Eb_i\otimes 1\boxplus \Kb_i \otimes\Eb_i,\qquad \forall i\in I.
\]
Moreover, if $\mbi$ is also integrable, then $(L_{\hbar}(B(\mbi)),\Delta)$ is a locally compact quantum group. 
\end{Theorem}

We prove the first part of Theorem \ref{TheoMain} as Theorem \ref{TheoExchvNbi}, and the second part as Theorem \ref{TheoIntLCQG}. The further structure of $(L_{\hbar}(B(\mbi)),\Delta)$ will be elaborated on in Section \ref{Sec8}.

If $I$ is a Cartan datum of classical type, the combination of Theorem \ref{TheoClassicalIntFull} and Theorem  \ref{TheoMain} gives us the locally compact quantum group $(L_{\hbar}(B),\Delta)$ mentioned in Theorem \ref{TheoA}. 

Let us at the end of this section comment on the content of the second part of Theorem \ref{TheoA}, which is mostly included for motivational purposes and follows from the well-known algebraic theory. We hence present the argument rather briefly.

Fix $I$ a Cartan datum of Lie type. Recall then first of all that, for $q$ not a root of unity, $U_q(\mfn_{\C})$ is the universal algebra generated by elements $\{E_i\mid i\in I\}$, satisfying as universal relations the \emph{quantum Serre relations} (see e.g.\ \cite{Lus94b}). Then $U_q(\mfb_{\C})$ is generated by $U_q(\mfn_{\C})$ together with a copy of a Laurent algebra, generated by elements $\{K_i^{\pm1}\mid i\in I\}$ which skew-commute with the $E_i$. One makes $U_q(\mfb_{\C})$ into a Hopf algebra by defining 
\[
\Delta(K_i) = K_i \otimes K_i,\qquad \Delta(E_i) = E_i\otimes 1+ K_i\otimes E_i
\]
(i.e.\ the algebraic analogue of \eqref{EqCoprodForm}). Since $|q|=1$, we get moreover a Hopf $*$-algebra structure by putting $E_i^*= E_i$ and $K_i^*= K_i$. It then is convenient to write $U_q(\mfb)$ when referring to $U_q(\mfb_{\C})$ with this particular $*$-structure, as this encodes the choice of a real form $\mfb$ of $\mfb_{\C}$.

Now for each $i\in I$ one has a quotient homomorphism 
\[
\pi_i: U_q(\mfb) \rightarrow A_q,\qquad E_j \mapsto \delta_{ij} e,\quad K_j \mapsto k^{\langle i,j\rangle},
\]
where $A_q$ is the universal algebra
\[
A_q := \C\langle k^{\pm1},e\rangle/(ke = qek).
\]
If then $\mbi = (i_1,i_2,\ldots,i_N)$ is a reduced word over $I$ with $w_{\mbi} = w_0$, and if we denote by $\Delta^{(N)}: U_q(\mfb)\rightarrow U_q(\mfb)^{\otimes N}$ the $N-1$-fold iterated coproduct, we obtain a homomorphism
\[
\pi_{\mbi} = (\pi_{i_1}\otimes \ldots \otimes \pi_{i_N})\circ \Delta^{(N)}: U_q(\mfb) \rightarrow A_q^{\otimes N},
\]
the \emph{Feigin homomorphism}, which can be shown to be faithful \cite{Jos95} (the faithfulness holds as well in the scalar case, and with the Cartan part added). However, the precise form of $\pi_{\mbi}$ entails that 
\[
\pi_{\mbi}(K_i) = \underset{b\in J}{\otimes} k^{\langle i,i_b\rangle},\qquad \pi_{\mbi}(E_i) = \sum_{a \in J_i} X_a,\qquad i\in I,
\]
where 
\[
X_a = \left(\underset{b<a}{\otimes} k^{\langle i_a,i_b\rangle}\right) \otimes e \otimes \left(\underset{b>a}{\otimes}1\right) \in A_q^{\otimes N}.
\]
Comparing with \eqref{EqXVar} and \eqref{EqXVarCart}, it now easily follows that we can realize the second part of Theorem \ref{TheoA} by choosing $\Hsp = \Hsp(\mbi)$ and picking $\msD = \odot_J \mcE$, with $\mcE$ as in \eqref{EqDenseSubsp}.

\begin{Rem}
Note that in this algebraic setting, the $K_i,X_a$ satisfy the (universal) commutation relations
\begin{equation}\label{EqAlgebraicCommGen}
K_iK_j = K_jK_i,\qquad X_b X_a = q^{\langle i_a,i_b\rangle}X_aX_b,\qquad b>a,\qquad K_iX_a = q^{\langle i,i_a\rangle} X_aK_i,\qquad i,j\in I,a,b\in J. 
\end{equation}
\end{Rem}

\section{Quantized Heisenberg commutation}\label{SecQHC}

\begin{Def}
Let $\hbar\in\R$. Let $\Ab,\Bb$ be positive operators on a Hilbert space. We say that $\Ab,\Bb$ \emph{$\hbar$-commute}, or that $(\Ab,\Bb)$ forms an $\hbar$-commuting pair, if
\begin{equation}\label{EqSkewComm}
\Ab^{it} \Bb^{is} = e^{2\pi i  \hbar st}\Bb^{is}\Ab^{it},\qquad \forall s,t\in \R. 
\end{equation}
\end{Def}
Algebraically, the $\hbar$-commutation between $\Ab$ and $\Bb$ corresponds to 
\begin{equation}
AB = q^2BA,\qquad q = e^{-\pi i \hbar},
\end{equation}
which however contains strictly less information than \eqref{EqSkewComm}, as it only recovers $\hbar$ up to an integer. On the other hand, \eqref{EqSkewComm} can hold for at most one $\hbar$, which we then write $\hbar = \hbar(\Ab,\Bb)$. If we want to let $\hbar$ be unspecified, we say that $\Ab,\Bb$ \emph{skew-commute}. If $\hbar(\Ab,\Bb) = 0$, this simply means that $\Ab,\Bb$ strongly commute. If $\hbar(\Ab,\Bb)\neq0$, we say that $\Ab,\Bb$ \emph{strictly} skew-commute.

The Weyl commutation relation \eqref{EqSkewComm} is viewed as an integrated version of the \emph{Heisenberg commutation relation} on the infinitesimal level: if $\mba = \ln(\Ab)$ and $\mbb = \ln(\Bb)$, then \eqref{EqSkewComm} corresponds algebraically to\footnote{We apologize for our mismatch by $2\pi$ of $\hbar$ with respect to the physicist's convention!}
\begin{equation}\label{EqOrdinHeis}
ab - ba = -2\pi i \hbar.
\end{equation}
With care, such commutation relation can also be formulated on the level of the unbounded selfadjoint operators $\mba,\mbb$, but is in general much harder to manipulate algebraically than \eqref{EqSkewComm}. 

Below, we consider similarly an integrated version for  \emph{quantized} Heisenberg relations. We first recall the following notation:  For any pair of skew-commuting positive operators $\Ab,\Bb$ on a Hilbert space $\Hsp$, one defines their \emph{balanced product} $\Ab\star\Bb$ as the (positive) generator of the unitary one-parametergroup 
\begin{equation}\label{EqStarOnePar}
(\Ab \star \Bb)^{it} := e^{-\pi i \hbar t^2} \Ab^{it} \Bb^{it},\qquad t\in \R, \hbar = \hbar(\Ab,\Bb). 
\end{equation}
Note that this operation is commutative: $\Ab\star\Bb = \Bb\star \Ab$. In the algebraic analogue, it corresponds to 
\[
AB = q^2BA \quad \Rightarrow \quad A \star B = q^{-1}AB = qBA,\qquad q = e^{-\pi i \hbar}. 
\]
For further use, we observe the following two facts: 
\begin{Lem}
If $\Ab,\Bb$ are positive operators affiliated to a von Neumann algebra $M \subseteq \Bc(\Hsp)$, then  $\Ab\star\Bb$ is affiliated to $M$. Moreover, if $\Phi: M \rightarrow N$ is a unital normal $*$-homomorphism into another von Neumann algebra $N$, we have 
\[
\Phi(\Ab\star \Bb) = \Phi(\Ab)\star\Phi(\Bb). 
\]
\end{Lem}
This follows immediately from the defining property \eqref{EqStarOnePar} and the observation \eqref{EqAffOnePar}.

We also have the following distributivity law.
\begin{Lem}\label{LemDistr}
Assume $\Ab,\Bb,\Cb$ are positive operators on a Hilbert space, with $\Bb,\Cb$ form-summable, with $\Ab,\Bb$ skew-commuting and with $\Ab,\Cb$ skew-commuting. Assume that $\hbar(\Ab,\Bb) = \hbar(\Ab,\Cb) = \hbar\in \R$. Then $\Ab$ and $\Bb\boxplus \Cb$ $\hbar$-commute, while  $\Ab\star\Bb$ and $\Ab\star \Cb$ are form-summable, and we have the identity 
\begin{equation}\label{EqCompDist}
\Ab\star(\Bb\boxplus \Cb) = \Ab\star\Bb\boxplus \Ab\star \Cb. 
\end{equation}
\end{Lem}
Note that this generalizes \cite{DCSSV25}*{Lemma 1.17} to the case where $\Bb,\Cb$ themselves do not necessarily skew-commute. The proof will be provided in Appendix \ref{SecLemm}. In fact, we will only need the version as in \cite{DCSSV25}*{Lemma 1.17}, but it seems useful to have the more general identity \eqref{EqCompDist} available for future purposes.

Fix now $\hbar \in \R^{\times}$. Recall that one can construct a particular continuous function
\begin{equation}
F_{\hbar}: [0,\infty) \rightarrow \C,
\end{equation}
called \emph{quantum exponential function} \cite{Wor00}, such that, for some real constant $\chi_{\hbar}$, 
\begin{equation}\label{EqPropsF}
F_{\hbar}(0) = 1,\qquad F_{-\hbar} = \overline{F}_{\hbar},\qquad |F_{\hbar}(t)|=1,\qquad F_{\hbar}(t^{-1})F_{\hbar}(t) = e^{-2i\chi_{\hbar}}e^{-i(\ln(t))^2/4\pi\hbar},\qquad \forall t\in \R_{>0},
\end{equation}
and such that the following two conditions hold for $\hbar$-commuting positive operators $\Ab,\Bb$: 
\begin{equation}\label{EqQuantExp}
F_{\hbar}(\Ab\boxplus \Bb) = F_{\hbar}(\Ab)F_{\hbar}(\Bb),
\end{equation}
\begin{equation}\label{EqPentKash}
F_{\hbar}(\Bb)F_{\hbar}(\Ab) = F_{\hbar}(\Ab)F_{\hbar}(\Ab\star\Bb)F_{\hbar}(\Bb)
\end{equation}
(see e.g.\ \cite{DCSSV25}*{Theorem 1.11, Theorem 1.16 and Appendix A} for more on these identities.) 
One refers to \eqref{EqQuantExp} as the \emph{(Woronowicz) quantum exponential property} and to \eqref{EqPentKash} as the \emph{(Kashaev) pentagon equation}. We recall as well the closely related identities
\begin{equation}\label{EqFormSumId}
\Ab\boxplus \Bb = F_{\hbar}(\Ab\star\Bb^{-1})^*\Bb F_{\hbar}(\Ab\star\Bb^{-1}) = F_{\hbar}(\Ab^{-1}\star \Bb)\Ab F_{\hbar}(\Ab^{-1}\star \Bb)^*,
\end{equation}
or equivalently, since also $(\Ab\star\Bb,\Bb)$ and $(\Ab,\Ab\star\Bb)$ are $\hbar$-commuting pairs,
\begin{equation}\label{EqFormSumIdAlt}
F_{\hbar}(\Ab)^*\Bb F_{\hbar}(\Ab) = \Ab\star \Bb\boxplus \Bb,\qquad F_{\hbar}(\Bb)\Ab F_{\hbar}(\Bb)^* = \Ab\boxplus \Ab\star\Bb,
\end{equation}
see \cite{DCSSV25}*{Proposition 1.14}.
We refer to e.g.\ \cite{DCSSV25}*{Appendix A and Appendix B} for some history and further properties of the function $F_{\hbar}$.  We will also need the following property: if $\Ab,\Bb$ $2\hbar$-commute (so algebraically $AB = q^4BA$), then 
\begin{equation}\label{EqPentTypeB}
F_{\hbar}(\Bb)F_{2\hbar}(\Ab) = F_{2\hbar}(\Ab)F_{\hbar}(\Ab\star\Bb) F_{2\hbar}(\Ab\star\Bb^2) F_{\hbar}(\Bb).
\end{equation}
This identity will be proven in Appendix \ref{SecLemmTypeB}.  

The following definition and theorem formalizes \cite{Ip15b}*{Proposition 3.1}

\begin{Def}\label{DefHeisComm}
Let $\Eb,\Fb$ be positive operators on a Hilbert space $\Hsp$. 

We say that $\Eb$ and $\Fb$ \emph{quantum Heisenberg commute} at $\hbar$ if there exists a positive operator $\Lb$ on $\Hsp$ such that $(\Lb,\Fb)$ and $(\Eb,\Lb)$ are $\hbar$-commuting, and such that \begin{equation}\label{EqPentagon}
F_{\hbar}(s\Fb)F_{\hbar}(t\Eb) = F_{\hbar}(t\Eb)F_{\hbar}(st\Lb)F_{\hbar}(s\Fb),\qquad \forall s,t\geq0.
\end{equation}

We then also say that $\Eb$ and $\Fb$ \emph{Heisenberg $\Lb$-commute} at $\hbar$, or that $(\Eb,\Fb)$ form a \emph{Heisenberg $\Lb$-pair} at $\hbar$. 
\end{Def}

\begin{Theorem}\label{TheoHeisPent}
Assume $\Lb,\Eb,\Fb$ are positive operators such that $(\Lb,\Fb)$ and $(\Eb,\Lb)$ are $\hbar$-commuting. Then the following are equivalent: 
\begin{enumerate}
\item $(\Eb,\Fb)$ form a Heisenberg $\Lb$-pair at $\hbar$.
\item We have
\begin{equation}\label{EqPentagonOr}
F_{\hbar}(\Eb)^*\Fb F_{\hbar}(\Eb) = \Lb \boxplus \Fb.
\end{equation}
\item We have
\begin{equation}\label{EqPentagonOr2}
F_{\hbar}(\Fb)\Eb F_{\hbar}(\Fb)^* = \Lb \boxplus \Eb.
\end{equation}
\end{enumerate}
\end{Theorem}
\begin{proof}
If \eqref{EqPentagonOr} holds, then also 
\[
F_{\hbar}(\Eb)^*s\Fb F_{\hbar}(\Eb) = s\Lb \boxplus s\Fb,\qquad \forall s\in \R_{\geq0}. 
\]
Conjugating with $\Lb^{it}$ for $t\in \R$, we get 
\[
F_{\hbar}(e^{-2\pi \hbar t}\Eb)^*se^{2\pi \hbar t}\Fb F_{\hbar}(e^{-2\pi \hbar t}\Eb) = s\Lb \boxplus se^{2\pi  \hbar t}\Fb,\qquad \forall s\in \R_{\geq0},
\]
so
\[
F_{\hbar}(t\Eb)^*s\Fb F_{\hbar}(t\Eb) = st \Lb \boxplus s\Fb,\qquad s,t\in \R_{\geq0}. 
\]
Applying the quantum exponential property, we find \eqref{EqPentagon}.

Similarly, we derive \eqref{EqPentagon} from \eqref{EqPentagonOr2}. 

Now conversely assume \eqref{EqPentagon} holds. By functional calculus and the quantum exponential property, we deduce
\[
F_{\hbar}(s F_{\hbar}(\Eb)^*\Fb F_{\hbar}(\Eb)) = F_{\hbar}(\Eb)^*F_{\hbar}(s\Fb)F_{\hbar}(\Eb) = F_{\hbar}(s(\Lb \boxplus \Fb)).
\]
But if $\Ab,\Bb$ are two positive operators with $F_{\hbar}(s\Ab) = F_{\hbar}(s\Bb)$ for all $s \geq0$, then $\Ab = \Bb$, cf. (the proof of) \cite{DCSSV25}*{Lemma 4.16}. Hence we find \eqref{EqPentagonOr}. The relation \eqref{EqPentagonOr2} is derived in a similar way.
\end{proof}

\begin{Rem}
Similarly using \cite{DCSSV25}*{Lemma 4.16}, we see that $\Lb$ in Definition \ref{DefHeisComm} is uniquely determined by $\Eb,\Fb$. 
\end{Rem}

\begin{Rem}
From e.g.\ \cite{Wor00}*{Equation (1.37)}, one sees that $F_{\hbar}$ is differentiable at $0$ for $\hbar$ small, with 
\[
F_{\hbar}'(0) = (e^{\pi i \hbar} -e^{-\pi i \hbar})^{-1} = (q^{-1}-q)^{-1}. 
\]
Differentiating \eqref{EqPentagon} at $s=0$ and $t=0$, we find the following infinitesimal equivalent of \eqref{EqPentagon}:  
\[
EF - FE = (q-q^{-1})L.
\]
This is a direct quantization of the usual Heisenberg commutation relation \eqref{EqOrdinHeis} (with $L$ playing the role of a `non-commutative $\hbar$'). 
\end{Rem}

\begin{Exa}\label{ExaCompHeis}
Assume that $\Ab,\Bb,\Lb$ are positive operators such that $(\Ab,\Bb)$, $(\Bb,\Lb)$ and $(\Lb,\Ab)$ form $\hbar$-commuting pairs. Put
\[
\Fb = \Ab,\qquad \Eb = \Lb\star \Ab^{-1} \boxplus \Lb \star \Bb. 
\]
Then $(\Eb,\Fb)$ form a Heisenberg $\Lb$-pair at $\hbar$. Indeed, clearly $(\Lb,\Fb)$ and $(\Eb,\Lb)$ are $\hbar$-commuting pairs. We further compute, using the quantum exponential property and the fact that $\Lb\star\Bb$ and $\Fb = \Ab$  strongly commute,
\begin{eqnarray*}
F_{\hbar}(\Eb)^*\Fb F_{\hbar}(\Eb) &=& 
F_{\hbar}(\Lb\star \Ab^{-1})^*F_{\hbar}(\Lb\star \Bb)^*\Fb F_{\hbar}(\Lb\star \Bb)F_{\hbar}(\Lb\star \Ab^{-1}) \\
&=& F_{\hbar}(\Lb\star \Ab^{-1})^*\Fb F_{\hbar}(\Lb\star \Ab^{-1}) \\
&\underset{\eqref{EqFormSumId}}{=}&\Lb \boxplus \Fb,
\end{eqnarray*}
so \eqref{EqPentagonOr} is satisfied.
\end{Exa}

\section{Quantum Borel von Neumann Algebras}\label{SecQuantBorel}

For the remainder of this section, we fix a deformation parameter $\hbar \in \R^{\times}$. We consider the constructions of Section \ref{SecMainTheo} in a more abstract framework. 

Let $V= (V,\epsilon)$ be a finite-dimensional skew-symmetric real vector space, i.e.\ $V$ is equipped with a skew-symmetric bilinear form
\[
\epsilon: V\times V \rightarrow \R.
\]
Recall that one can associate to $V$ a von Neumann algebra 
\begin{equation}\label{EqAssvN}
L_{\hbar}(V) := \{\qb(v)\mid v\in V\}'' \subseteq \Bc(L^2(V)),
\end{equation}
where, for $v\in V$, the positive operator $\qb(v)$ on $L^2(V)$ is defined via 
\begin{equation}
(\qb(v)^{it}g)(w) = e^{\pi i \hbar t \epsilon(v,w)}g(w-tv),\qquad g\in L^2(V),w\in V, t\in\R.
\end{equation}
\begin{Rem} 
Momentarily allowing $\hbar=0$, we indeed get that $L_{0}(V) = L(V)$, the group von Neumann algebra of $V$, with $\qb(v)^{it} = \lambda_{tv}$ for $v\in V$ and $t\in \R$.  
\end{Rem}

For any $v,w\in V$ we have that $\qb(v)$ and $\qb(w)$ form a $\hbar \epsilon(v,w)$-commuting pair, and 
\[
\qb(v+w) = \qb(v)\star\qb(w),\qquad v,w\in V.
\] 
We further write
\begin{equation}\label{EqNotationGen}
\vb = (-\pi i \hbar)^{-1}\ln(\qb(v)),\qquad v\in V,
\end{equation}
so $\vb$ is a skew-hermitian operator  affiliated with $L_{\hbar}(V)$. Algebraically/formally, we can then interpret
\[
q(v) = q^v =e^{-\pi i \hbar v},\qquad v\in V,
\]
with 
\[
q(v)q(w) = q^{2\epsilon(v,w)} q(w)q(v),\qquad v,w\in V.
\]

One easily shows that if $\{e_i\mid i\in I\}$ is  a basis of $V$, then  
\[
L_{\hbar}(V) = \{\qb(e_i)\mid i \in I\}'',
\]
and that if $W$ is another skew-symmetric space with basis $\{f_i \mid i \in I\}$ such that 
\[
\epsilon_V(e_i,e_j) = \epsilon_W(f_i,f_j),\qquad \forall i,j,
\]
there exists a unique $*$-isomorphism
\[
\Psi:L_{\hbar}(V) \rightarrow L_{\hbar}(W),\qquad \qb(e_i)\mapsto \qb(f_i)
\]
(one implements $\Psi = \Ad_U$, with $U:L^2(V) \rightarrow L^2(W)$ the unitary implementing the underlying vector space isomorphism).

Moreover, for any inclusion of skew-symmetric spaces $V \subseteq W$, one has a normal embedding of von Neumann algebras
\begin{equation}\label{EqVonNeum}
L_{\hbar}(V) \hookrightarrow L_{\hbar}(W),\qquad \qb(v) \mapsto \qb(v),\qquad v\in V,
\end{equation}
and we will usually simply interpret this as an actual inclusion $L_{\hbar}(V)\subseteq L_{\hbar}(W)$ (cf. \cite{DCSSV25}*{Lemma 4.5}).

We also note the following easy observation.

\begin{Lem}\label{LemDirSumSkew}
Let $W$ be a skew-symmetric space, and assume 
\begin{equation}\label{EqSkewDecompTriv}
W = V\oplus Z,\qquad \epsilon_{\mid Z\times Z} = 0.
\end{equation}
Then there exists a unique $*$-isomorphism 
\begin{equation}
Z \ltimes_{\gamma} L_{\hbar}(V)\cong L_{\hbar}(W),\quad \qb(v) \mapsto \qb(v),\quad \qb(z) \mapsto \qb(z),\qquad v\in V,z\in Z,
\end{equation}
where the action $\gamma$ of $Z$ on $L_{\hbar}(V)$ is uniquely determined through
\[
\gamma_z(\qb(v)) =e^{-2\pi \hbar \epsilon(v,z)}\qb(v),\qquad z\in Z,v\in V.
\]
\end{Lem} 
\begin{proof}
Uniqueness is clear, since $\{\qb(z),\qb(v)\mid z\in Z,v\in V \}'' = Z\ltimes_{\gamma} L_{\hbar}(V)$. 

Consider the unitary 
\[
U: L^2(Z)\otimes L^2(V) \rightarrow L^2(W),\qquad U(f\otimes g)(v+z) = e^{\pi i \hbar \epsilon(z,v)}f(z)g(v),\qquad z \in Z,v\in V,
\]
where we identify $Z \cong 0 \oplus Z$, resp.\ $V \cong V \oplus 0$. Then an easy computation shows that 
\[
U(\qb(z)^{it}\otimes 1)U^* = \qb(z)^{it},\qquad U \pi_{\ltimes}(\qb(v))^{it}U^* = \qb(v)^{it}
\]
(using the representations in \eqref{EqAssvN} and \eqref{EqCanConstCross}). This finishes the proof. 
\end{proof}

\begin{Rem}\label{RemDirSumExa}
Any skew-symmetric space $W$ of the form \eqref{EqSkewDecompTriv} is determined by a triple $((V,\epsilon_V),Z,\epsilon_{Z,V})$, consisting of a skew-symmetric vector space $V$ and a vector space $Z$, together with a bilinear pairing 
\[
\epsilon_{Z,V}: Z\times V \rightarrow \R.
\]
Indeed, then $W$ can be identified with the direct sum vector space $W:= V\oplus Z$, endowed with the unique skew-symmetric space structure such that $\epsilon_W$ restricts to $\epsilon_V$ on $V\times V$, to $\epsilon_{Z,V}$ on $Z\times V$, and to $0$ on $Z\times Z$.
\end{Rem} 

\begin{Exa} 
Write $Z =\R^2$, and write its standard basis vectors as $h,p$. Endow $Z$ with the skew-symmetric form 
\[
\epsilon(h,p) = 1/2. 
\]
Then by the Stone--von Neumann theorem, one has a $*$-isomorphism 
\[
L_{\hbar}(\R^2) \cong \Bc(L^2(\R)),\qquad \qb(h) \mapsto \kb,\quad \qb(p) \mapsto \eb,
\]
with $\kb,\eb$ as in \eqref{EqGenke}. 
\end{Exa}

\begin{Exa}
More generally, for any direct sum $Z^T$ over a finite index set $T$ (with the direct sum skew-symmetric form), we denote $z_a$ for the copy of $z\in Z$ at position $a\in T$. We can then identify 
\begin{equation}\label{EqLongTens}
L_{\hbar}(Z^T) = \overline{\otimes}_T\Bc(L^2(\R)),\qquad \qb(z_a) = \qb(z)_a = \left(\otimes_{b<a} 1\right) \otimes \qb(z)\otimes \left(\otimes_{b>a} 1\right).
\end{equation}
\end{Exa}

Consider now a Cartan datum $I$ and word $\mbi$ in $I$. Put 
\[
\mfn(\mbi) = \R[J],
\]
endowed with the unique skew-symmetric form such that
\begin{equation}\label{EqPairNilp}
\epsilon(b,a) = \langle i_b,i_a\rangle/2,\qquad \forall a,b\in J \textrm{ with }b>a.
\end{equation}
We get an embedding of skew-symmetric spaces 
\begin{equation}\label{EqEmbSkewSymm}
\mfn(\mbi) \subseteq Z^J,\qquad a \mapsto \left(\sum_{b<a} \langle i_a,i_b\rangle h_b\right)+ p_a,
\end{equation}
and so by \eqref{EqVonNeum} and \eqref{EqLongTens} we see that the abstract definition \eqref{EqAssvN} of $L_{\hbar}(\mfn(\mbi))$ indeed coincides with the one introduced in Definition \ref{DefNilpvN}. We call the specific realisation on $\Hsp(\mbi)$ as in Definition \ref{DefBorelQA} the \emph{tensor (or Feigin) realisation} of $L_{\hbar}(\mfn(\mbi))$.

Now if we consider the bilinear pairing
\begin{equation}\label{EqPairCart}
\mfh \times \mfn(\mbi)\rightarrow \R,\qquad \epsilon(i,a) = \langle i,i_a\rangle/2,\qquad i\in I,a\in J,
\end{equation}
we can make 
\begin{equation}
\mfb(\mbi) := \mfn(\mbi) \oplus \mfh
\end{equation}
into a skew-symmetric space via the construction in Remark \ref{RemDirSumExa}. We then refer to 
\begin{equation}
L_{\hbar}(\mfb(\mbi)) \cong \mfh \ltimes L_{\hbar}(\mfn(\mbi))
\end{equation}
as the \emph{quantum Borel von Neumann algebra} associated to $\mbi$. 

\begin{Rem}\label{RemPresNonDeg} 
If $I$ is non-degenerate and $\mbi$ is surjective, we can extend \eqref{EqEmbSkewSymm} to an inclusion  of skew-symmetric spaces 
\[
\mfb(\mbi) \subseteq Z^J,\qquad a \mapsto \left(\sum_{b<a} \langle i_a,i_b\rangle h_b\right)+ p_a,\qquad i \mapsto \sum_b \langle i,i_b\rangle h_b,\qquad a\in J,i\in I.
\]
It follows that, with the $\kb_{\omega}$ as in \eqref{EqXVarCart}, we can then identify 
\begin{equation}
L_{\hbar}(\mfb(\mbi)) \cong \{\kb_{\alpha_i},\Xb_a\mid i\in I,a\in J\}'' \subseteq \Bc(\Hsp(\mbi)),\qquad \Kb_i\mapsto \kb_{\alpha_i}.  
\end{equation}
\end{Rem}

\begin{Rem} If $I'$ is the image of the word $\mbi$, we have a natural decomposition 
\begin{equation}\label{EqSplittingOff}
L_{\hbar}(\mfb(\mbi)) \cong L_{\hbar}(\mfb(\mbi^{\mid I'})) \bar{\otimes} L(\R[I\setminus I']),
\end{equation}
where $\mbi^{\mid I'}$ is the corestriction of $\mbi$ to the codomain $I'$. 
\end{Rem}

\section{Quantum Borel von Neumann algebras as von Neumann bialgebras}\label{SecQBvNBi}

We fix again a parameter $\hbar \in \R^{\times}$ and a Cartan datum $I$. In this section we prove the first part of Theorem \ref{TheoMain}, namely:

\begin{Theorem}\label{TheoExchvNbi}
Let $\mbi$ be a fully exchangeable word. Then there exists a unique von Neumann bialgebra structure on $L_{\hbar}(B(\mbi))$ such that 
\[
\Delta(\Kb_i) = \Kb_i \otimes \Kb_i,\qquad \Delta(\Eb_i) = \Eb_i\otimes 1\boxplus \Kb_i \otimes\Eb_i,\qquad \forall i\in I.
\]
\end{Theorem}

Note that the unicity is immediate by the definition \eqref{DefBorelQA} of $L_{\hbar}(B(\mbi))$ (recall also the definition \eqref{EqNewDefK} of the $\Kb_i$). This gives as well that $\Delta$ will be coassociative, as it is coassociative on the generators $\Eb_i,\Kb_i$. We hence only need to prove existence of $\Delta$ as a normal faithful $*$-homomorphism.

Recall the notations $\mfh$ and $\mfh'$ from \eqref{EqCart1} and \eqref{EqCart2}. We view $\mfh'$ as a trivial skew-symmetric space, i.e.\ $\epsilon(\omega,\chi) = 0$ for all $\omega,\chi \in \mfh'$.  We denote the associated positive operators affiliated to $L(\mfh')$ as 
\[
\Lb_{\omega} = \qb(\omega),\qquad \omega\in \mfh',
\]
so $\Lb_{\omega}^{it} = \lambda_{t\omega}$ in the notation of \eqref{EqDefGenGroupvN}.

Recall now the map $\alpha: \mfh \rightarrow \mfh'$ from \eqref{EqMapToDual}. We will refer to
\begin{equation}\label{EqRoots}
\alpha_i = \alpha(i),\qquad i\in I
\end{equation}
as the \emph{simple roots} in $\mfh'$, while we denote 
\[
\{\varpi_i^{\vee}\mid i \in I\} \subseteq \mfh'
\]
for the dual basis to $\{i\}\subseteq \mfh$. We refer to the $\varpi_i^{\vee}$ as the \emph{fundamental coweights}.

We then abbreviate
\begin{equation}\label{EqAbbrevL}
\Lb_i = \Lb_{\alpha_i},\qquad   \Lb_i' = \Lb_{\varpi_i^{\vee}},\qquad i\in I. 
\end{equation}

\begin{Def}\label{DefRedHeisSpace}
Let $\mbi$ be a word over $I$. We define the \emph{(reduced) Heisenberg space} $\mfd = \mfd(\mbi)$ as the direct sum vector space $\mfb\oplus \mfh'$, equipped with the unique skew-symmetric form $\epsilon$ extending $\epsilon$ on $\mfb$ and $\mfh'$, and with
\begin{equation}\label{EqPairHeis}
\epsilon(i,\omega) = \langle i,\omega\rangle/2,\qquad \epsilon(a,\omega) = \langle i_a,\omega\rangle/2,\qquad \forall i\in I,a\in J,\omega \in \mfh'.  
\end{equation}
We then also put
\[
\mfb' = \mfb'(\mbi) = \mfn(\mbi)\oplus \mfh'\subseteq \mfd(\mbi). 
\]
\end{Def}

This gives us the von Neumann algebra $L_{\hbar}(\mfd(\mbi))$, and
\[
L_{\hbar}(\mfd(\mbi)) = \{\Kb_i,\Lb_i',\Xb_a\mid i\in I,a\in J\}''.
\]

\begin{Rem}
On the algebraic level, this means that the commutation relations \eqref{EqAlgebraicCommGen} get extended by the following commutation relations between the $L_{\omega}$ and the other generators:
\[
L_{\omega}L_{\chi} =L_{\chi}L_{\omega},\quad L_{\omega}K_i = q^{-\langle \omega,i\rangle}K_iL_{\omega},\quad L_{\omega}X_a = q^{-\langle \omega,i_a\rangle}X_aL_{\omega},\qquad \forall i\in I,a\in J,\omega,\chi \in \mfh'.
\]
\end{Rem} 

It will be convenient to use a specific realisation of $L_{\hbar}(\mfd(\mbi))$: by similar considerations as in Section \ref{SecQuantBorel}, we have a unique faithful normal $*$-representation of $L_{\hbar}(\mfd(\mbi))$ on 
\begin{equation}\label{EqTensDecompExt}
\widetilde{\Hsp}(\mbi) = \Hsp(\mbi) \otimes L^2(\mfh)
\end{equation}
via 
\begin{equation}\label{EqTensProdDecomp}
\Xb_a \mapsto \Xb_a\otimes 1,\quad \Kb_i \mapsto \left(\underset{a\in J}{\otimes} \kb^{\langle i,i_a\rangle}\right) \otimes \kb_i,\qquad \Lb_i' \mapsto  \left(\underset{a\in J_i}{\Prod} \kb_a^{-1}\right) \otimes\eb_i,
\end{equation}
and we see in particular that this gives a realisation
\begin{equation}\label{EqInclDoubleHeis}
L_{\hbar}(\mfd(\mbi)) \subseteq L_{\hbar}(\mfb(\mbi)) \bar{\otimes} \Bc(L^2(\mfh)).
\end{equation}

\begin{Def}\label{DefHeisDualOp}
Let $\mbi$ be a word over $I$. By a \emph{dual simple root operator} at $i\in I$, we mean a positive operator $\Fb_i$ affiliated to $L_{\hbar}(\mfd(\mbi))$, satisfying the following conditions: 
\begin{enumerate}
\item\label{EqCond1} $\Kb_j$ and $\Fb_i$ strongly commute for all $j\in I$,
\item\label{EqCond2} $\Lb_j'$ and $\Fb_i$ form a $\delta_{ij}\hbar/2$-commuting pair, for all $j\in I$,
\item\label{EqCond3} $\Eb_j$ and $\Fb_i$ strongly commute for all $j\neq i$, 
\item\label{EqCond4} $\Eb_i$ and $\Fb_i$ form a Heisenberg $\Lb_i$-pair at $d_i\hbar$,
\item\label{EqCond5} $\Fb_i \star \Lb_i^{-1}$ is affiliated with $L_{\hbar}(\mfn)$.
\end{enumerate}
We call a word $\mbi$ over $I$  \emph{dualizable} if it admits dual simple root operators for each $i\in I$.
\end{Def}
\begin{Rem}
On the algebraic level, \eqref{EqCond2} means that 
\[
L_{\omega}F_i =  q^{\langle \omega,i\rangle}F_iL_{\omega},\qquad \forall i\in I,\omega\in \mfh'.
\]
\end{Rem} 

Whenever $\mbi$ is a dualizable word, we usually simply pick for each $i\in I$ an associated dual simple root operator $\Fb_i$: we do not need to invoke any uniqueness of such dual simple root operators.

\begin{Prop}\label{PropExchDual} 
Let $\mbi$ be a fully exchangeable word. Then $\mbi$ is dualizable. 
\end{Prop}
\begin{proof}
We may, without loss of generality, assume that $\mbi$ is surjective. 

Pick then $i \in I$. We show that there exists a dual simple root vector at $i\in I$. 

Choose $\mbi'$ and $\Phi$ as in Definition \ref{DefFullExch}. By a similar reasoning as in Section \ref{SecQuantBorel}, we have an isomorphism 
\begin{equation}\label{EqIsoCrossWeight}
\mfh'\ltimes L_{\hbar}(\mfn(\mbi)) \cong \{\Xb_a,\kb_{\omega}\mid a\in J,\omega \in \mfh'\}'':= L_{\hbar}(\widetilde{\mfb}(\mbi)) \subseteq \Bc(\Hsp(\mbi)),
\end{equation}
with the $\kb_{\omega}$ as in \eqref{EqXVarCart}. A similar statement holds for $\mfh'\ltimes L_{\hbar}(\mfn(\mbi'))$. 

Then by equivariance of $\Phi$, we can extend $\Phi$ to a $*$-isomorphism 
\begin{multline*}
\Phi: L_{\hbar}(\widetilde{\mfb}(\mbi))\cong L_{\hbar}(\widetilde{\mfb}(\mbi')),\quad \Phi(L_{\hbar}(\mfn(\mbi)) = L_{\hbar}(\mfn(\mbi')),\quad \Phi(\kb_{\omega}(\mbi))= \kb_{\omega}(\mbi'),\quad \Phi(\Eb_i(\mbi))= \Eb_i(\mbi'),\\ \forall i\in I,\omega \in \mfh'. 
\end{multline*}
In fact, we then also obtain an isomorphism 
\[
\widetilde{\Phi} = \Phi\otimes \id: L_{\hbar}(\widetilde{\mfb}(\mbi)) \bar{\otimes}\Bc((\otimes_I L^2(\R))\rightarrow L_{\hbar}(\widetilde{\mfb}(\mbi')) \bar{\otimes}\Bc((\otimes_I L^2(\R)).
\]
With respect to the inclusion \eqref{EqTensDecompExt}, we clearly have 
\[
\widetilde{\Phi}(L_{\hbar}(\mfd(\mbi))) = L_{\hbar}(\mfd(\mbi')),\qquad \widetilde{\Phi}(L_{\hbar}(\mfn(\mbi))) = L_{\hbar}(\mfn(\mbi')),
\]
\[
\widetilde{\Phi}(\Kb_i(\mbi)) = \Kb_i(\mbi'),\qquad  \widetilde{\Phi} (\Eb_i(\mbi)) = \Eb_i(\mbi'),\qquad \widetilde{\Phi} (\Lb_i'(\mbi)) = \Lb_i'(\mbi').
\]
To prove the proposition, we may hence without loss of generality assume already that $i_1 = i$. 

Now one easily verifies that 
\[
\Fb_i = \Lb_i\star \Xb_1^{-1}
\]
satisfies conditions \eqref{EqCond1},\eqref{EqCond2} and \eqref{EqCond3} of Definition \ref{DefHeisDualOp}. Also condition \eqref{EqCond5} is immediate. Finally, we can write 
\[
\Eb_i = \Xb_1 \boxplus \Yb,\qquad \Yb = \boxplus_{a>1,a\in J_i}  \Xb_a,
\]
and condition \eqref{EqCond4} then holds by Example \ref{ExaCompHeis} applied with 
\[
\Ab  = \Fb_i =  \Xb_1^{-1}\star \Lb_i,\qquad \Bb = \Lb_i^{-1}\star\Yb.
\]
\end{proof}

By means of Proposition \ref{PropExchDual}, Theorem \ref{TheoExchvNbi} will follow from the following theorem. 

\begin{Theorem}\label{TheoDualvNbi}
Let $\mbi$ be a dualizable word. Then there exists a unique von Neumann bialgebra structure on $L_{\hbar}(B(\mbi))$ such that 
\[
\Delta(\Kb_i) = \Kb_i \otimes \Kb_i,\qquad \Delta(\Eb_i) = \Eb_i\otimes 1\boxplus \Kb_i \otimes\Eb_i,\qquad \forall i \in I.
\]
\end{Theorem}

The proof of Theorem \ref{TheoDualvNbi} will take up the remainder of this section. 

Write 
\begin{equation}\label{EqBasisGen1}
\Hb_v = (-\pi i \hbar)^{-1}\ln(\Kb_v),\qquad \Hb_{\omega}' = (-\pi i \hbar)^{-1}\ln(\Lb_{\omega}),\qquad v\in \mfh,\omega\in \mfh',
\end{equation}
and put also
\[
\Hb_i' = \Hb_{\varpi_i^{\vee}} = (-\pi i \hbar)^{-1}\ln(\Lb_i'),\qquad i\in I. 
\]

Consider the rescaled exponential function
\begin{equation}\label{eqGaussFunction}
e_{\hbar}:\R \rightarrow \{z\in \C\mid |z|=1\},\qquad e_{\hbar}(x) = e^{\pi i\hbar x},\qquad \overline{e}_{\hbar}(x) = e^{-\pi i \hbar x},\qquad x\in \R.
\end{equation}
Recall the following property, which we refer to as the \emph{Gaussian commutation law} (and which follows straightforwardly from functional calculus): when, for $i\in \{1,2\}$ and $t_i \in \R$, we have $t_i\hbar/2$-commuting operators $\Ab_i,\Bb_i$, and if we put 
\[
\mathbf{b}_i = (-\pi i \hbar)^{-1}\ln(\Bb_i),
\]
then 
\begin{equation}\label{EqqCommGauss}
e_{\hbar}(\mathbf{b}_1\otimes \mathbf{b}_2)(\Ab_1\otimes \Ab_2)e_{\hbar}(\mathbf{b}_1\otimes \mathbf{b}_2)^* = \Ab_1\star \Bb_1^{t_2} \otimes \Ab_2\star \Bb_2^{t_1}.
\end{equation}

\begin{Def}\label{DefMUCartan}
The \emph{fundamental Cartan unitary} for the Cartan datum $I$ is defined as the unitary 
\begin{equation}\label{EqCartanMU}
\Kc := \prod_{i\in I} \overline{e}_{\hbar}\left(\Hb_i\otimes \Hb_i'\right) \in L(\mfh)\bar{\otimes}L(\mfh').
\end{equation}
\end{Def}

\begin{Def}\label{DefFundUnit}
Assume $\mbi$ is a dualizable word, with associated dual simple root operators $\{\Fb_i\}$.
The \emph{fundamental braided unitary} associated to $(\mbi,\{\Fb_i\})$ is
\begin{equation}\label{EqSimpGenMU}
W_{\br} = \prod^{\longra}_{a\in J} {}^aW_{\br},\qquad {}^aW_{\br} = \overline{F}_{d_{i_a}\hbar}(\Xb_a \otimes \Fb_{i_a}) \in L_{\hbar}(\mfn)\bar{\otimes} L_{\hbar}(\mfb') \subseteq L_{\hbar}(\mfd)\bar{\otimes}L_{\hbar}(\mfd),
\end{equation}
where $\prod^{\longra}_{a\in J}$ means we write out the product from left to right with the index in $J$ increasing.

The \emph{fundamental unitary} associated to $\mbi$ is 
\begin{equation}\label{EqFundamUnit}
W = W_{\br}\Kc \in L_{\hbar}(\mfb)\bar{\otimes}L_{\hbar}(\mfb') \subseteq L_{\hbar}(\mfd)\bar{\otimes}L_{\hbar}(\mfd). 
\end{equation}
\end{Def}

Note that by the Gaussian commutation property, we can also write 
\begin{equation}\label{EqFundamUnitAlt}
W = \Kc \prod^{\longra}_{a\in J} W_{\br}^a,\qquad W_{\br}^a = \overline{F}_{d_{i_a}\hbar}(\Xb_a\star \Kb_{i_a}^{-1} \otimes \Fb_{i_a}\star \Lb_{i_a}^{-1}).
\end{equation}

When the simple root operators $\Fb_i$ are not specified, we refer to $W$ as `a' fundamental unitary. 

Theorem \ref{TheoDualvNbi} is now a direct consequence of the following proposition. 

\begin{Prop}\label{PropIndepHeis}
Let $W$ be a fundamental unitary for the dualizable word $\mbi$. Then the unital normal faithful $*$-homomorphism
\[
\Delta: L_{\hbar}(\mfd) \rightarrow L_{\hbar}(\mfd)\bar{\otimes} L_{\hbar}(\mfd),\qquad x \mapsto W^*(1\otimes x)W
\]
satisfies 
\begin{equation}\label{EqCoproductRestr}
\Delta(L_{\hbar}(\mfb)) \subseteq L_{\hbar}(\mfb)\bar{\otimes}L_{\hbar}(\mfb)
\end{equation}
and
\begin{equation}\label{EqCrodQUE}
\Delta(\Kb_i) = \Kb_i\otimes \Kb_i,\qquad \Delta(\Eb_i) =\Eb_i \otimes 1 \boxplus \Kb_i\otimes \Eb_i.
\end{equation}
\end{Prop} 
\begin{proof}
The inclusion \eqref{EqCoproductRestr} follows from the alternative expression \eqref{EqFundamUnitAlt}, the assumption $\Fb_i \star\Lb_i^{-1}\aff L_{\hbar}(\mfn)$ and the fact that
\[
\Kc^*(1\otimes L_{\hbar}(\mfb))\Kc \subseteq L_{\hbar}(\mfb)\bar{\otimes}L_{\hbar}(\mfb),
\]
which is immediate by the Gaussian commutation relation.

Now by the first commutation relation in Definition \ref{DefHeisDualOp}, we get from \eqref{EqFundamUnit} that
\[
W^*(1\otimes \Kb_i)W = \left(\prod_{j\in I} e_{\hbar}(\Hb_j \otimes \Hb_j')\right)(1\otimes \Kb_i)\left(\prod_{j\in I} \overline{e}_{\hbar}(\Hb_j \otimes \Hb_j')\right),
\]
which by \eqref{EqqCommGauss} indeed reduces to the first identity in \eqref{EqCrodQUE}. Similarly, using those commutation relations and Theorem \ref{TheoHeisPent}, we find 
\[
\Ad_{F_{d_{i_a}\hbar}(\Xb_a\otimes \Fb_{i_a})}(1\otimes \Eb_i) = \left\{\begin{array}{ll} 1 \otimes \Eb_i & \textrm{if }i_a\neq i,\\ \Xb_a \otimes \Lb_{i} \boxplus 1\otimes \Eb_i,&\textrm{if }i_a = i,\end{array}\right.
\]
so combined with the second commutation relation in Definition \ref{DefHeisDualOp}, we find
\[
\Ad_{\prod_{a\in J}^{\longla} F_{d_{i_a}\hbar}(\Xb_a\otimes \Fb_{i_a})}(1\otimes \Eb_i) = \left(\left(\boxplus_{a\in J_i} \Xb_a\right) \otimes \Lb_i\right) \boxplus \left(1\otimes\Eb_i\right) = \Eb_i \otimes \Lb_i \boxplus 1\otimes \Eb_i
\]
and hence, as $\alpha_i = \sum_{j\in I} \langle i,j\rangle \varpi_j^{\vee}$, 
\[
W^*(1\otimes \Eb_i)W = \Ad_{\prod_{j\in I} e_{\hbar}(\Hb_j \otimes \Hb_j')}(\Eb_i \otimes \Lb_i \boxplus 1\otimes \Eb_i) = \Eb_i\otimes 1 \boxplus \Kb_i \otimes \Eb_i.  
\]
\end{proof}

\begin{Rem}
We do not claim that $W$ is a \emph{multiplicative} unitary, i.e.\ that it satisfies the pentagon equation 
\[
W_{12}W_{13}W_{23}= W_{23}W_{12}.
\]
This identity will hold under a small strengthening of Definition \ref{DefHeisDualOp}, which we explore further in Section \ref{SecFullDual}.
\end{Rem}

\section{Quantum Borel von Neumann algebras as locally compact quantum groups}

Fix again a Cartan datum $I$ and $\hbar\in \R^{\times}$. In this section we prove the second part of Theorem \ref{TheoMain}, namely:

\begin{Theorem}\label{TheoIntLCQG}
Let $\mbi$ be a fully exchangeable  and integrable word. Then $(L_{\hbar}(B(\mbi)),\Delta)$ defines a locally compact quantum group.
\end{Theorem}

The proof strategy is similar to the one in \cite{DCSSV25}, and has two main technical components, which we elaborate on in the following two subsections. We then present the proof of the theorem in the third subsection.

\subsection{Modularity}

We first recall the following important concept of \emph{modularity} in the sense of \cite{SW01}*{Definition 2.1} (we slightly change the auxiliary terminology and notation introduced in \cite{DCSSV25}*{Definition 3.6}):

\begin{Def}\label{DefModular}
Let $\Hsp,\Hsp'$ be Hilbert spaces and $W \in \Bc(\Hsp\otimes \Hsp')$ a  unitary. 

We call $W$ \emph{modular} if there exist positive operators $\Qb$ on $\Hsp$ and $\hat{\Qb}$ on $\Hsp'$, and a (necessarily unique) unitary operator $W_c \in \Bc(\overline{\Hsp}\otimes \Hsp')$, with $\overline{\Hsp}$ the complex-conjugate Hilbert space to $\Hsp$, such that 
\begin{equation}\label{EqCommModOp}
W^*(\Qb\otimes \hat{\Qb})W = \Qb\otimes \hat{\Qb}
\end{equation}
and
\begin{equation}\label{EqLeftInvOp}
\langle \Qb\xi\otimes \eta',W(\eta\otimes \xi')\rangle = \langle \overline{\Qb\eta}\otimes \eta',W_c(\overline{\xi} \otimes \xi')\rangle,\qquad \forall \xi,\eta\in \msD(\Qb),\xi',\eta'\in \Hsp'.
\end{equation}
We refer to $\Qb,\hat{\Qb}$ as an associated \emph{modular} and \emph{dual modular} operator.
\end{Def}

The following easy (but crucial) lemma is \cite{DCSSV25}*{Lemma 3.7}.

\begin{Lem}\label{LemMod}
Let $\Hc,\Hc',\Hc''$ be Hilbert spaces, and assume $(V,\Qb,V_c)$ and $(W,\Pb,W_c)$ satisfy \eqref{EqLeftInvOp} with respect to respectively $\Hc,\Hc''$ and $\Hc',\Hc''$. Then $(U,\Qb\otimes \Pb,U_c)$ satisfies \eqref{EqLeftInvOp} with respect to $\Hc\otimes \Hc'$ and $\Hc''$, where 
\begin{equation}\label{EqTensProdMod}
U = V_{13}W_{23} \in \Bc((\Hc\otimes \Hc')\otimes \Hc''),\qquad U_c = V_{c,13}W_{c,23} \in \Bc((\overline{\Hc}\otimes \overline{\Hc'})\otimes \Hc'').
\end{equation}
\end{Lem}

To apply this in our setting, we first observe that we have an alternative expression for the fundamental unitary \eqref{EqFundamUnit}, adapted to the tensor product representation. 

With $\kb,\eb$ as in \eqref{EqGenke}, write
\begin{equation}
\hb = (-\pi i\hbar)^{-1}\ln(\kb) \aff \Bc(L^2(\R)).
\end{equation}
Recall further the Hilbert space $\widetilde{\Hsp}(\mbi)$ from \eqref{EqTensDecompExt}, and the faithful representations 
\begin{equation}\label{EqFaithfRepBor}
L_{\hbar}(\mfb(\mbi))\subseteq \Bc(\widetilde{\Hsp}(\mbi)),\qquad L_{\hbar}(\mfb'(\mbi))\subseteq \Bc(\widetilde{\Hsp}(\mbi)).
\end{equation}

\begin{Lem}
Let $\mbi$ be a dualizable word, and consider its fundamental unitary $W$, defined in \eqref{EqFundamUnit}, as an element of 
\[
W \in L_{\hbar}(\mfb) \bar{\otimes}L_{\hbar}(\mfb') \subseteq \Bc(\widetilde{\Hsp}(\mbi))\bar{\otimes} L_{\hbar}(\mfb') \cong (\Bc(L^2(\R))^{\bar{\otimes} J}\bar{\otimes} \Bc(L^2(\R))^{\bar{\otimes} I})\bar{\otimes} L_{\hbar}(\mfb'),
\]
using the representation \eqref{EqFaithfRepBor}.  Then we have the identity 
\begin{equation}\label{EqWProdOneSide}
W = \left(\prod^{\longra}_{a\in J} {}^aW_{a,\bullet}\right)  \left(\prod_{i\in I} \overline{e}_{\hbar}(\hb_i\otimes \Hb_i')\right),
\end{equation}
where
\begin{equation}\label{EqParsToto}
{}^aW = \overline{F}_{d_{i_a}\hbar}(\eb \otimes \Fb_{i_a})\overline{e}_{\hbar}(\hb \otimes \Hb_{\alpha_{i_a}}') \in \Bc(L^2(\R))\bar{\otimes}L_{\hbar}(\mfb'),
\end{equation}
using the notation \eqref{EqBasisGen1}.
\end{Lem}

Here the leg numbering notation of $W$ in \eqref{EqWProdOneSide} has its first component referring to the position in the $J$-part, resp.\ $I$-part of the tensor product decomposition $\Bc(L^2(\R))^{\bar{\otimes} J}\bar{\otimes} \Bc(L^2(\R))^{\bar{\otimes}I}$, while $\bullet$ refers to the leg living in $L_{\hbar}(\mfb')$. So,
\[
    {}^a W_{a,\bullet} = \overline{F}_{d_{i_a}\hbar}(\eb_a \otimes \Fb_{i_a})\overline{e}_{\hbar}(\hb_a \otimes \Hb_{\alpha_{i_a}}') \in (\Bc(L^2(\R))^{\bar{\otimes J}}\bar{\otimes} \Bc(L^2(\R))^{\bar{\otimes I}}) \bar{\otimes} L_{\hbar}(\mfb').
\]

\begin{proof}
A straightforward application of the Gaussian commutation relation \eqref{EqqCommGauss} and the identity \eqref{EqXVar} gives 
\begin{eqnarray*}
\left(\prod^{\longra}_{a\in J} {}^aW_{a,\bullet}\right)  \left(\prod_{i\in I} \overline{e}_{\hbar}(\hb_i\otimes \Hb_i')\right)&=& W_{\br} \left(\prod_{a\in J} \overline{e}_{\hbar}(\hb_a\otimes \Hb_{\alpha_{i_a}}')\right)\left(\prod_{i\in I} \overline{e}_{\hbar}(\hb_i\otimes \Hb_i')\right)\\
&=& W_{\br}\Kc \\
&=& W.
\end{eqnarray*}
\end{proof}

Recall now the notation for the roots in \eqref{EqRoots}, and define the coroots by 
\begin{equation}
\alpha_i^{\vee} = d_i^{-1} \alpha_i \in \mfh'.
\end{equation}
Extending the notation \eqref{EqAbbrevL}, we also put
\begin{equation}
\Lb_{i^\vee} = \Lb_{\alpha_i^{\vee}},\qquad i\in I.
\end{equation}

Then for the individual components in \eqref{EqWProdOneSide}, we have: 

\begin{Lem}\label{LemModRank1}
The unitaries ${}^aW$ in \eqref{EqParsToto} are modular, with associated modular operators
\begin{equation}\label{EqRank1Q}
\Qb = \kb^{-(d_{i_a}+|\hbar|^{-1})},\qquad \hat{\Qb} =\Lb_{i_a}^{1/2}\star \Lb_{i_a^{\vee}}^{|\hbar|^{-1}/2}.
\end{equation}
\end{Lem} 
\begin{proof}
Using the Gaussian commutation relation \eqref{EqqCommGauss}, we can rewrite 
\begin{eqnarray*}
{}^aW&=& \overline{e}_{\hbar}(\hb \otimes \Hb_{\alpha_{i_a}}') \overline{F_{d_{i_a}\hbar}}(\kb^{-2d_{i_a}}\star \eb \otimes \Lb_{i_a}^{-1}\star \Fb_{i_a}) \\
&=& e^{\frac{i}{2\pi \hbar d_{i_a}}\left((2\pi i d_{i_a}\hbar \hb)\otimes (\pi i \hbar \Hb_{\alpha_{i_a}}')\right)} \overline{F_{d_{i_a}\hbar}}(\kb^{-2d_{i_a}}\star \eb \otimes \Lb_{i_a}^{-1}\star \Fb_{i_a}), 
\end{eqnarray*}
with $e^{2\pi i d_{i_a}\hbar \hb} = \kb^{-2 d_{i_a}}$ and $e^{\pi i \hbar \Hb_{\alpha_{i_a}}'} = \Lb_{i_a}^{-1}$. Noting that $\kb^{-2d_{i_a}}\star \eb$ and $\kb^{-2d_{i_a}}$, resp.\ $\Lb_{i_a}^{-1}\star \Fb_{i_a}$  and $\Lb_{i_a}^{-1}$ are $d_{i_a}\hbar$-commuting, we get from \cite{DCSSV25}*{Theorem 3.8}, upon rescaling $\hbar$ to $d_{i_a}\hbar$ and noting that the convention $q = e^{\pi i \hbar}$ is used there, that 
\[
({}^aW,\Qb,({}^aW)_c)
\]
satisfies \eqref{EqLeftInvOp} with $\Qb$ as in \eqref{EqRank1Q} and with 
\begin{eqnarray*}
({}^aW)_c &=& e_{\hbar}(\overline{\hb}\otimes \Hb_{\alpha_{i_a}})F_{d_{i_a}\hbar}(\overline{\kb^{-2d_{i_a}}\star \eb}\otimes \Fb_{i_a})\\
&=& F_{d_{i_a}\hbar}(\overline{ \eb}\otimes \Lb_{i_a}^{-1}\star\Fb_{i_a})e_{\hbar}(\overline{\hb}\otimes \Hb_{\alpha_{i_a}}).
\end{eqnarray*}
The result now follows upon noting that, taking $\hat{\Qb}$ as in \eqref{EqRank1Q}, the operators $\Qb\otimes \hat{\Qb}$ and ${}^aW$ (strongly) commute.
\end{proof}

We can now conclude the following theorem. Write 
\[
\varpi_i = d_i \varpi_i^{\vee}\in \mfh',\qquad i\in I,
\]
which we refer to as the \emph{fundamental weights}, and write 
\begin{equation}\label{DefModElement}
\rho = \sum_{i\in I} \varpi_i \in \mfh',\qquad \rho^{\vee} = \sum_{i\in I} \varpi_i^{\vee} \in \mfh',
\end{equation}
which we refer to as the \emph{Weyl vector} and \emph{dual Weyl vector}. 

\begin{Theorem}\label{TheoModFundMultUni}
Let $\mbi$ be a dualizable word. Then its fundamental unitary 
\[
W\in \Bc(\widetilde{\Hsp}(\mbi)\otimes \widetilde{\Hsp}(\mbi))
\]
is modular, with modular operators given by 
\begin{equation}\label{EqModDatumvN}
\Qb = \left(\otimes_{a\in J} \kb^{-(d_{i_a}+|\hbar|^{-1})}\right)\otimes \left(\otimes_{i\in I} 1\right),\qquad
\hat{\Qb} =\Lb_{\rho}\star\Lb_{\rho^{\vee}}^{1/|\hbar|}.
\end{equation}
\end{Theorem} 
\begin{proof}
We note that the second factor $\prod_{i\in I} \overline{e}_{\hbar}(\hb_i\otimes \Hb_i')$ in \eqref{EqWProdOneSide} is trivially modular as an operator on $L^2(\R)^{\otimes I} \otimes \widetilde{\Hsp}(\mbi)$, with modular operators equal to $1$ (cf.\ \cite{DCSSV25}*{Lemma C.2}). 

The theorem hence follows immediately from Lemma \ref{LemMod}, Lemma \ref{LemModRank1}, the product decomposition \eqref{EqWProdOneSide} and an easy verification that $\Qb\otimes \hat{\Qb}$ commutes with $W$. 
\end{proof}

\begin{Rem}
If $I$ is non-degenerate, we can identify $\mfh \cong \mfh'$ via \eqref{EqMapToDual}, and the representations of $L_{\hbar}(\mfb)$ and $L_{\hbar}(\mfb')$ on $\Hsp(\mbi)$ are already faithful. In this case, the modular operator for $W \in \Bc(\Hsp(\mbi)\otimes \Hsp(\mbi))$ can be written a bit more conveniently and symmetrically as
\[
\Qb = \Kb_{\rho}^{-1}\star \Kb_{\rho^{\vee}}^{-1/|\hbar|}.
\]
\end{Rem}

\subsection{Integration on quantum tori} 

As mentioned in the introduction, we will use \cite{Tak03} as a general reference for weight theory on von Neumann algebras, see also \cite{DCSSV25}*{Appendix F}. If $M$ is a von Neumann algebra, $M^+$ its positive cone and 
\[
\varphi: M^+ \rightarrow [0,\infty]
\]
a weight, we write
\[
\mfm_{\varphi}^+ = \{x\in M^+ \mid \varphi(x)<\infty\}. 
\]

We have the following general result: 

\begin{Theorem}\label{TheoLeftInvWeight}
Let $V$ be a skew-symmetric space, and assume $V_1,V_2\subseteq V$ are subspaces for which the restricted pairing $\epsilon: V_1\times V_2\rightarrow \R$ is non-degenerate. Let $L_{\hbar}(V) \subseteq \Bc(\Hsp)$ be a faithful normal unital $*$-representation of $L_{\hbar}(V)$, and let 
\[
W \in L_{\hbar}(V_1)\bar{\otimes}L_{\hbar}(V_2)\subseteq \Bc(\Hsp)\bar{\otimes}\Bc(\Hsp)
\]
be a modular unitary.

Then there exists a normal semi-finite weight $\Phi:L_{\hbar}(V)^+ \rightarrow [0,\infty]$ such that 
\begin{enumerate}
\item the normal weight
\begin{equation}\label{EqWeightRest}
\varphi = \Phi_{\mid L_{\hbar}(V_1)^+}: L_{\hbar}(V_1)^+\rightarrow [0,\infty]
\end{equation}
is semi-finite and faithful, and
\item the following left invariance property holds: 
\begin{equation}\label{EqLeftInvGenW}
\Phi((\omega\otimes \id)(\Delta_W(x)) = \varphi(x),\qquad \forall x\in \mfm_{\varphi}^+,\omega \in S_*(L_{\hbar}(V)),
\end{equation}
where, viewing $W \in L_{\hbar}(V_1)\bar{\otimes}L_{\hbar}(V_2)\subseteq L_{\hbar}(V)\bar{\otimes}L_{\hbar}(V)$, we put 
\begin{equation}\label{EqConjW}
\Delta_W: L_{\hbar}(V_1) \rightarrow L_{\hbar}(V)\bar{\otimes}L_{\hbar}(V),\qquad x \mapsto  W^*(1\otimes x)W,\qquad x\in L_{\hbar}(V_1).
\end{equation}
\end{enumerate}
\end{Theorem}

Let us give some details on how the proof is structured, using the results of \cite{DCSSV25}. 

First, we recall (see e.g.\ \cite{DCSSV25}*{Section 5.2}) that if $V$ is any skew-symmetric space, we can endow $L_{\hbar}(V)$ with a canonical nsf trace 
\begin{equation}\label{EqIntV}
\int_{\hat{V}}^{\hbar}: L_{\hbar}(V)^+ \rightarrow [0,\infty].  
\end{equation}
For $v\in V$, we denote by 
\[
\int_{\hat{V}}^{\hbar,\qb(v)}:  L_{\hbar}(V)^+ \rightarrow [0,\infty]
\]
its deformation by $\qb(v)$, so $\int_{\hat{V}}^{\hbar,\qb(v)}$ is again an nsf weight, which, suitably interpreted, can be written as 
\[
\int_{\hat{V}}^{\hbar,\qb(v)} x= \int_{\hat{V}}^{\hbar} \qb(v)^{1/2}x\qb(v)^{1/2},\qquad x\in L_{\hbar}(V)^+
\]
(see again \cite{DCSSV25}*{Section 5.2} for the technical details).

Assume now that we are in the setting of  Theorem \ref{TheoLeftInvWeight}. Then we can embed $V$ into a symplectic vector space $Z$ such that 
\[
L_{\hbar}(Z) \cong \Bc(\Hsp),
\]
with the inclusion $L_{\hbar}(V) \subseteq \Bc(\Hsp)$ arising via the inclusion $V\subseteq Z$. Then, up to normalization, the integral $\int_{\hat{Z}}^{\hbar}$ coincides with the trace $\Tr$ on $\Bc(\Hsp)$. 

Now as $Z$ is symplectic, we have $W^{\perp\perp} = W$ for any subspace $W \subseteq Z$, where
\[
W^{\perp} = \{z\in Z\mid \epsilon(z,w) = 0,\forall w\in W\}. 
\]
As the $\epsilon$-pairing between $V_1$ and $V_2$ is non-degenerate, it follows that we have 
\begin{equation}\label{EqOrthDecomp}
Z = V_1 \oplus V_2^{\perp}.
\end{equation}
By \cite{DCSSV25}*{Corollary 5.9}, there must then exist $a\in L_{\hbar}(V_2^{\perp})$ such that the normal semi-finite weight
\[
\Psi_a: L_{\hbar}(Z)^+ \rightarrow [0,\infty],\quad x \mapsto \int_{\hat{Z}}^{\hbar,\hat{\Qb}^2} axa^*
\]
restricts to a faithful and semi-finite weight
\[
\varphi: L_{\hbar}(V_1)^+\rightarrow [0,\infty],\qquad x \mapsto \Psi_a(x).
\]
In particular, also the restriction $\Phi$ of  $\Psi_a$ to $L_{\hbar}(V)$ will be normal and semi-finite. 

Now since $a \in L_{\hbar}(V_2^{\perp})$, we must have $a \in L_{\hbar}(V_2)' \subseteq \Bc(\Hsp)$. Then viewing \eqref{EqConjW} as a normal faithful $*$-homomorphism
\[
\Delta_W: L_{\hbar}(Z) \rightarrow L_{\hbar}(Z)\bar{\otimes}L_{\hbar}(Z),
\]
we have by \cite{DCSSV25}*{Proposition 5.1} that\footnote{The proof of this proposition does not make use of the condition in that theorem that $W$ should be a multiplicative unitary.}
\begin{equation}
\Psi_a((\omega\otimes \id)\Delta_W(x)) = \Psi_a(x),\qquad \forall x\in \mfm_{\Psi_a}^+,\omega \in S_*(L_{\hbar}(Z)).
\end{equation}
 This immediately implies that also \eqref{EqLeftInvGenW} holds.

\subsection{Proof of Theorem \ref{TheoIntLCQG}}

We can now easily provide a proof of Theorem  \ref{TheoIntLCQG}. We need two small preparations.

\begin{Lem}\label{LemNonDegWord}
Let $\mbi$ be a word over $I$. Inside $\mfd(\mbi)$, the $\epsilon$-pairing between $\mfb(\mbi)$ and $\mfb'(\mbi)$ is non-degenerate. 
\end{Lem}
\begin{proof}
    We prove this by induction on the length $\ell(\mbi)$. If $\ell(\mbi) = 0$ then $\mbi$ is the empty word, $\mfb = \mfh$ and $\mfb' = \mfh'$. The $\epsilon$-pairing is then just the non-degenerate duality pairing between $\mfh$ and $\mfh'$. 

    Now assume that $\mbi'$ is given by concatenating $k \in I$ on the left of the word $\mbi$. We write $J = J(\mbi) = \{1,2,\ldots, N\}$, $J' = J(\mbi') = \{1,\dots,N+1\}$ and $I = \{1,\dots,M\}$, so in particular 
    \[
    (i_1',i_2',\ldots,i_{N+1}') = (k,i_1,i_2,\ldots,i_N).
    \]
    Denote with $\bar{\epsilon}_{\mbi'}$ the matrix of the pairing between the canonical basis vectors of $\mfb(\mbi') = \mfh \oplus \mfn(\mbi')$ and $\mfb'(\mbi') = \mfn(\mbi')\oplus\mfh'$, and similarly for $\bar{\epsilon}_{\mbi}$. The first column in the $\mfn(\mbi')$-part and the $i$-th column in the $\mfh'$-part of $2\bar{\epsilon}_{\mbi'}$ are given by
    \begin{equation*}
       C_1^{\mfn} =  \begin{pmatrix}
            \langle 1, k\rangle \\
            \vdots\\
            \langle M, k\rangle \\
            \hline
             0 \\
            \langle i_2', k\rangle \\
            \vdots\\
            \langle i_{N+1}', k\rangle \\
        \end{pmatrix}
        \quad \text{and}\quad
        C_i^{\mfh}= \begin{pmatrix}
            \langle 1, \varpi_i^{\vee}\rangle \\
            \vdots\\
            \langle M, \varpi_i^{\vee}\rangle \\
            \hline
            \langle i_1', \varpi_i^{\vee}\rangle \\
            \langle i_2', \varpi_i^{\vee}\rangle \\
            \vdots\\
            \langle i_{N+1}', \varpi_i^{\vee}\rangle \\
        \end{pmatrix},
    \end{equation*}
    respectively. Hence $C_1^{\mfn} - \sum_{i=1}^M \langle k,i\rangle C_i^{\mfh}$ gives a column consisting of entirely zeros except at the first index of the $\mfn(\mbi')$-part where we get 
    \[
    -\langle i_1', k\rangle = -\langle k,k\rangle = -2d_k.
    \]
Hence $\bar{\epsilon}_{\mbi'}$ will be non-degenerate if and only if the matrix obtained by removing the first row and column in the $\mfn$-part remains non-degenerate. But this gives exactly the matrix for $\bar{\epsilon}_{\mbi}$, which is non-degenerate by induction.
\end{proof}

We further introduce the following definition: 
\begin{Def}\label{DefReflMap}
Let $\mbi$ be a word over $I$. We define the \emph{reflection map} as the linear map 
\begin{equation}\label{EqReflection}
\varsigma: \mfb(\mbi) \rightarrow \mfb(\mbi),\qquad i \mapsto -i,\qquad a \mapsto a-i_a.
\end{equation}
\end{Def}
The map $\varsigma$ is involutive and satisfies 
\begin{equation}\label{EqInvEpsilon}
\epsilon(\varsigma(v),\varsigma(w))  = -\epsilon(v,w),\qquad v,w\in \mfb.
\end{equation}

\begin{Lem}\label{RefUnitAnti}
There exists a unique unital $*$-anti-isomorphism
\[
R: L_{\hbar}(\mfb) \rightarrow L_{\hbar}(\mfb)
\]
such that 
\begin{equation}\label{EqUnitAntipode}
R(\qb(v)) = \qb(\varsigma(v)),\qquad \forall v\in \mfb,
\end{equation}
where $\varsigma$ is the reflection map. Moreover, we then have $R(L_{\hbar}(B)) = L_{\hbar}(B)$, and more precisely
\begin{equation}\label{EqUnitAntipodeForm}
R(\Kb_i) = \Kb_i^{-1},\qquad R(\Eb_i) = \Eb_i\star\Kb_i^{-1}.
\end{equation}
In particular, 
\begin{equation}\label{EqOppCopr}
\Delta(R(x)) = (R\otimes R) \Delta^{\opp}(x),\qquad x\in L_{\hbar}(B).
\end{equation}
\end{Lem}
\begin{proof}
This is straightforward from \eqref{EqInvEpsilon} and the concrete formulas for the $\Kb_i,\Eb_i$ and their coproduct. 
\end{proof}

\begin{Theorem}\label{TheoInvWeight}
If $\mbi$ is an integrable and fully exchangeable word, then $(L_{\hbar}(B(\mbi)),\Delta)$ defines a locally compact quantum group. 
\end{Theorem} 
\begin{proof}
By integrability 
\[
L_{\hbar}(N(\mbi)) = L_{\hbar}(\mfn(\mbi)),
\]
and hence also 
\[
L_{\hbar}(B(\mbi)) = \R[I]\ltimes L_{\hbar}(N(\mbi)) = \R[I]\ltimes L_{\hbar}(\mfn(\mbi)) = L_{\hbar}(\mfb(\mbi)).
\]
We then get from Theorem \ref{TheoModFundMultUni}, Lemma \ref{LemNonDegWord} and Theorem \ref{TheoLeftInvWeight} that $L_{\hbar}(B(\mbi))$ possesses an nsf weight $\varphi$ satisfying the left invariance condition \eqref{EqLeftInvariance} for all $x\in \mfm_{\varphi}^+$. 

Using the properties of $R$ in $\eqref{RefUnitAnti}$, we see that the nsf weight $\psi = \varphi \circ R$ on $L_{\hbar}(\mfb(\mbi))$ satisfies the right invariance condition \eqref{EqRightInvariance} for all $x\in \mfm_{\psi}^+$.  

This entails that the strong form of invariance in Definition \ref{DefLCQGFormal} holds, see \cite{KV03}*{Proposition 3.1}. 
\end{proof}

\section{Full exchangeability}\label{SecFullExch}

The following theorem uses the cluster algebra intuition of e.g.\ \cite{Ip18a}*{Section 7}), but set up in a way such that all operations are directly analytically implemented.

\begin{Theorem}\label{TheoExchangeable}
Assume $I$ is a classical Cartan datum, and let $\mbi$ be a word in $I$ such that 
\[
w_{\mbi} = w_0.
\] 
Then $\mbi$ is fully exchangeable. 
\end{Theorem} 

We will prove the theorem by (basically) a reduction to the rank $2$ case ($|I|=2$), where the proof will proceed by a case-by-case analysis. 

Before we start, we make the following observation:

\begin{Rem}
In case of a non-degenerate Cartan datum, the definition of equivalence of words $\mbi$ and $\mbi'$ in Definition \ref{DefFullExch} can also be formulated as follows: there exists a unital $*$-isomorphism 
\[
\Phi: L_{\hbar}(\mfb(\mbi)) \rightarrow L_{\hbar}(\mfb(\mbi'))
\]
such that 
\begin{equation}\label{EqConjDiffOrd}
\Phi(L_{\hbar}(\mfn(\mbi))) = L_{\hbar}(\mfn(\mbi')),\qquad \Phi(\Kb_i(\mbi)) = \Kb_i(\mbi'),\qquad 
\Phi(\Eb_i(\mbi)) = \Eb_i(\mbi'),\qquad \forall i\in I.
\end{equation}
This is the version we will use in this section.
\end{Rem}

\subsection{Reduction to the rank \texorpdfstring{$2$}{} case}

We first start with the following observations for a general Cartan datum $I$. 

Let $\mbi,\mbi'$ be words in $I$. By Matsumoto's theorem, we have  $w_{\mbi} = w_{\mbi'}$ inside $W(I)$ if and only if $\mbi$ and $\mbi'$ can be connected through one of the following moves for a subword: 
\begin{enumerate}
\item[(I)] $(i,i) \leftrightarrow i$.
\item[(II)] $\underbrace{(i,j,i,\ldots)}_{m_{ij}} \leftrightarrow \underbrace{(j,i,j,\ldots)}_{m_{ij}}$ for $i\neq j$.
\end{enumerate}
Moreover, if both $\mbi$ and $\mbi'$ are reduced, then only moves of type (II) are needed.

Now if $I$ is of classical type, one can for any $i\in I$ find a reduced word $\mbi$ starting with $i$, and with $w_{\mbi} = w_0$. Hence Theorem \ref{TheoExchangeable} will follow from the following two theorems. 

\begin{Theorem}\label{TheoDoubleLetter}
Let $\mbi = (i_1,i_2,\ldots,i_n)$ be a word with $J(\mbi) = \{1,2,\ldots,n\}$, and let $a<a+1$ be consecutive in $J(\mbi)$ with $i_a = i_{a+1}$. Let $\mbj = (i_1,i_2,\ldots, i_a,i_{a+2},\ldots,i_n,i_n)$. Then $\mbi$ and $\mbj$ are equivalent.  
\end{Theorem}
\begin{proof}
This follows from the following rank $1$ computation which is well-known, but which we spell out for completeness.

It is sufficient to show that if $i\in I$, there exists an $\mfh'$-equivariant isomorphism 
\begin{equation}\label{EqIsoTrivAction}
\Phi: L_{\hbar}(N(i,i)) \cong L_{\hbar}(N(i))\bar{\otimes} L(\R),
\end{equation}
with the trivial action of $\mfh'$ on the $L(\R)$-component, such that 
\[
\Phi(\Eb_j(i,i))  = \Eb_j(i) \otimes 1,\qquad \forall j\in I. 
\]
Now $\Eb_j(i,i) = 0$ and $\Eb_j(i) = 0$ for $j \neq i$, while
\[
\Xb_1(i,i) = \eb\otimes 1,\quad \Xb_2(i,i) = \kb^{2d_i} \otimes \eb,\quad \Eb_i(i,i) = \eb \otimes 1 \boxplus \kb^{2d_i} \otimes \eb,\qquad \Xb_1(i) = \Eb_i(i) = \eb.
\]
Moreover, if we view the second tensorand in \eqref{EqIsoTrivAction} as being generated by $\kb$, then the action of $\mfh'$ is implemented on both sides of \eqref{EqIsoTrivAction} by the conjugation with imaginary powers of the operator 
\[
\kb \otimes \kb. 
\]
Hence using \eqref{EqFormSumId}, we see that
\[
\Phi = \Ad_W,\qquad W = F_{d_i\hbar}(\kb^{2d_i}\star \eb^{-1}\otimes \eb) \in \Bc(L^2(\R)\otimes L^2(\R))
\]
will satisfy \eqref{EqIsoTrivAction}.

\end{proof}

\begin{Theorem}\label{TheoMatsum}
Let $I$ be a Cartan datum, and let $i,j\in I$ be distinct letters with 
\[
\widehat{A}(i,j) = \begin{pmatrix} \langle i,i\rangle & \langle i,j\rangle \\
\langle j,i\rangle & \langle j,j\rangle\end{pmatrix}
\]
classical. Let $\mbi,\mbi'$ be two words over $I$ connected through a type (II)-move for these specific $i,j$, so 
\[
\mbi = (\mbi_1,\mbj,\mbi_2),\qquad \mbi' = (\mbi_1,\mbj',\mbi_2)
\]
with 
\[
\mbj = \underbrace{(i,j,i,\ldots)}_{m_{ij}},\qquad \mbj'= \underbrace{(j,i,j,\ldots)}_{m_{ij}}.
\]
Then $\mbi$ and $\mbi'$ are equivalent.
\end{Theorem} 
\begin{Rem}
Note that by non-degeneracy, the Coxeter matrix $(m_{kl})_{k,l}$ for $\widehat{A}$ will have the Coxeter matrix for $\widehat{A}(i,j)$ as a submatrix. 
\end{Rem} 

We will prove Theorem \ref{TheoMatsum} through a case-by-case analysis. Note that, fixing the setup of the theorem, we can only have $m_{ij} \in \{2,3,4\}$. We will further identify \[
J(\mbi) = J(\mbi') = \{\ldots,-1,0,1,2,\ldots\},
\]
with $1$ the first index of $\mbj$ (resp.\ $\mbj'$). We then also identify 
\[
J(\mbj)  =  J(\mbj') = \{1,2,\ldots,m_{ij}\}.
\]
As we are now identifying $J$ with actual numbers, it will be convenient to write 
\[
\{e_a\mid a\in J\}
\]
for the canonical basis of $\R[J]$, whenever ambiguity could arise. We then use the similar convention for $\mfh = R[I]$. 

The isomorphism $\Phi$ will be constructed in the following way: first create a new basis $\{e_a'\mid a\in J\}$ of $\R[J]$ such that $e_a'= e_a$ for $a\notin J(\mbj)$ and such that $e_a'\in \R[J(\mbj)]$ for $a\in J(\mbj)$, with
\begin{equation}\label{EqSkewProdChange}
\epsilon_{\mbi}(e_a',e_b') = \epsilon_{\mbi'}(e_a,e_b),\qquad \epsilon_{\mbi}(e_i,e_a') = \epsilon_{\mbi'}(e_i,e_a),\qquad \forall a,b\in  J,\forall i\in I.  
\end{equation}
We denote the associated $L_{\hbar}(\mfn(\mbi))$-affiliated operators as 
\[
\Xb_a'(\mbi) = \qb(e_a'). 
\]
This will give us a $*$-isomorphism 
\[
\Psi: L_{\hbar}(\mfb(\mbi)) \rightarrow L_{\hbar}(\mfb(\mbi')),\qquad \Kb_i(\mbi)\mapsto \Kb_i(\mbi'),\quad \Xb_a'(\mbi) \mapsto \Xb_a(\mbi'). 
\]
Then construct a particular unitary $U \in L_{\hbar}(\mfn(\mbj)) \subseteq L_{\hbar}(\mfn(\mbi))$ such that 
\begin{equation}\label{EqCommAllRest}
U\Kb_iU^* = \Kb_i,\qquad U \Xb_a U^*= \Xb_a,\qquad i\in I,a\notin J(\mbi),
\end{equation}
and put 
\[
\Yb_a(\mbi) = U\Xb_a'(\mbi)U^*.
\]
Then we obtain an isomorphism 
\[
\Phi = \Psi \circ \Ad_{U^*}: L_{\hbar}(\mfb(\mbi)) \rightarrow L_{\hbar}(\mfb(\mbi')),\qquad \Kb_i(\mbi)\mapsto \Kb_i(\mbi'),\quad \Yb_a(\mbi) \mapsto \Xb_a(\mbi'), 
\]
which clearly satisfies 
\[
\Phi(L_{\hbar}(\mfn(\mbi))) = L_{\hbar}(\mfn(\mbi')),\qquad \Phi(\Kb_i(\mbi)) = \Kb_i(\mbi'),\qquad \forall i\in I.
\]
To satisfy also the last condition in \eqref{EqConjDiffOrd}, stating that the simple root operators get sent to each other, we will choose our basis $\{a'\}$ and unitary $U$ such that 
\begin{equation}\label{EqNewEsLoc}
\Yb_1(\mbi)\boxplus \Yb_3(\mbi)\boxplus \ldots = \Xb_2(\mbi)\boxplus \Xb_4(\mbi) \boxplus \ldots,\qquad 
\Yb_2(\mbi)\boxplus \Yb_4(\mbi)\boxplus \ldots = 
\Xb_1(\mbi)\boxplus \Xb_3(\mbi) \boxplus \ldots,
\end{equation}
where all sums have $m_{ij}/2$ terms if $m_{ij}$ is even, and where the sums with odd indices have $(m_{ij}+1)/2$ terms and those with even indices have $(m_{ij}-1)/2$ terms if $m_{ij}$ is odd. Due to the other conditions on $U$ and the new basis $\{e_a'\}$, this will entail  
\begin{equation}\label{EqNewEs}
\Eb_i(\mbi) =   \boxplus_{a,i_a = i} \Xb_a(\mbi) = \boxplus_{a,i_a' = i} \Yb_a(\mbi),  
\end{equation}
and so our $\Phi$ will satisfy \eqref{EqConjDiffOrd}. 

Below we show the existence of $\{e_a'\}$ and $U$ for $m_{ij}\in \{2,3,4\}$. Since we have now reformulated the task at hand referring only to operators affiliated to $L_{\hbar}(\mfb(\mbi))$, we can simplify the notation and write
\[
L_{\hbar}(\mfn),L_{\hbar}(\mfb),\Xb_a,\Xb_a',\Yb_a,\ldots
\]

\subsection{The case \texorpdfstring{$m_{ij}=2$}{}}

This means that 
\begin{equation}\label{EqInProdZero}
\langle i,j\rangle = 0.
\end{equation}
In this case, we take our new basis of $\R[J]$ to be the original one, except for 
\[
e_1' = e_2,\qquad e_2'= e_1.
\]
Then  \eqref{EqInProdZero} gives that \eqref{EqSkewProdChange} is satisfied. We have now 
\[
\Xb_1' = \Xb_2,\qquad \Xb_2'= \Xb_1.
\]
Further take $U = 1$, which trivially satisfies \eqref{EqCommAllRest}. Then 
\[
\Yb_1 = \Xb_1' =\Xb_2,\qquad \Yb_2 = \Xb_2' = \Xb_1,
\]
so \eqref{EqNewEsLoc} is satisfied.

\subsection{The case \texorpdfstring{$m_{ij}= 3$}{}}

In this case, we have $a_{ij} = a_{ji} = -1$, and (by rescaling $\hbar$) we may assume $d_i =d_j =1$. 

Consider then for $J$ the new basis which equals the original one, except for 
\[
e_1' = e_2,\qquad e_2' = e_3,\qquad e_3' = e_1+e_2-e_3.
\]
One easily checks that \eqref{EqSkewProdChange} is satisfied. We thus get 
\[
\Xb_1' = \Xb_2,\quad \Xb_2' = \Xb_3,\quad \Xb_3' = \Xb_1\star\Xb_2\star \Xb_3^{-1},\quad \Xb_a' = \Xb_a,\qquad a\in J \setminus \{1,2,3\}.
\]
Now put 
\[
U = F_{\hbar}(\Xb_1\star\Xb_3^{-1}) \in  L_{\hbar}(\mfn).
\]
Then $U$ satisfies \eqref{EqCommAllRest}. Moreover, from \eqref{EqFormSumId} and \eqref{EqFormSumIdAlt} we get that 
\[
\Yb_2 = U\Xb_2'U^* = \Xb_1\boxplus \Xb_3,\qquad \Yb_1\boxplus \Yb_3 = U(\Xb_1'\boxplus \Xb_3')U^* = \Xb_2, 
\]
while $U$ commutes with all $\Xb_a$ for $a\in J \setminus \{1,2,3\}$. Hence \eqref{EqNewEsLoc} holds.

\subsection{The case \texorpdfstring{$m_{ij}=4$}{}}\label{Secmij4}

In this case, we may by symmetry assume that $d_i  = 2d_j$, so  $a_{ij} = -1$ and $a_{ji} = -2$. By rescaling $\hbar$, we may also assume $d_i =2$ and $d_j=1$. 

Consider then for $J$ the new basis which equals the original one, except for 
\[
e_1' = e_4,\qquad e_2' = e_3+2e_2-2e_4,\qquad e_3' = 2e_4-e_2,\qquad e_4' = e_1+2e_2 -2e_4.
\]
One easily checks that \eqref{EqSkewProdChange} is satisfied. We thus get 
\[
\Xb_1' = \Xb_4,\qquad \Xb_2' = \Xb_2^2\star\Xb_3\star\Xb_4^{-2},\qquad \Xb_3' = \Xb_2^{-1}\star\Xb_4^2,\qquad \Xb_4' = \Xb_1\star\Xb_2^2\star\Xb_4^{-2}.
\]
Now put 
\begin{equation}\label{EqDefU}
U =  F_{\hbar}(\Xb_2\star\Xb_4^{-1}) F_{2\hbar}(\Xb_1\star\Xb_3^{-1}) F_{\hbar}(\Xb_4\star\Xb_2^{-1}) \in L_{\hbar}(\mfn),
\end{equation}
and note that we can write $U = U_3U_2U_1$ where 
\begin{equation}\label{EqDefUi}
U_1 = F_{\hbar}(\Xb_2\star\Xb_4^{-1}),\qquad U_2 = U_1F_{2\hbar}(\Xb_1\star\Xb_3^{-1})U_1^*,\qquad U_3 = U_2U_1F_{\hbar}(\Xb_4\star\Xb_2^{-1})U_1^*U_2^*.
\end{equation}
Then clearly $U_1,U_2,U_3$ and hence $U$ satisfy \eqref{EqCommAllRest}. 

It is now sufficient to prove that the
\[
\Yb_a = U\Xb_a'U^*
\]
satisfy 
\begin{equation}\label{Eqm4}
\Yb_1\boxplus \Yb_3 \underset{(1)}{=} \Xb_2\boxplus \Xb_4,\qquad 
\Yb_2\boxplus \Yb_4 \underset{(2)}{=} \Xb_1\boxplus \Xb_3.
\end{equation}

Let us first prove $(1)$ in \eqref{Eqm4}. First note that, by the second identity in \eqref{EqFormSumIdAlt}, 
\begin{equation}\label{EqIdAux1}
U_1\Xb_4U_1^* = \Xb_2\boxplus \Xb_4.
\end{equation}
As $\Xb_1\star\Xb_3^{-1}$ strongly commutes with $\Xb_4$, we then also get 
\begin{equation}\label{EqIdAux2}
U_2U_1\Xb_4U_1^*U_2^* = U_1 \Ad_{F_{2\hbar}(\Xb_1\star\Xb_3^{-1})}(\Xb_4) U_1^* = U_1\Xb_4U_1^* =  \Xb_2\boxplus \Xb_4.
\end{equation}
Using now the first identity in \eqref{EqFormSumIdAlt}, we find 
\begin{equation}\label{EqAux3}
U_3^*(\Xb_2 \boxplus \Xb_4)U_3 \underset{\eqref{EqIdAux2}}{=} U_3^*U_2U_1\Xb_4U_1^*U_2^* = U_2U_1 \Ad_{F_{\hbar}(\Xb_4\star\Xb_2^{-1})}(\Xb_4)U_1^*U_2^* = U_2U_1(\Xb_4\boxplus \Xb_2^{-1}\star\Xb_4^2)U_1U_2.
\end{equation}
Bringing $U_3$ to the other side, we find the first identity in \eqref{Eqm4}.

Let us now prove $(2)$ in \eqref{Eqm4}. Write 
\[
t = e_2-e_4,\quad s = e_1-e_3,\qquad \Tb=  \Xb_2\star \Xb_4^{-1},\quad \Sb = \Xb_1\star\Xb_3^{-1}. 
\]
Noting that $r_0 = e_1+e_2-e_4$ lies in the radical of the restriction of $\epsilon$ to $\R[J(\mbj)]$, we get that $\Cb := \Xb_1\star \Xb_2\star \Xb_4^{-1}$ commutes with all $\Xb_a$ for $a\in \{1,2,3,4\}$, and hence we can write 
\[
\Xb_1 = \Cb\star \Tb^{-1},\qquad \Xb_3 = \Cb\star \Tb^{-1}\star\Sb^{-1},\qquad \Yb_2 = \Cb \star U(\Tb\star\Sb^{-1})U^*,\qquad \Yb_4 = \Cb\star U\Tb U^{-1}.
\]
Hence the second identity in \eqref{Eqm4} is equivalent with 
\begin{equation}\label{Eqm4Alt}
U(\Tb\star \Sb^{-1} \boxplus \Tb)U^* = \Tb^{-1}\star \Sb^{-1} \boxplus \Tb^{-1}. 
\end{equation}
Noting that we can rewrite 
\begin{equation}\label{EqAltDefU}
U = F_{\hbar}(\Tb)F_{2\hbar}(\Sb)F_{\hbar}(\Tb^{-1}),
\end{equation}
it is now sufficient to show that \eqref{Eqm4Alt} holds for any pair of $2\hbar$-commuting positive operators $\Sb,\Tb$, with $U$ as in \eqref{EqAltDefU}. 

Note that the identities \eqref{EqFormSumIdAlt} give
\[
 F_{2\hbar}(\Sb^{-1})\Tb  F_{2\hbar}(\Sb^{-1})^* = \Tb\star\Sb^{-1}\boxplus \Tb,\qquad  F_{2\hbar}(\Sb^{-1})^*\Tb^{-1} F_{2\hbar}(\Sb^{-1}) = \Tb^{-1}\boxplus \Sb^{-1}\star\Tb^{-1}.
\]
We can thus rewrite \eqref{Eqm4Alt} in the equivalent form 
\begin{equation}\label{Eqm4AltAlt}
R\Tb R^* = \Tb^{-1},\qquad R =  F_{2\hbar}(\Sb^{-1})U F_{2\hbar}(\Sb^{-1}) = F_{2\hbar}(\Sb^{-1})F_{\hbar}(\Tb)F_{2\hbar}(\Sb)F_{\hbar}(\Tb^{-1})F_{2\hbar}(\Sb^{-1}).
\end{equation}
Our final task is hence to prove \eqref{Eqm4AltAlt}. 

Recall from \eqref{EqPropsF} that
\[
F_{\hbar}(t^{-1})F_{\hbar}(t)
= J_{\hbar}(t) := e^{-2i\chi_{\hbar}}e^{-i(\ln(t))^2/4\pi\hbar},\qquad t>0.
\]
One easily shows that for any $\hbar$-commuting positive operators $\Ab,\Bb$, one has 
\begin{equation}\label{EqboInv}
J_{\hbar}(\Bb)\Ab J_{\hbar}(\Bb)^* = \Ab\star \Bb 
\end{equation}
(see e.g.\ \cite{DCSSV25}*{Equation (B.1)}). 

For ease of notation, we introduce now the following abbreviations:
\[
\begin{array}{llll}
 A:=F_{2\hbar}(\Sb^{-1}), &
 B:=F_{\hbar}(\Tb), & 
 L:=F_{2\hbar}(\Tb^{2}\star\Sb^{-1}),&
 M:=F_{\hbar}(\Tb\star\Sb^{-1}),\\
 A^{-}:=F_{2\hbar}(\Sb),& 
 B^{-}:=F_{\hbar}(\Tb^{-1}), & 
 L^{-}:=F_{2\hbar}(\Tb^{-2}\star\Sb), &
 M^{-}:=F_{\hbar}(\Tb^{-1}\star\Sb).
\end{array}
\]

Since $(\Sb,\Tb)$ is $2\hbar$-commuting, also the pairs $(\Tb,\Sb^{-1}),(\Tb^{-1},\Sb)$ and $(\Tb \star\Sb^{-1}, \Tb^{-2}\star \Sb)$ are $2\hbar$-commuting. Applying the type $B$ pentagon equation \eqref{EqPentTypeB} to each of them yield the following equations:

\begin{equation}\label{the three PE}
    \begin{array}{lcl}
         AB& = & BLMA  \\
          A^{-} B^{-} & =&  B^{-} L^{-}M^{-}A^{-} \\
          L^{-}M & = & M A B^{-} L^{-}
    \end{array}
\end{equation}

Using the equations in \eqref{the three PE} in turn gives
\begin{equation*}
\begin{array}{lcl}
R  &= & BLMA A_-B_-A \\
 &= & BLM A B_-L_-M_-A_-A \\
 &= & BL\,L^{-}M\,M^{-}A^{-}A \\
&= & B\,(LL^{-})\,(MM^{-})\,(A^{-}A). \\
& = & F_{\hbar}(\Tb) J_{2\hbar}(\Tb^2\star\Sb^{-1}) J_{\hbar}(\Tb\star\Sb^{-1}) J_{2\hbar}(\Sb).
  \end{array}
\end{equation*}
Hence, using repeatedly \eqref{EqboInv} and $J_{\hbar}(t) = J_{\hbar}(t^{-1})$ for $t>0$, we find
\begin{eqnarray*}
R\Tb R^{-1} &=& F_{\hbar}(\Tb) J_{2\hbar}(\Tb^2\star\Sb^{-1}) J_{\hbar}(\Tb\star\Sb^{-1})J_{2\hbar}(\Sb^{-1})\; \Tb \; J_{2\hbar}(\Sb^{-1})^* J_{\hbar}(\Tb\star\Sb^{-1})^*  J_{2\hbar}(\Tb^2\star\Sb^{-1})^*F_{\hbar}(\Tb)^* \\
&=& F_{\hbar}(\Tb) J_{2\hbar}(\Tb^2\star\Sb^{-1}) J_{\hbar}(\Tb\star\Sb^{-1})\; \Tb\star\Sb^{-1} \; J_{\hbar}(\Tb\star\Sb^{-1})^*  J_{2\hbar}(\Tb^2\star\Sb^{-1})^*F_{\hbar}(\Tb)^*\\
&=& F_{\hbar}(\Tb) J_{2\hbar}(\Tb^{-2}\star\Sb) \; \Tb\star\Sb^{-1} \;   J_{2\hbar}(\Tb^{-2}\star\Sb)^*F_{\hbar}(\Tb)^*\\
&=& F_{\hbar}(\Tb)  \; \Tb^{-1}\;   F_{\hbar}(\Tb)^*\\
&=&\Tb^{-1},
\end{eqnarray*} 
proving the identity \eqref{Eqm4AltAlt}.

\section{Integrability}\label{SecIntegrability}

\begin{Theorem}\label{TheoIntegrableGen}
Assume $I$ is a classical Cartan datum, and let $\mbi$ be a reduced word in $I$. Then $\mbi$ is integrable.
\end{Theorem} 

We first show that it is sufficient to treat the case where $w_{\mbi} = w_0$, i.e.\ $\mbi$ is a reduced word of maximal length. This follows directly from the following general lemma, and the fact, true for any finite Coxeter system, that for any reduced word $\mbi$, there exists a reduced word $\mbj$ such that the concatenation $(\mbi,\mbj)$ is a reduced word of maximal length. 

\begin{Lem}
Let $I$ be a Cartan datum, and let $\mbi$ be a word in $I$. Let $j \in I$, and let $\mbj$ be the word $(\mbi,j)$. Then if $\mbj$ is integrable, also $\mbi$ is integrable. 
\end{Lem} 
\begin{proof}
We can write $\mfn(\mbj) = \mfn(\mbi) \oplus \R a$, where $a$ refers to the last position of the word $\mbj$. This decomposition satisfies the hypothesis of Lemma \ref{LemDirSumSkew}, so we can write 
\begin{equation}\label{EqCrossProdIso}
\Phi: L_{\hbar}(\mfn(\mbj)) \overset{\cong}{\rightarrow} \R a \ltimes_{\beta} L_{\hbar}(\mfn(\mbi)),
\end{equation}
where $\R a \overset{\beta}{\curvearrowright} L_{\hbar}(\mfn(\mbi))$ is determined by 
\[
\beta_{ta}(\Xb_b) = e^{2\pi \hbar \langle i_b,j\rangle}\Xb_b,\qquad b\in J(\mbi),t\in \R.
\]
It follows in particular that $\beta$ restricts to an action on $L_{\hbar}(N(\mbi))$, giving us the inclusion 
\begin{equation}\label{EqCrossProdInclu}
\R a \ltimes_{\beta} L_{\hbar}(N(\mbi)) \subseteq \R a \ltimes_{\beta} L_{\hbar}(\mfn(\mbi)). 
\end{equation}
Now \eqref{EqCrossProdIso} satisfies, for $b<a$, 
\[
\Xb_b(\mbj) = \left\{\begin{array}{ll} \Phi(\Xb_b(\mbi)),& i_b \neq j,\\
\Phi(\Xb_b(\mbi))\boxplus \Xb_a(\mbj),& i_b =j,\end{array}\right. \qquad \Eb_i(\mbj) = \left\{\begin{array}{ll} \Phi(\Eb_i(\mbi)),& i \neq j,\\
\Phi(\Eb_i(\mbi))\boxplus \Xb_a(\mbj),& i =j,\end{array}\right. 
\]
while $\Phi(\qb(a)) = \Xb_a(\mbj)$. It follows that  
\[
L_{\hbar}(N(\mbj)) \subseteq \Phi(\R a \ltimes_{\beta} L_{\hbar}(N(\mbi))) \subseteq \Phi(\R a \ltimes_{\beta} L_{\hbar}(\mfn(\mbi))) =  L_{\hbar}(\mfn(\mbj)).
\]
By integrability of $\mbj$, equality must hold everywhere, so pulling back by $\Phi$ we see that equality holds in \eqref{EqCrossProdInclu}. By the general identity \eqref{EqSliceCrossed}, we can conclude that $L_{\hbar}(N(\mbi)) = L_{\hbar}(\mfn(\mbi))$, i.e.\ $\mbi$ is integrable.
\end{proof}

To prove Theorem \ref{TheoIntegrableGen}, we are hence reduced to proving the following theorem.  
 
\begin{Theorem}\label{TheoIntegrable}
Assume $I$ is a classical Cartan datum, and let $\mbi$ be a reduced word in $I$ such that 
\[
w_{\mbi} = w_0,
\]
i.e.\ its associated Coxeter group element is the longest word in the Coxeter group. Then $\mbi$ is integrable.
\end{Theorem}

Note that integrability is clearly preserved under the equivalence relation on words in Definition \ref{DefFullExch}.  Hence by Theorem \ref{TheoMatsum} (and Matsumoto's theorem), it is enough to show Theorem \ref{TheoIntegrable} for one particular word. Moreover, we can clearly restrict to the case of a simple classical Cartan matrix: if $A = \begin{pmatrix} A_1 & 0 \\ 0 & A_2\end{pmatrix}$ splits over sets $I_1,I_2$, we can choose our longest word of the form $\mbi = (\mbi_1,\mbi_2)$ for $\mbi_1,\mbi_2$ longest words with respect to $I_1,I_2$, and one then clearly gets an isomorphism 
\[
L_{\hbar}(\mfn(\mbi)) \cong L_{\hbar}(\mfn(\mbi_1))\bar{\otimes} L_{\hbar}(\mfn(\mbi_2))
\]
sending $L_{\hbar}(N(\mbi))$ to $L_{\hbar}(N(\mbi_1))\bar{\otimes}L_{\hbar}(N(\mbi_2))$.

We then again the theorem via a case-by-case analysis, picking a particular longest word for each simple classical Cartan matrix. Before we start this up, we first gather some general lemmas. 

\subsection{Preparatory lemmas}

 The following is an analogue of \eqref{EqAffOnePar} for the quantum exponential function.
\begin{Lem}\label{LemAff}
Let $\Hsp$ be a Hilbert space and $M \subseteq \Bc(\Hsp)$ a von Neumann algebra. Assume $\Ab$ is a positive operator on $\Hsp$. Then 
\begin{equation}\label{EqQExpInM}
\Ab \aff M \quad \iff \quad F_{\hbar}(s\Ab) \in M,\qquad \forall s>0.
\end{equation} 
\end{Lem} 
\begin{proof}
This follows similarly as in the proof of Theorem \ref{TheoHeisPent}.
\end{proof}

\begin{Lem}\label{LemStrongCont0}
Let $\Hsp$ be a Hilbert space and $\Ab$ is a positive operator on $\Hsp$. Then
\begin{equation}\label{EqLimToZero}
  \lim_{t\rightarrow 0} F_{\hbar}(t\Ab) \rightarrow 0\textrm{ in the
  strong$^*$-topology}.
\end{equation}
\end{Lem} 
\begin{proof}
This follows from $F_{\hbar}$ being a continuous function with values on the unit circle and with $F_{\hbar}(0) = 1$. 
\end{proof}

For readability, we will in this section informally write
\[
q^z = e^{-\pi i \hbar z},\qquad z\in \C,
\]
and we write e.g.\
\[
AB = q^{2a}BA
\]
to mean that $\Ab$ and $\Bb$ are an $a\hbar$-commuting pair of positive operators, where $a\in \R$.

\begin{Lem}[First Splitting Lemma]\label{LemSplit1}
  Let $\Hsp$ be a Hilbert space and $M \subseteq \Bc(\Hsp)$ a von
  Neumann algebra.
  Suppose $\Ab,\Bb,\Cb$ are positive operators on
  $\Hsp$ such that
  \begin{equation*}
AB = q^{2a}BA,\qquad AC = q^{b} CA,\qquad BC= q^{c}CB 
\end{equation*}
  for some $a,b,c \in \R$ with $a\neq0$ and $b\neq c$. Write $\Xb = \Ab \boxplus
  \Bb$. If $\Xb \aff M$ and $\Cb \aff M$ then also $\Ab\aff M$
  and $\Bb \aff M$.
\end{Lem}
\begin{proof}
  Since $\Xb \aff M$ and $\Cb \aff M$ we have
  \begin{equation*}
    \Cb^{it}F_{a\hbar}(s\Xb)\Cb^{-it} \in M, \quad
    \text{for $s > 0$ and $t \in \R$}.
  \end{equation*}
  Hence, making use of \eqref{EqQuantExp},
  \begin{equation}\label{EqPrelimAff}
    \Cb^{it}F_{a\hbar}(s\Xb)\Cb^{-it} =
    \Cb^{it}F_{a\hbar}(s\Ab)F_{a\hbar}(s\Bb)\Cb^{-it} =
    F_{a\hbar}(sq^{-bit}\Ab)F_{a\hbar}(sq^{-cit}\Bb) \in M,
  \end{equation}
  where we note that e.g.\ $q^{-bit} = e^{-b\pi  \hbar t} \in
  (0, \infty)$. Writing $s = r q^{bit}$ for some $r > 0$,
   we then find
  \begin{equation*}
    F_{a\hbar}(r\Ab)F_{a\hbar}(rq^{(b-c)it}\Bb) \in
    M,\quad \text{for all $r > 0$ and $t\in \R$.}
  \end{equation*}
  Taking the limit $t \to -\sgn((b-c)\hbar)\infty$, we find
  by Lemma \ref{LemStrongCont0}   that
  \begin{equation*}
    F_{a\hbar}(r\Ab)\in M \quad \text{for all $r > 0$.}
  \end{equation*}
  Using Lemma \ref{LemAff} we can then conclude that $\Ab
  \aff M$. By \eqref{EqPrelimAff}, we get $F_{a\hbar}(s\Bb) \in M$ for all $s>0$, hence also  $\Bb \aff M$.
\end{proof}
The `trick' in the following lemma is inspired by \cite{Ip18a}*{Appendix}. 
\begin{Lem}[Second Splitting Lemma]\label{LemSplit2}
  Let $\Hsp$ be a Hilbert space and $M \subseteq \Bc(\Hsp)$ a von
  Neumann algebra. Suppose $\Ab,\Bb,\Cb$ and $\Db$ are positive
  operators on $\Hsp$ such that
  \begin{align*}
    AB &= q^{2}BA,\qquad &&AC = q^{-1} CA,\qquad && BC = qCB \\
    CD &= q^{2}DC,\qquad &&AD = q^{-1}DA,\qquad &&BD = q^{-1}DB.
  \end{align*} 
  Write $\Xb = \Ab \boxplus \Bb$
  and $\Yb = \Cb \boxplus \Db$. If $\Xb \aff M$ and $\Yb \aff M$ then
  also $\Cb \star \Bb \aff M$.
\end{Lem}
\begin{proof}
Write 
\[
\Zb = \Cb\star \Bb,
\]
and put
\[
\widetilde{\Hsp} = \Hsp\otimes L^2(\R),\qquad \widetilde{M} = M\bar{\otimes} \Bc(L^2(\R)),\qquad \wbX = \Xb\otimes \kb,\qquad \wbY = \Yb\otimes \eb,\qquad \wbZ = \Zb \otimes \kb\star \eb.
\]
Then we claim that 
 \begin{equation}\label{EqComplPent}
    F_{\hbar}(s\wbY) F_{\hbar}(t\wbX) =
    F_{\hbar}(t\wbX)F_{\hbar}(st\wbZ)F_{\hbar}(s\wbY),
    \text{ for all } s,t \geq 0.
  \end{equation}
Indeed, writing also 
\[
\wbA = \Ab \otimes \kb,\qquad \wbB = \Bb \otimes \kb,\qquad \wbC = \Cb \otimes \eb,\qquad \wbD = \Db \otimes \eb,
\]
note that 
 \begin{align*}
    \wA \wB &= q^{2}\wB\wA,\qquad && \wA \wC =  \wC \wA,\qquad && \wB \wC = q^2\wC \wB \\
    \wC \wD &= q^{2}\wD \wC,\qquad && \wA\wD = \wD\wA,\qquad &&\wB\wD = \wD\wB
  \end{align*} 
  and
\[
\wZ \wA = q^{-1}\wA\wZ,\quad \wZ \wB = q^{-1}\wB\wZ,\quad \wZ\wC= q\wC\wZ,\quad \wZ\wD = q\wD\wZ. 
\]
Then since 
\[
\wbX = \wbA\boxplus \wbB,\qquad \wbY = \wbC\boxplus \wbD,\qquad \wbZ = \wbB\star\wbC,
\]
we can compute:
\begin{eqnarray*}
F_{\hbar}(t\wbX)F_{\hbar}(st\wbZ)F_{\hbar}(s\wbY) 
&\underset{\eqref{EqQuantExp}}{=}&  F_{\hbar}(t\wbA)F_{\hbar}(t\wbB)F_{\hbar}(st\wbZ)F_{\hbar}(s\wbC)F_{\hbar}(s\wbD)\\
&\underset{\eqref{EqPentKash}}{=}& F_{\hbar}(t\wbA)F_{\hbar}(s\wbC)F_{\hbar}(t\wbB)F_{\hbar}(s\wbD)\\
&=& F_{\hbar}(s\wbC)F_{\hbar}(s\wbD)F_{\hbar}(t\wbA)F_{\hbar}(t\wbB)\\
&\underset{\eqref{EqQuantExp}}{=}& F_{\hbar}(s\wbY)  F_{\hbar}(t\wbX),
\end{eqnarray*} 
proving the claim.

Now since $\wbX,\wbY \aff \widetilde{M}$, we obtain from \eqref{EqComplPent} that $F_{\hbar}(s\wbZ) \in \widetilde{M}$ for all $s>0$, and hence from Lemma \eqref{LemAff} that
\[
\Zb \otimes \kb\star\eb = \wbZ \aff \widetilde{M} = M\bar{\otimes} \Bc(L^2(\R)).
\]
But this necessarily implies that $\Zb = \Cb\star\Bb \aff M$, finishing the proof. 
\end{proof}
\begin{Lem}[Third Splitting Lemma]\label{LemSplit3}
  Let $\Hsp$ be a Hilbert space and $M \subseteq \Bc(\Hsp)$ a von
  Neumann algebra. Suppose $\Ab,\Bb,\Cb$ and $\Db$ are positive
  operators on $\Hsp$ such that
  \begin{align*}
    AB &= qBA,\qquad && AC = q^{-1} CA,\qquad && BC = q CB\\
    CD &= q^{2}DC,\qquad && AD = q^{-1} DA,\qquad && BD = q^{-1}DB.
  \end{align*} 
  Write $\Xb = \Ab \boxplus \Bb$
  and $\Yb = \Cb \boxplus \Db$. If $\Xb \aff M$ and $\Yb \aff M$ then
  $\Ab,\Bb,\Cb$ and $\Db$ are also affiliated to $M$.
\end{Lem}
\begin{proof}
  Write 
  \begin{equation*}
    u = F_{\hbar/2}(\Ab\star\Bb^{-1}). 
  \end{equation*}
  By \eqref{EqFormSumId} we get
  \begin{equation*}
    u \Xb u^* = F_{\hbar/2}(\Ab \star \Bb^{-1})(\Ab \boxplus
    \Bb)F_{\hbar/2}(\Ab \star \Bb^{-1})^* = \Bb.
  \end{equation*}
  Thus, $\Bb \aff u M u^*$. Now, since $\Ab \star \Bb^{-1}$ and $\Db$ strongly 
  commute, we obtain
  \begin{equation*}
    u \Cb u^* \boxplus \Db = u \Yb u^* \aff u M u^*.
  \end{equation*}
  Then, making use of \eqref{EqQuantExp}, we obtain that for all $s>0$
  \begin{equation*}
    F_{\hbar}(s(u \Cb u^* \boxplus \Db)) = F_{\hbar/2}(\Ab \star
    \Bb^{-1}) F_{ \hbar}(s\Cb)F_{\hbar/2}(\Ab
    \star\Bb^{-1})^*F_{\hbar}(s\Db) \in u M u^*.
  \end{equation*}
  Using that $\Bb \aff u M u^*$ we then find for $t \in \R$ that
  \begin{equation*}
    \Bb^{it} F_{\hbar}(s(u \Cb u^* \boxplus \Db))\Bb^{-it}
    = F_{\hbar/2}(q^{-it}\Ab \star
    \Bb^{-1}) F_{\hbar}(sq^{it}\Cb)F_{\hbar/2}(q^{-it}\Ab
    \star\Bb^{-1})^*F_{\hbar}(sq^{-it}\Db) \in u M u^*.
  \end{equation*}
  Taking the limit $t \to \sgn(\hbar)\infty$ after setting $s = r
  q^{-it}$ and using \eqref{EqLimToZero} we find
  \begin{equation*}
    F_{\hbar}(r\Cb) \in u M u^*\quad \text{for all $r > 0$}.
  \end{equation*}
  Lemma \ref{LemAff} now gives $\Cb \aff u M u^*$. This means that also
  \begin{equation*}
    (\Bb^2 \star \Cb)^{it} F_{\hbar}(s(u \Cb u^* \boxplus
    \Db))(\Bb^2 \star \Cb)^{-it} \in u M u^*.
  \end{equation*}
  Thus
  \begin{equation*}
    F_{\hbar/2}(\Ab \star
    \Bb^{-1}) F_{ \hbar}(sq^{2it}\Cb)F_{\hbar/2}(\Ab
    \star\Bb^{-1})^*F_{\hbar}(s\Db) \in u M u^*.
  \end{equation*}
  Taking the limit $t \to - \sgn(\hbar)\infty$ and using once again
  \eqref{EqLimToZero} and Lemma \ref{LemAff}, we find that $\Db \aff u M u^*$.
  But $\Db$ and $u$ commute, so $\Db \aff M$.

  The rest now follows relatively easily. For $s > 0$ we have, using \eqref{EqQuantExp}, that 
  \begin{equation*}
    F_{ \hbar}(s \Yb) = F_{\hbar}(s \Cb)F_{\hbar}(s \Db),
  \end{equation*}
   so 
  \begin{equation*}
    F_{ \hbar}(s \Cb) = F_{\hbar}(s \Db)^*F_{\hbar}(s \Yb) \in M,
  \end{equation*}
  and hence $\Cb \aff M$ by Lemma  \ref{LemAff}. Finally, it now follows
  from the first splitting lemma, Lemma \ref{LemSplit1}, that $\Ab \aff
  M$ and $\Bb \aff M$.
\end{proof}

With these lemmas in hand, we can now prove Theorem \ref{TheoIntegrable} via a case-by-case analysis. We will from now on write simply 
\[
M = L_{\hbar}(N(\mbi))
\]
for the word $\mbi$ under consideration. 

\subsection{Cartan matrix of type \texorpdfstring{$A_n$}{}}

Our Cartan datum in this case is
\[
  A = 
  \begin{pmatrix}
    2 & -1 & 0 & \cdots & 0 \\
    -1 & 2 & -1 & \ddots & \vdots \\
    0 & -1 & \ddots & \ddots & 0\\
    \vdots & \ddots & \ddots  & 2& -1\\
    0 & \cdots & 0 & -1 & 2\\
  \end{pmatrix}
  ,\qquad  D = \diag(1,1,\dots,1),\qquad \widehat{A} = A.
\]

We choose here the longest word 
\[
w_0 = (1;2,1;3,2,1;\ldots;n,n-1,\ldots,2,1). 
\]
Then the generating property was proven in \cite{DCSSV25}*{Proposition 4.18} (it is basically a direct application of Lemma \ref{LemSplit1}).  

\subsection{Cartan matrix of type \texorpdfstring{$B_n$}{}}

Take $n \geq 2$. Then we have as Cartan datum for type $B_n$
\[
  A =
  \begin{pmatrix}
    2 & -2 & 0 & \cdots & 0 \\
    -1 & 2 & -1 & \ddots & \vdots \\
    0 & -1 & \ddots & \ddots & 0\\
    \vdots & \ddots & \ddots  & 2& -1\\
    0 & \cdots & 0 & -1 & 2\\
  \end{pmatrix}
  ,\qquad  D = \diag(1,2,\dots,2),\qquad \widehat{A} = DA =
  \begin{pmatrix}
    2 & -2 & 0 & \cdots & 0 \\
    -2 & 4 & -2 & \ddots & \vdots \\
    0 & -2 & \ddots & \ddots & 0\\
    \vdots & \ddots & \ddots  & 4& -2\\
    0 & \cdots & 0 & -2 & 4\\
  \end{pmatrix}.
\]
As longest word we take
\[
  \underline{w}_0 = \mbi_{B_n} =
  (1;\;2,1,2;\;3,2,1,2,3;\;\dots;\;n,\dots,1,\dots,n).
\]
It will be beneficial to view $\mbi_{B_n}$ as
$\mbi_{B_{n-1}}\mbb_{n}$ with
\begin{equation*}
  \mbb_n = (n, n-1, \dots,2, 1,2 \dots, n-1,n).
\end{equation*}

Consider first the base case, type $B_2$. We have
\begin{align*}
  \Eb_1 &= \Xb_1 \boxplus \Xb_3 \\
  \Eb_2 &= \Xb_2 \boxplus \Xb_4.
\end{align*}
Upon replacing $\hbar$ by $-2\hbar$ and taking $\Ab = \Xb_1, \Bb = \Xb_3, \Cb = \Xb_2, \Db =
\Xb_4$, we see that this case is an immediate consequence of the third
splitting lemma, Lemma \ref{LemSplit3}.

Now consider the general $B_n$-case, with $n\geq 3$. Let $\Gb_j$ denote the form sum of all generators labeled by $j$ in
$\mathbf{i}_{B_{n-1}}$. The simple root vectors can then be written as
\begin{align*}
  \Eb_1 &= \Gb_1 \boxplus \Xb_n \\
  \Eb_2 &= \Gb_2 \boxplus \Xb_{n-1} \boxplus \Xb_{n+1} \\
  &\vdots \\
  \Eb_{n-1} &= \Gb_{n-1} \boxplus \Xb_{2} \boxplus \Xb_{2n-2}\\
  \Eb_{n} &= \Xb_{1} \boxplus \Xb_{2n-1}.
\end{align*}

We can now apply the second splitting lemma, Lemma \ref{LemSplit2}, with
\begin{align*}
  \Ab &= \Gb_{n-1} \qquad &\Bb &= \Xb_{2} \boxplus \Xb_{2n-2}\\
  \Cb &= \Xb_{1}  \qquad &\Db &= \Xb_{2n-1},
\end{align*}
and replace again $\hbar$ by $-2\hbar$ in that lemma. We then find that $(\Xb_{2}\boxplus \Xb_{2n-2})\star
\Xb_{1} \aff M$. Now, since 
\begin{equation*}
  (\Xb_{2}\boxplus \Xb_{2n-2})\star\Xb_{1} = \Xb_{2}\star
  \Xb_{1}\boxplus \Xb_{2n-2}\star \Xb_{1},
\end{equation*}
and $\Xb_{1}$ strongly commutes with every term of $\Eb_{n-2}$,
we can use the second splitting lemma again, this time with $\hbar$ replaced by $-2\hbar$ and
\begin{align*}
  \Ab &= \Gb_{n-2} \qquad  &\Bb &= \Xb_{3} \boxplus \Xb_{2n-3}\\
  \Cb &= \Xb_{2}\star \Xb_{1}  \qquad &\Db &= \Xb_{2n-2}\star \Xb_{1}.
\end{align*}
Note that the $q$-commutations still work out so that $CD=
q^{-4}D C$. This time the lemma gives
\begin{equation*}
  (\Xb_{3}\boxplus\Xb_{2n-3})\star \Xb_{2}\star \Xb_{1} \aff M.
\end{equation*}

Proceeding like this inductively we obtain
\begin{equation*}
  (\Xb_{n-1} \boxplus \Xb_{n+1}) \star \Xb_{n-2} \star \dotsb \star
  \Xb_1 \aff M.
\end{equation*}

Now we arrive at the location of the short root, so we can use the
third splitting lemma, Lemma \ref{LemSplit3}, again with $\hbar$ replaced by $-2\hbar$ and with 
\begin{align*}
  \Ab &= \Gb_1 \qquad &\Bb &= \Xb_{n}\\
  \Cb &= \Xb_{n-1}\star \Xb_{n-2} \star \dotsb \star \Xb_{1}  \qquad
  &\Db &= \Xb_{n+1}\star\Xb_{n-2}\star \dotsb \star \Xb_{1},
\end{align*}
to conclude that each of these operators is affiliated to $M$.

We can now move backwards using the first splitting lemma, Lemma \ref{LemSplit1}. Taking
\begin{equation*}
  \Ab = \Gb_2 \boxplus \Xb_{n-1}, \qquad \Bb = \Xb_{n+1},
  \qquad \Cb = \Xb_{n},
\end{equation*}
we find that $\Gb_{2} \boxplus \Xb_{n-1} \aff M$ and $\Xb_{n+1}\aff
M$. Moving down like this we can split off the last term in each
$\Eb_i$. At the end we obtain $\Xb_1 \aff M$ and $\Xb_{2n-1} \aff M$.

Now we can move back up: Taking
\begin{equation*}
  \Ab = \Gb_{n-1}, \qquad \Bb = \Xb_{2},
  \qquad \Cb = \Xb_{1},
\end{equation*}
we find by the first splitting lemma that $\Gb_{n-1} \aff M$ and $\Xb_{2}\aff M$. By repeated use of this technique, all the $\Gb_j$ can be split off from the remain $\Xb_k$. We hence obtain that 
\[
\Gb_a \eta M,\qquad \Xb_b \eta M,\qquad \forall a\in J(\mbi_{B_{n-1}}),b\in J(\mbb_n). 
\]
Since the $\Gb_a$ simply are the simple root operators one obtains in the $B_{n-1}$-case with respect to the word $\mbi_{B_{n-1}}$, we can conclude by induction. 

\subsection{Cartan matrix of type \texorpdfstring{$C_n$}{}}\label{SecCartTypeC}

For $n \geq 2$ we have Cartan data
\[
  A =
  \begin{pmatrix}
    2 & -1 & 0 & \cdots & 0 \\
    -2 & 2 & -1 & \ddots & \vdots \\
    0 & -1 & \ddots & \ddots & 0\\
    \vdots & \ddots & \ddots  & 2& -1\\
    0 & \cdots & 0 & -1 & 2\\
  \end{pmatrix}
  ,\qquad  D = \diag(2,1,\dots,1),\qquad \widehat{A} = DA =
  \begin{pmatrix}
    4 & -2 & 0 & \cdots & 0 \\
    -2 & 2 & -1 & \ddots & \vdots \\
    0 & -1 & \ddots & \ddots & 0\\
    \vdots & \ddots & \ddots  & 2& -1\\
    0 & \cdots & 0 & -1 & 2\\
  \end{pmatrix}.
\]
We take the same longest word as in type $B$:
\[
  \underline{w}_0 = \mbi = (1;\;2,1,2;\;3,2,1,2,3;\;\dots;\;n,\dots,1,\dots,n).
\]

The remainder is now completely analogous to type $B$. The only real
difference is that we now apply Lemma
\ref{LemSplit3} upon swapping $\hbar \leftrightarrow -\hbar$ and 
\[
\Ab \leftrightarrow \Db,\qquad \Bb\leftrightarrow \Cb.
\]

\subsection{Cartan matrix of type \texorpdfstring{$D_n$}{}}

We take the following labeling for $D_n$:
\begin{center}
  \begin{tikzpicture}[x=1.15cm,y=1.0cm,baseline=(current bounding box.center)]
    \node[dynkin node] (d1) at (0,.62) {$1$};
    \node[dynkin node] (d2) at (0,-.62) {$2$};
    \node[dynkin node] (d3) at (1.15,0) {$3$};
    \node[dynkin node] (d4) at (2.30,0) {$4$};
    \node (dots) at (3.43,0) {$\cdots$};
    \node[dynkin node] (dm) at (4.55,0) {$n$};
    \draw[dynkin edge] (d1)--(d3) (d2)--(d3) (d3)--(d4)
    (d4)--(dots) (dots)--(dm);
  \end{tikzpicture}
\end{center}

Since $D_n$ is simply-laced we can take $\widehat{A}=A$
where 
\begin{equation}\label{eq:Dn-Cartan-matrix}
  A=
  \begin{pmatrix}
    2&0&-1&0&0&\cdots&0\\
    0&2&-1&0&0&\cdots&0\\
    -1&-1&2&-1&0&\cdots&0\\
    0&0&-1&2&-1&\ddots&\vdots\\
    0&0&0&-1&2&\ddots&0\\
    \vdots&\vdots&\vdots&\ddots&\ddots&\ddots&-1\\
    0&0&0&\cdots&0&-1&2
  \end{pmatrix}.
\end{equation}

For every $n\geq3$, we will now recursively define a reduced
expression $\mathbf{i}_{D_n}$ for the longest element $w_0(D_n)$. Namely
let $ \mathbf{i}_{D_3}=(2,3\,;\,1,2,3,1)$ and
\begin{equation}\label{eq:new-block type D}
  \mathbf{d}_k
  = (k,k-1,\ldots,3,1,2,3,\ldots,k-1,k),
\end{equation}
for $k \geq 3$. Then, for $n \geq 4$, we define
\begin{equation}\label{eq:recursive-word type D}
  \mathbf i_{D_n}=\mathbf{i}_{D_{n-1}}\,\mathbf{d}_n.
\end{equation}

As before, let $\Gb_j$ denote the form sum of all generators labeled by $j$ in
$\mathbf{i}_{D_{n-1}}$. The simple root operators for $\mbi_{D_n}$ are then
\begin{align*}
  \Eb_1&=\Gb_1\boxplus \Xb_n,\\
  \Eb_2&=\Gb_2\boxplus \Xb_{n+1},\\
  \Eb_j&=\Gb_j\boxplus \Xb_{n-j+1}\boxplus \Xb_{n+j-2},
  \qquad 3\leq j\leq n-1,\\
  \Eb_n&=\Xb_1\boxplus \Xb_{2n-2}.
\end{align*}

As before we proceed by induction on $n$. The first step is exactly
the same as in type $B$ or $C$. Repeated use of Lemma~\ref{LemSplit2} gives that
\begin{equation*}
  \Yb_1 := \Xb_n \star \Xb_{n-2} \star \dotsb \star \Xb_1 \aff M \qquad
  \Yb_2 := \Xb_{n+1}
  \star \Xb_{n-2} \star \dotsb \star \Xb_1 \aff M.
\end{equation*}
Then it follows that also 
\[
\Yb = \Xb_n \star \Xb_{n+1}^{-1} = \Yb_1
\star \Yb_2^{-1} \aff M.
\]

Now since 
\[
Y G_1 = q^{2} G_1 Y \qquad \textrm{and}\qquad 
Y X_n = X_n Y,
\]
we can use the first splitting lemma to show that $\Xb_n \aff M$.
Similarly, $\Xb_{n+1} \aff M$. From this point on the rest of the
proof proceeds as in type $B$.

\section{Modular structure of split real Borel quantum groups}\label{Sec8}

In this section, we fix a Cartan datum $I$ and a fully exchangeable integrable word 
\[
\mbi = (i_1,i_2,\ldots,i_N)
\]
over $I$ (so in particular we take $J = \{1,2,\ldots,N\}$), with associated word 
\[
w = w_{\mbi} \in W(I). 
\]
Recall that we can let $W(I)$ act on $\mfh = \R[I]$ by
\[
s_i \mapsto r_i \in \End(\R[I]),
\]
using the elementary reflections $r_i$ from \eqref{EqElRefl}. We then also write $r_i$ for the dual reflections 
\[
r_i: \mfh'\rightarrow \mfh',\quad \omega \mapsto \omega - \langle \omega,i^{\vee}\rangle \alpha_i,
\]
using the coroot notation 
\[
i^{\vee} = d_i^{-1}i\in \mfh,\qquad i \in I.
\]

We let $(L_{\hbar}(B),\Delta)$ be the locally compact quantum group associated to $\mbi$ (Theorem \ref{TheoInvWeight}). 

In this section we determine the associated left and right invariant nsf weights of $(L_{\hbar}(B),\Delta)$ more explicitly, and give further information on its modular structure. 

\subsection{Left invariant integral}

Recall from Definition \ref{DefRedHeisSpace} the Heisenberg space $\mfd$, with its associated skew-symmetric form $\epsilon$ determined by \eqref{EqPairNilp}, \eqref{EqPairCart} and \eqref{EqPairHeis}. Recall from Lemma \ref{LemNonDegWord} that $\mfb,\mfb'\subseteq \mfd$ are non-degenerately paired within it. We can hence consider the general orthogonal decomposition \eqref{EqOrthDecomp} in this case (which is also valid for non-symplectic $Z$), giving
\begin{equation}\label{EqDecompHeis}
\mfd = \mfb \oplus (\mfb')^{\perp},\qquad (\mfb')^{\perp} = \{x\in \mfd\mid \epsilon(x,\mfb') =0\}.  
\end{equation}
Recall further the Weyl vector $\rho \in \mfh' \subseteq \mfb'$ from \eqref{DefModElement}. 
\begin{Def}
If $\mbi$ is a word, we write
\begin{equation}
\pi_{\mfb}: \mfd \twoheadrightarrow \mfb,\qquad x \mapsto x,\quad y \mapsto 0,\qquad x\in \mfb,y\in (\mfb')^{\perp}
\end{equation}
for the projection onto the first component of \eqref{EqDecompHeis}.

We then define the \emph{left modular vector} $\theta_l$ and \emph{dual left modular vector} $\theta_l^{\vee}$ by  
\begin{equation}
\theta_l = \pi_{\mfb}(\rho) \in \mfb,\qquad \theta_l^{\vee} = \pi_{\mfb}(\rho^{\vee}) \in \mfb.
\end{equation}
\end{Def}

Define now by
\begin{equation}
\beta_a^{\vee} = r_{i_1}\circ \dotsb \circ r_{i_{a-1}} (i_a^{\vee}) \in \mfh,\qquad a\in J
\end{equation}
the associated \emph{positive coroots} for $\mbi$. 

\begin{Lem}
Let $\lambda \in\mfh'$. Then
\begin{equation}\label{EqExplicitFormulaProjection}
    \pi_\mfb (\lambda) =  -\frac{1}{2}H_{\lambda} + \frac{1}{2} N_{\lambda} \in \mfb,
\end{equation}
where 
\begin{equation}\label{EqComppi}
H_{\lambda} = \sum_{a \in J}   
    \langle \lambda, \beta_a^{\vee}\rangle i_a,\qquad N_{\lambda} = \sum_{a \in J}\langle \lambda, \beta_a^{\vee}\rangle a.
\end{equation}
Furthermore, recalling the notation \eqref{EqMapToDual}, we have
\begin{equation}\label{EqLambdaAsSumRoots}
\alpha_{H_{\lambda}} = \lambda - w^{-1} \lambda.
\end{equation}
\end{Lem}
In particular, we get that 
\begin{equation}\label{EqFormLeftMod}
\theta_l = -\frac{1}{2}H_{\rho} + \frac{1}{2} N_{\rho},\qquad \theta_l^{\vee} = -\frac{1}{2}H_{\rho^{\vee}} + \frac{1}{2} N_{\rho^{\vee}}.
\end{equation}
If $I$ is of Lie type and $w = w_0$, we can by \eqref{EqLambdaAsSumRoots} rewrite the first term in these expressions via 
\begin{equation}\label{EqHrhoLongestWord}
H_{\rho} = 2\rho_{\mfh},\qquad H_{\rho^{\vee}} = 2\rho_{\mfh}^{\vee},
\end{equation}
where $\rho_{\mfh},\rho_{\mfh}^{\vee}$ are the copies of respectively $\rho,\rho^{\vee}$ under the isomorphism $\mfh \overset{\alpha}{\cong}\mfh'$.
\begin{proof}
	Write $z = \pi_\mfb (\lambda)$. 
    Since $z \in \mfb = \mfh \oplus \mfn$, we have
    \[
    z = H + \sum_{a \in J} t_a a
    \]
    for some $H \in \mfh$ and $t_a \in \R$. 
    
    We first show that 
    \begin{equation}\label{EqEquatH}
    H = - \sum_{a \in J} t_a i_a.
    \end{equation}
    As $z = \pi_\mfb (\lambda)$ we have $z - \lambda\in (\mfb')^\perp \subseteq \mfd$. Now for $\omega \in \mfh'$ we have
    \[
    0 = \epsilon(z - \lambda, \omega) = \epsilon(z, \omega) = \frac{1}{2}(\langle H, \omega \rangle + \sum_{a \in J} t_a \langle i_a, \omega \rangle),
    \]
    since $\lambda \in \mfh'$ and thus $\epsilon(\lambda, \omega)=0$. It follows that
    \[
    \langle H, \omega \rangle = \langle -\sum_{a \in J} t_a i_a, \omega \rangle,\qquad \forall \omega \in \mfh',
    \]
    hence \eqref{EqEquatH} holds.

    We now show that the $t_b, b \in J$ satisfy a recursive relationship. Write $\lambda = z + x$ with $x \in (\mfb')^\perp$. Then, for $b \in J$,  we have
    \begin{equation}
        \epsilon(z, b) = \epsilon(z + x, b) = \epsilon(\lambda,b) = - \frac{1}{2}\langle \lambda, i_b \rangle.
    \end{equation}
    On the other hand we have
    \begin{align*}
        \epsilon(z , b)
        &= \epsilon(H, b) + \sum_{a \in J} t_a \epsilon(a, b) \\
        &= \frac{1}{2}\left(\langle H, i_b \rangle + \sum_{a > b} t_a \langle i_a, i_b \rangle - \sum_{a < b} t_a \langle i_a, i_b \rangle \right)\\
        &= \frac{1}{2}\left(-\sum_{a \in J} t_a\langle i_a, i_b \rangle + \sum_{a > b} t_a \langle i_a, i_b \rangle - \sum_{a < b} t_a \langle i_a, i_b \rangle \right)\\
        &= - \frac{1}{2}t_b\langle i_b , i_b \rangle- \sum_{a < b} t_a \langle i_a, i_b \rangle.
    \end{align*}
    Combining the two gives 
    \[
    \langle \lambda, i_b \rangle =  t_b \langle i_b, i_b \rangle+  \sum_{a < b} 2 t_a \langle i_a, i_b\rangle.
    \]
    Dividing both sides by $d_{i_b} = \langle i_b, i_b \rangle / 2$ we obtain
    \begin{equation*}
       \langle \lambda, i_b^{\vee}\rangle -2t_b= \sum_{a < b} 2 t_a \langle i_a, i_b^{\vee} \rangle,
    \end{equation*}
    and hence
\begin{equation}\label{EqRecurrenceTb}
    t_b = \langle \frac{1}{2} \lambda - \sum_{a < b} t_a \alpha_{i_a}, i_b^{\vee}\rangle.
    \end{equation}
    Solving this recurrence will finish the proof of the lemma.

    Write $v_0 = \frac{1}{2}\lambda$ and $v_b = \frac{1}{2}\lambda - \sum_{a \leq b} t_a \alpha_{i_a}$ for $b \in J$. Using \eqref{EqRecurrenceTb} we then have
    \[
    v_b = v_{b-1} - t_b \alpha_{i_b} = v_{b-1} - \langle v_{b-1}, i_b^{\vee}\rangle \alpha_{i_b} = r_{i_b} (v_{b-1}).
    \]
    Thus
    \[
    t_b = \langle v_{b-1}, i_b^{\vee}\rangle = \frac{1}{2}\langle r_{i_{b-1}} \circ \dotsb \circ r_{i_1}(\lambda), i_b^{\vee}\rangle = \frac{1}{2}\langle \lambda, r_{i_{1}} \circ \dotsb \circ r_{i_{b-1}}(i_b^{\vee})\rangle = \frac{1}{2}\langle \lambda, \beta_b^{\vee}\rangle.
    \]
    Finally, $v_N = \frac{1}{2}\lambda - \sum_{a \in J} t_a \alpha_{i_a} = \frac{1}{2}\lambda + \alpha_{H}$, but also $v_N = \frac{1}{2}w^{-1}\lambda$, thus 
    \[
    \alpha_{H} = - \frac{1}{2}(\lambda - w^{-1}\lambda).
    \]
\end{proof}

Using now \cite{DCSSV25}*{Lemma 5.8}, we find that the weight \eqref{EqWeightRest}, which appears in the general procedure for constructing the left invariant nsf weight in Theorem \ref{TheoInvWeight}, leads to the following explicit formula\footnote{The notation $\Eb(v)$ in \cite{DCSSV25} coincides with $\qb(v)$ in the current paper, for $v$ in a skew-symmetric vector space.}:
\begin{Theorem}
A left invariant nsf weight for $(L_{\hbar}(B),\Delta)$ is given by 
\begin{equation}\label{EqExplicitFormLeftInv}
\varphi = \int_{\widehat{\mfb}}^{\hbar,\qb(2(\theta_l+|\hbar|^{-1}\theta_l^{\vee}))}.
\end{equation}
\end{Theorem}

In particular, the associated modular group is given by 
\[
\sigma_t^{\varphi} = \Ad_{\qb(2(\theta_l+|\hbar|^{-1}\theta_l^{\vee}))^{it}},
\]
so since $\Kb_i = \qb(i)$ and $\Xb_a = \qb(a)$, we find
\[
\sigma_t^{\varphi}(\Kb_i) = e^{-4\pi\hbar t\epsilon(i,\theta_l+|\hbar|^{-1}\theta_l^{\vee})}\Kb_i,\qquad \sigma_t^{\varphi}(\Xb_a) = e^{-4\pi\hbar t\epsilon(a,\theta_l+|\hbar|^{-1}\theta_l^{\vee})}\Xb_a,\qquad i\in I,a\in J. 
\]
\begin{Lem}\label{LemFormInp}
    We have
    \begin{equation}
        \epsilon(i, \theta_l) = \frac{1}{2}\langle i, \frac{1}{2}(\rho - w^{-1} \rho)\rangle, \quad \epsilon(a,\theta_l)=\frac{1}{2}d_{i_a}, \quad \epsilon(i,\theta_l^{\vee}) = \frac{1}{2}\langle i, \frac{1}{2}(\rho^{\vee} - w^{-1} \rho^{\vee})\rangle, \quad \epsilon(a,\theta_l^{\vee}) = \frac{1}{2}
    \end{equation}
    for all $i \in I$ and $a \in J$. 
\end{Lem}

\begin{proof}
    We have, for $i \in I$, using \eqref{EqExplicitFormulaProjection} and \eqref{EqLambdaAsSumRoots}
    \begin{align*}
        \epsilon(i, \theta_l)
        &= \epsilon(i, \frac{1}{2}\sum_{a \in J} \langle \rho, \beta_a^{\vee}\rangle a)\\
        &= \frac{1}{2}\sum_{a \in J} \langle \rho, \beta_a^{\vee}\rangle \langle i , i_a\rangle /2\\
        &= \frac{1}{2}\langle i,  \frac{1}{2}\sum_{a \in J} \langle \rho, \beta_a^{\vee}\rangle i_a\rangle\\
        &= \frac{1}{2}\langle i, \frac{1}{2}(\rho - w^{-1}\rho)\rangle,
    \end{align*}
    and similarly for $\epsilon(i,\theta_l^{\vee})$.
    On the other hand, for $a \in J$ we have
    \begin{equation*}
        \epsilon(a,\theta_l) = \epsilon(a, \theta_l + (\rho - \theta_l)) = \epsilon(a, \rho) =\frac{1}{2} \langle i_a, \rho \rangle = \frac{1}{2} d_{i_a}.
    \end{equation*}
\end{proof}
In particular, if $I$ is of Lie type and $w=w_0$, then
    \begin{equation}
        \epsilon(i,\theta_l) = \frac{1}{2} d_i,\qquad \epsilon(a,\theta_l)=\frac{1}{2}d_{i_a}, \qquad \epsilon(i,\theta_l^{\vee}) = \frac{1}{2},\qquad \epsilon(a,\theta_l^{\vee}) = \frac{1}{2},
    \end{equation}
and hence
\begin{equation}
\sigma_t^{\varphi}(\Kb_i) = e^{-2\pi \hbar t (d_i + |\hbar|^{-1})},\qquad \sigma_t^{\varphi}(\Xb_a) = e^{-2\pi \hbar t (d_{i_a}+|\hbar|^{-1})},\qquad i \in I,a\in J. 
\end{equation}

\subsection{Right invariant nsf weight} 

\begin{Def}
We define the \emph{right (dual) modular vector} as 
\begin{equation}\label{EqFormRightMod}
\theta_r = \varsigma(\theta_l)\in \mfb,\qquad \theta_r^{\vee} = \varsigma(\theta_l^{\vee}) \in \mfb,
\end{equation}
where $\varsigma$ is the reflection map defined in \eqref{EqReflection}.
\end{Def}
\begin{Lem}
We have the following explicit formulas for the right (dual) modular vector. 
\begin{equation}\label{EqExplicitFormulaRightModular}
\theta_r = \frac{1}{2}N_{\rho}=  \frac{1}{2}\sum_{a\in J} \langle\rho,\beta_a^{\vee}\rangle a,\qquad \theta_r^{\vee} = \frac{1}{2}N_{\rho^{\vee}}=  \frac{1}{2}\sum_{a\in J} \langle\rho^{\vee},\beta_a^{\vee}\rangle a,
\end{equation}
\end{Lem}
\begin{proof}
This follows immediately from \eqref{EqExplicitFormulaProjection} and \eqref{EqComppi} given that 
\[
\varsigma(H_{\rho}) = -H_{\rho},\qquad \varsigma(N_{\rho}) = N_{\rho} - H_{\rho}.
\]
Similarly for the right dual modular vector. 
\end{proof}

\begin{Theorem}
A right invariant nsf weight for $(L_{\hbar}(B),\Delta)$ is given by 
\begin{equation}\label{EqExplicitFormRightInv}
\psi  = \int_{\widehat{\mfb}}^{\hbar,\qb(2(\theta_r+|\hbar|^{-1}\theta_r^{\vee}))}.
\end{equation}
\end{Theorem}
\begin{proof}
Recall the map $R$ defined by \eqref{EqUnitAntipode}. From the precise definition of $\int_{\widehat{\mfb}}^{\hbar}$, it follows that 
\[
\int_{\widehat{\mfb}}^{\hbar} \circ R = \int_{\widehat{\mfb}}^{\hbar}
\]
(see \cite{DCSSV25}*{Equation (5.8)})). 

On the other hand, as in the proof of Theorem \ref{TheoInvWeight}, we get that $\psi = \varphi\circ R$ will be a right invariant nsf weight if $\varphi$ is a left invariant nsf weight. 
Hence using the explicit formula for $\varphi$ in \eqref{EqExplicitFormLeftInv}, the theorem follows from the following simple computation:
\[
\psi = \varphi\circ R = \int_{\widehat{\mfb}}^{\hbar,\qb(2(\theta_l+|\hbar|^{-1}\theta_l^{\vee}))}\circ R = \int_{\widehat{\mfb}}^{\hbar,R(\qb(2(\theta_l+|\hbar|^{-1}\theta_l^{\vee})))} = \int_{\widehat{\mfb}}^{\hbar,\qb(2(\theta_r+|\hbar|^{-1}\theta_r^{\vee}))}.
\]
\end{proof}

In particular, the associated modular group is given by 
\[
\sigma_t^{\psi} = \Ad_{\qb(2(\theta_r+|\hbar|^{-1}\theta_r^{\vee}))^{it}},
\]
so
\[
\sigma_t^{\psi}(\Kb_i) = e^{-4\pi\hbar t\epsilon(i,\theta_r +|\hbar|^{-1}\theta_r^{\vee})}\Kb_i,\qquad \sigma_t^{\psi}(\Xb_a) = e^{-4\pi\hbar t\epsilon(a,\theta_r+|\hbar|^{-1}\theta_r^{\vee})}\Xb_a,\qquad i\in I,a\in J. 
\]
Similar to \eqref{LemFormInp}, or using \eqref{EqInvEpsilon}, we get:
\begin{Lem}
    We have
    \begin{equation}
        \epsilon(i, \theta_r) = \frac{1}{2}\langle i, \frac{1}{2}(\rho - w^{-1} \rho)\rangle, \qquad \epsilon(a,\theta_r)=-\frac{1}{2}d_{i_a} + \frac{1}{2}\langle i_a, \frac{1}{2}(\rho - w^{-1} \rho)\rangle,
    \end{equation}
    \begin{equation}
\epsilon(i,\theta_r^{\vee}) = \frac{1}{2}\langle i, \frac{1}{2}(\rho^{\vee} - w^{-1} \rho^{\vee})\rangle, \qquad \epsilon(a,\theta_r^{\vee}) = -\frac{1}{2} + \frac{1}{2}\langle i_a, \frac{1}{2}(\rho^{\vee} - w^{-1} \rho^{\vee})\rangle,
\end{equation}
    for all $i \in I$ and $a \in J$.
\end{Lem}

In particular, if $I$ is of Lie type and $w=w_0$, then
\begin{equation}
 \epsilon(i, \theta_r) = \frac{1}{2}d_i, \qquad \epsilon(a,\theta_r)=0,\qquad \epsilon(i,\theta_r^{\vee}) = \frac{1}{2},\qquad \epsilon(a,\theta_r^{\vee}) = 0,
\end{equation}
and hence 
\begin{equation}
\sigma_t^{\psi}(\Kb_i) = e^{-2\pi \hbar t (d_i + |\hbar|^{-1})},\qquad \sigma_t^{\psi}(\Xb_a) = \Xb_a,\qquad i\in I,a\in J. 
\end{equation}

\subsection{Modular element}

To any locally compact quantum group $(M,\Delta)$ can be associated the \emph{modular element} $\delta$. This is a particular \emph{group-like} positive operator, affiliated to $M$:
\begin{equation}
\Delta(\delta) = \delta\otimes \delta.
\end{equation}
As in the classical setting, it measures how $\varphi$ and $\psi$ differ, and can concretly be seen as a non-commutative Radon--Nikodym derivative of $\varphi$ with respect to $\psi$. In our setting, it is easily computed from the explicit formulas we already obtained: 

Recall the elements $H_{\rho},H_{\rho^{\vee}}$ defined in \eqref{EqComppi}. 
\begin{Theorem}
The modular element of $(L_{\hbar}(B),\Delta)$ is given by 
\begin{equation}\label{EqFormModEl}
\delta =  \qb(H_{\rho}+|\hbar|^{-1}H_{\rho^{\vee}}) \aff L(\mfh)\subseteq L_{\hbar}(B).
\end{equation}
\end{Theorem}
\begin{proof}
By general weight theory and the explicit forms for $\varphi$ and $\psi$ in \eqref{EqExplicitFormLeftInv} and \eqref{EqExplicitFormRightInv}, the modular element must be a positive multiple of 
\begin{equation}\label{EqModElUnref}
\qb(2(\theta_r+|\hbar|^{-1}\theta_r^{\vee}))\star \qb(-2(\theta_l+|\hbar|^{-1}\theta_l^{\vee})). 
\end{equation}

Now from \eqref{EqFormLeftMod} and \eqref{EqFormRightMod} we get
\[
2\theta_r - 2\theta_l = H_{\rho} \in \mfh,\qquad 2\theta_r^{\vee}-2\theta_l^{\vee} = H_{\rho^{\vee}}. 
\]
In particular, the element \eqref{EqModElUnref} equals the element \eqref{EqFormModEl}. As the latter element is group-like, and the modular element is also group-like, they must then coincide. 
\end{proof}

\subsection{Scaling group}

As $(L_{\hbar}(B),\Delta)$ is a locally compact quantum group, $L_{\hbar}(B)$ acquires a particular one-parameter group of $*$-automorphisms $(\tau_t)_{t\in \R}$, the \emph{scaling group}. It also acquires an involutive, anti-comultiplicative $*$-anti-isomorphism
\[
R: L_{\hbar}(B)\rightarrow L_{\hbar}(B),
\]
the \emph{unitary antipode}. The latter notation is not at odds with the notation \eqref{EqUnitAntipode}, as the final part of the following result shows:

\begin{Theorem}
The scaling group of $(L_{\hbar}(B),\Delta)$ is determined by
\begin{equation}\label{EqFormScalGroup}
\tau_t(\Kb_i) = \Kb_i,\qquad \tau_t(\Eb_i)  =  e^{-2\pi t \hbar (d_i +|\hbar|^{-1})}\Eb_i,\qquad i \in I,t\in \R.
\end{equation}
The unitary antipode $R$ coincides with the anti-isomorphism given in Lemma \ref{RefUnitAnti}.
\end{Theorem}
\begin{proof}
By the Stone--von Neumann theorem and the explicit formulas for the coproduct, we have that, for fixed $i$, the von Neumann subalgebra $\{\Kb_i,\Eb_i\}''$, together with the restriction of the coproduct,   defines a locally compact quantum group, namely the quantum $ax+b$-group (see \cite{Ip13a}*{Section 4}, or \cite{DCSSV25}*{Theorem 5.17} for a complete proof with arbitrary parameter). By \cite{BV05}*{Proposition 10.5}, the unitary antipode and scaling group of $(L_{\hbar}(B),\Delta)$ must restrict to those of the locally compact quantum group $(\{\Kb_i,\Eb_i\}'',\Delta)$, which are indeed given by \eqref{EqFormScalGroup} and \eqref{EqUnitAntipodeForm} by \cite{DCSSV25}*{Theorem 6.11 and Theorem 6.12}.
\end{proof}
Remark that \eqref{EqFormScalGroup}  implies 
\begin{equation}\label{EqFormScalGroupX}
 \tau_t(\Xb_a)  =  e^{-2\pi t \hbar (d_{i_a} +|\hbar|^{-1})}\Xb_a = e^{4\pi t \hbar \epsilon(a,\rho+|\hbar|^{-1}\rho^{\vee})}\Xb_a,\qquad a \in J,t\in \R,
\end{equation}
as this defines a one-parameter group of $*$-automorphisms of $L_{\hbar}(\mfn)$, coinciding with $\tau_t$ on the generators $\Eb_i$. 

\begin{Rem}
Recall that analytically, $\tau_t$ and $R$ determine the antipode through its polar decomposition 
\[
S = R\circ \tau_{-i/2}. 
\]
Implementing this algebraically, we find the expected formulas
\[
S(K_i) = K_i^{-1},\qquad S(E_i) = - K_i^{-1} E_i,\qquad i\in I.
\]
\end{Rem}

\subsection{Scaling constant}\label{SecScaling}

\begin{Def}
If $\mbi$ is a word, its associated \emph{fundamental (linear) Laurent polynomial} $p_{\mbi}$ is given as 
\[
p_{\mbi}(x) = 2
\epsilon(\rho +x^{-1}\rho^{\vee},x\theta_l + \theta_l^{\vee}),\qquad x \in \R^{\times}.
\]
\end{Def}

Recall now that the scaling constant $\nu>0$ of a locally compact quantum group $(M,\Delta)$ with left invariant nsf weight $\varphi$ is the unique positive real number such that 
\[
\varphi\circ \tau_t = \nu^{-t}\varphi,\qquad t\in \R. 
\]

From the definition \eqref{EqExplicitFormLeftInv} and the formula \eqref{EqFormScalGroupX}, we find:

\begin{Theorem}
The scaling constant of $(L_{\hbar}(B),\Delta)$ is given by 
\begin{equation}\label{EqScalingConstant}
\nu = e^{-4\pi \sgn(\hbar) p_{\mbi}(|\hbar|)},
\end{equation}
where $p_{\mbi}$ is the fundamental polynomial of $\mbi$.
\end{Theorem}

\begin{Lem}
    We have
    \begin{equation*}
        \epsilon(\rho, \theta_l) = -\frac{1}{4}\langle \rho,H_{\rho}\rangle,
        \quad \epsilon(\rho^{\vee},\theta_l^{\vee})=-\frac{1}{4}\langle \rho^{\vee}, H_{\rho^{\vee}}\rangle,
        \quad \epsilon(\rho,\theta_l^{\vee})=-\frac{1}{4}\langle \rho,H_{\rho^{\vee}}\rangle,
        \quad\epsilon(\rho^{\vee},\theta_l)=-\frac{1}{4}\langle \rho^{\vee},H_{\rho}\rangle.
    \end{equation*}
    
\end{Lem}
\begin{proof}
    We have
    \begin{align*}
        \epsilon(\rho, \theta_l)
        &= \epsilon(\rho, \frac{1}{2}\sum_{a \in J}\langle \rho, \beta_a^{\vee}\rangle a)\\
        &= -\frac{1}{2}\sum_{a \in J}\langle \rho, \beta_a^{\vee}\rangle d_{i_a}/2\\
        &= -\frac{1}{4}\sum_{a \in J}\langle \rho, \beta_a^{\vee}\rangle \langle \varpi_{i_a}, i_a\rangle\\
        &= -\frac{1}{4}\langle \sum_{i \in J} \varpi_i, \sum_{a \in J}\langle \rho, \beta_a^{\vee}\rangle i_a\rangle\\
        &= -\frac{1}{4}\langle \sum_{i \in J} \varpi_i,\rho - w^{-1} \rho\rangle\\
        &= -\frac{1}{4}\langle \rho,\rho - w^{-1} \rho\rangle.
    \end{align*}
    The other cases follow similarly.
\end{proof}

\begin{Rem}
Consider the case where $I$ is of Lie type and $w = w_0$. Then using \eqref{EqHrhoLongestWord} we have $H_{\rho} = 2 \rho_{\mfh}$ and $H_{\rho^\vee} = 2\rho^{\vee}_{\mfh}$. Thus, the scaling constant is given by
\begin{equation}\label{EqFormScaling}
     \nu = \exp\left(4\pi \sgn(\hbar) (|\hbar| \langle \rho, \rho \rangle + 2 \langle \rho^{\vee}, \rho \rangle + |\hbar|^{-1} \langle \rho^{\vee}, \rho^{\vee} \rangle )\right).
\end{equation}
Using the Cartan data for the classical types as in Section \ref{SecIntegrability} one computes the following exact values. The final form given in \eqref{EqIn1} and \eqref{EqIn3} relates to the inner product given by the Killing form, as suggested by the `strange formula' (see e.g.\  \cites{FdV69,Bu00}):
    \begin{enumerate}
        \item\label{EqIn1} $ \langle \rho, \rho\rangle $
        \begin{itemize}
            \item $A_n:$ $n(n+1)(n+2)/12 = 2(n+1) \dim(\mathfrak{sl}_{n+1}(\C))/24$. 
            \item $B_n:$ $n(4n^2-1)/6 = 4(2n-1) \dim (\mathfrak{so}_{2n+1}(\C))/24$.
            \item $C_n:$ $n(n+1)(2n+1)/6 = 4(n+1)\dim(\mathfrak{sp}_{2n}(\C))/24$.
            \item $D_n:$ $n(n-1)(2n-1)/6 = 4(n-1)\dim(\mathfrak{so}_{2n}(\C))/24$.
        \end{itemize}
        \item $\langle \rho, \rho^{\vee}\rangle$ in type $B_n$ and $C_n$ both give $n(n+1)(4n-1)/12$. 
        \item\label{EqIn3}$\langle \rho^{\vee}, \rho^{\vee}\rangle$
        \begin{itemize}
            \item $B_n:$ $n(n+1)(2n+1)/12 = 2(n+1)\dim(\mathfrak{sp}_{2n}(\C))/24$.
            \item $C_n:$ $n(4n^2-1)/12 = 2(2n-1) \dim (\mathfrak{so}_{2n+1}(\C))/24$.
        \end{itemize}
    \end{enumerate}
\end{Rem}
\begin{Rem}\label{RemScalingConstant}
The expression (in the exponent) of the scaling constant in \eqref{EqFormScaling} is very similar to the central charge in the setting of $\mathscr{W}$-algebras at a specific level $k$, see e.g.\ \cite{Ar07}*{Page 220}. Although we do not know why this particular value of the central charge relates to the scaling constant in our setting, we briefly mention what is known in general about the connection between $\mathscr{W}$-algebras and quantum groups. Namely, by means of the Feigin-Frenkel duality \cite{FF91}*{Proposition 5}, the generic logarithmic Kazhdan-Lusztig conjecture (see e.g.\ \cite{CDN26} for a discussion and recent progress) gives in the case of the above specific $k$, through the associated screening operators, a correspondence with representation categories of a `modular double' quantum group at dual roots of unity $q$ and $q'$ (see \cite{CL26} for a discussion in rank 1, and references to the literature). The cross term $\langle\rho,\rho^\vee\rangle$ however indicates a particular interaction between the two parts of this modular double, which prevents the complete commutation between the latter. In the Liouville regime, the above connections have also been explored in work of Teschner \cite{Tes11}, see as well \cite{GS19}*{Section 7}. In this setting, the modular dual is implemented through the operators $\Kb_i^{1/\hbar},\Eb_i^{1/\hbar}$, which indeed fail to commute strongly with the $\Kb_j,\Eb_j$, although they commute (up to a sign) on a dense subspace.
\end{Rem}

\section{Full dualizability}\label{SecFullDual}

In Definition \ref{DefFundUnit}, we associated to any dualizable word a fundamental unitary $W$. Without further conditions, it would seem that one can say little about the structural properties of $W$. What is needed is a strengthening of Definition \ref{DefHeisDualOp}. We again fix a Cartan datum $I$ and $\hbar\in \R^{\times}$. For simplicity, we will assume $I$ non-degenerate, which eases the presentation.

\begin{Def}
We say a word $\mbi$ over $I$ is \emph{fully dualizable} if it admits dual simple root operators $\Fb_i$ for each $i\in I$, such that there exists a word $\mbj$ and a (necessarily unique) faithful normal unital $*$-homomorphism
\begin{equation}\label{EqHatPiFundUnit}
\hat{\pi}: L_{\hbar}(B(\mbj)) \rightarrow L_{\hbar}(\mfb'(\mbi)) \subseteq L_{\hbar}(\mfd(\mbi))
\end{equation}
such that 
\begin{equation}\label{EqHatPiFullDual}
\hat{\pi}(\Eb_i) = \Fb_i,\qquad \hat{\pi}(\Kb_i) = \Lb_i.
\end{equation}
\end{Def}

We will show that if $\mbi$ is fully dualizable, the associated fundamental unitary is a \emph{multiplicative unitary}. For the moment, we however start with a more general setup. 

Fix two words $(J,\mbi),(J',\mbj)$ over the Cartan datum $I$. We denote the associated skew-symmetric vector spaces as 
\[
\mfn = \mfn(\mbi),\qquad \mfb = \mfb(\mbi),\qquad \mfn' = \mfn(\mbj),\qquad \mfb' = \mfb'(\mbj),
\]
but we do for the moment not assume that $\mfb'$ is embedded in $\mfd(\mbi)$. Write again the Hilbert spaces for the tensor realisations of $\mfn$, resp.\ $\mfn'$ as $\Hsp(\mbi)$, resp.\ $\Hsp(\mbj)$, and recall from Remark \ref{RemPresNonDeg} that we can naturally endow them (by non-degeneracy of $I$) also with faithful normal $*$-representations of  respectively $L_{\hbar}(\mfb)$ and $L_{\hbar}(\mfb')$, such that 
\begin{equation}
\Kb_i = \prod_{a\in J} \kb_a^{\langle i,i_a\rangle},\qquad \Lb_{\omega} = \prod_{a\in J'} \lb_a^{\langle \omega,i_a\rangle},\qquad \qquad \Lb_i' = \Lb_{\varpi_i^{\vee}} =  \prod_{b\in J_i'} \lb_b,
\end{equation}
where for the purposes of differentiation we will write the elementary components $(\kb,\eb)$ as $(\lb,\fb)$ when considered for $L_{\hbar}(\mfb')$. Similarly, we then rather write 
\begin{equation}\label{EqNegSimpRootOp}
\Yb_a = \left(\prod_{b \in J',b<a} \kb_b^{\langle i_a,i_b\rangle}\right)\fb_a,\qquad \Fb_i = \boxplus_{a\in J_i'} \Yb_a \aff L_{\hbar}(\mfn'),
\end{equation}
and continue to write
\begin{multline}\label{EqNotParts}
\Hb_{\omega}' = (-\pi i \hbar)^{-1}\ln(\Lb_{\omega}),\qquad \Hb_i' = \Hb_{\varpi_i^{\vee}}',\qquad \hb' = (-\pi i \hbar)^{-1}\ln(\lb),\qquad \hb_a' = (-\pi i \hbar)^{-1}\ln(\lb_a),\\ a\in J',i\in I,\omega \in \mfh'.
\end{multline} 

On the algebraic level we note the corresponding commutation relations
\[
Y_bY_a= q^{\langle j_a,j_b\rangle}Y_aY_b,\quad L_i' Y_a = q^{\delta_{ij}}Y_aL_i',\quad L_i'L_j' = L_j'L_i',\qquad \forall b>a \textrm{ in } J'\textrm{ and } i \in I.
\]

\begin{Def}
For $a\in J,b\in J'$, we define the unitary 
\begin{equation}\label{EqMultUnRank1}
W_{ab} = \overline{F}_{d_{i_a}\hbar}(\delta_{i_a,j_b}\eb_{a} \otimes \fb_b) \overline{e}_{\hbar}(\langle i_a,j_b\rangle \hb_a\otimes \hb_b') \in \Bc(\Hsp(\mbi)\otimes \Hsp(\mbj)).
\end{equation}
\end{Def}

\begin{Rem}
If we write
\begin{equation}\label{EqConstLongMU}
W^{(a,b)} = \overline{F}_{d_{i_a}\hbar}(\delta_{i_a,j_b}\eb \otimes \fb) \overline{e}_{\hbar}(\langle i_a,j_b\rangle \hb\otimes \hb') \in \Bc(L^2(\R)\otimes L^2(\R)),
\end{equation}
then 
\begin{equation}\label{EqMultUnRank1Com}
W_{ab} = W_{ab}^{(a,b)}
\end{equation}
where on the right we use the ordinary leg numbering notation. So, \eqref{EqMultUnRank1} is a shorthand leg numbering notation since the upper indices on the right hand side of \eqref{EqMultUnRank1Com} are determined by the lower ones.
\end{Rem} 

\begin{Def}\label{DefLongProd}
We define 
\begin{equation}\label{EqDecompV}
W = \prod_{a\in J}^{\longra}\prod_{b\in J'}^{\longra} W_{ab} \in \Bc(\Hsp(\mbi)\otimes\Hsp(\mbj)).
\end{equation}
\end{Def}

\begin{Lem}
We have 
\begin{equation}\label{EqDecompVAlt}
W= \prod_{b\in J'}^{\longra}\prod_{a\in J}^{\longra} W_{ab}.
\end{equation}
\end{Lem} 
\begin{proof}
    First note that if $a \neq c$ and $b \neq d$, that then $W_{ab} W_{cd} = W_{cd} W_{ab}$ as they live in different legs of the tensor product. Without loss of generality let $J = \{1, \dots,N\}$ and $J' = \{1, \dots, M\}$. For $0 \leq k \leq M - 1$ we will show that
    \begin{equation*}
        W = W^{(k)} = (W_{11} \dotsb W_{N1}) \dotsb (W_{1k} \dotsb W_{Nk})(W_{1,k+1} \dotsb W_{1,M}) \dotsb (W_{N,k+1}\dotsb W_{N,M}).
    \end{equation*}
    The case $k = 0$ is true by definition of $W$. To go from $k$ to $k+1$ we proceed as follows. By assumption,
    \begin{equation*}
        W^{(k)} = (W_{11} \dotsb W_{N1}) \dotsb (W_{1k} \dotsb W_{Nk})(W_{1,k+1} \dotsb W_{1,M}) \dotsb (W_{N,k+1}\dotsb W_{N,M}).
    \end{equation*}
    Now note that
    \begin{align*}
        &(W_{1,k+1} \dotsb W_{1,M})(W_{2,k+1}\dotsb W_{2,M}) \dotsb (W_{N,k+1}\dotsb W_{N,M})\\
        &= (W_{1,k+1} W_{2,k+1})(W_{1,k+2}\dotsb W_{1,M})(W_{2,k+2}\dotsb W_{2,M})(W_{3,k+1}\dotsb W_{3,M}) \dotsb (W_{N,k+1}\dotsb W_{N,M})\\
        &= \dots \\
        &= (W_{1,k+1} \dotsb W_{N,k+1}) (W_{1,k+2} \dotsb W_{N,k+2})\dotsb (W_{N,k+2}\dotsb W_{N,M}),
    \end{align*}
    where in each step we bring a term $W_{l,k +1}$ to the front. This is possible because for the terms $W_{ab}$ before it we have $a < l$ and $b > k+1$, so that they commute by our initial observation. We now have
    \begin{align*}
        W^{(k)} 
        &= (W_{11} \dotsb W_{N1}) \dotsb (W_{1k} \dotsb W_{Nk})(W_{1,k+1} \dotsb W_{1,M}) \dotsb (W_{N,k+1}\dotsb W_{N,M})\\
        &= (W_{11} \dotsb W_{N1}) \dotsb (W_{1k} \dotsb W_{Nk})(W_{1,k+1}\dotsb W_{N,k+1})(W_{1,k+2} \dotsb W_{1,M}) \dotsb (W_{N,k+2}\dotsb W_{N,M})\\
        &= W^{(k+1)}.
    \end{align*}
    This completes the proof of the lemma, as
    \begin{align*}
        W^{(M-1)} 
        &= (W_{11} \dotsb W_{N1})\dotsb (W_{1,M-1} \dotsb W_{N,M-1})(W_{1M})\dotsb (W_{NM})\\
        &= (W_{11} \dotsb W_{N1})\dotsb (W_{1,M-1} \dotsb W_{N,M-1})(W_{1M}\dotsb W_{NM})\\
        &= \prod_{b\in J'}^{\longra}\prod_{a\in J}^{\longra} W_{ab}.
    \end{align*}
\end{proof}

Let now $\Kc$ be as in Definition \ref{DefMUCartan}. 
\begin{Lem}\label{LemDecomp}
With $W$ as in \eqref{EqDecompVAlt}, we have
\[
W = W_{\br}\Kc,
\]
where 
\begin{equation}\label{EqBraidedV}
W_{\br} = \prod^{\longra}_{a\in J} \prod_{b\in J_{i_a}'}^{\longra}\overline{F}_{d_{i_a}\hbar}(\Xb_a \otimes \Yb_b) = \prod^{\longra}_{b\in J'} \prod^{\longra}_{a\in J_{j_b}}  \overline{F}_{d_{i_a}\hbar}(\Xb_a \otimes \Yb_b).
\end{equation}
\end{Lem}
\begin{proof}
From the Gaussian commutation relation \eqref{EqqCommGauss}, it follows that
\begin{eqnarray*}
W &=& \prod^{\longra}_{a\in J}\left(\left(\prod^{\longra}_{b\in J_{i_a}'}  \overline{F}_{d_{i_a}\hbar}(\eb_a \otimes \Yb_b)\right) \left(\prod_{b\in J'} \overline{e}_{\hbar}(\langle i_a,j_b\rangle \hb_a\otimes \hb_b')\right) \right) \\
&=&  \left(\prod^{\longra}_{a\in J} \prod_{b\in J_{i_a}'}^{\longra}\overline{F}_{d_{i_a}\hbar}(\Xb_a \otimes \Yb_b)\right)\left(\prod_{a\in J,b\in J'}\overline{e}_{\hbar}(\langle i_a,j_b\rangle \hb_{a}\otimes \hb_b') \right),
\end{eqnarray*}
which immediately implies the first identity in \eqref{EqBraidedV}. The other identity follows similarly.
\end{proof}

\begin{Cor}
We have
\begin{equation}
W \in L_{\hbar}(\mfb) \bar{\otimes} L_{\hbar}(\mfb') \subseteq \Bc(\Hsp(\mbi)\otimes \Hsp(\mbj)). 
\end{equation}
\end{Cor}

Consider now again the simple root operators $\Eb_i$ affiliated to $L_{\hbar}(\mfb)$ and $\Fb_i$ affiliated to $L_{\hbar}(\mfb')$. Consider for $a\in J$ and $b\in J'$  the unitary operators
\begin{equation}
W_{\br}^b = \overline{F}_{d_{j_b}\hbar}(\Eb_{j_b} \otimes \Yb_b),\qquad {}^{a}W_{\br} = \overline{F}_{d_{i_a}\hbar}(\Xb_a \otimes \Fb_{i_a})
\end{equation}
on $\Hsp(\mbi)\otimes \Hsp(\mbj)$. 
\begin{Lem}\label{LemBraidPentAlt}
With $W_{\br}$ as in \eqref{EqBraidedV}, we have
\begin{equation}\label{EqSimpGenMUAlt}
W_{\br} = \prod^{\longra}_{b\in J'} W_{\br}^b = \prod^{\longra}_{a\in J} {}^aW_{\br}.
\end{equation}
\end{Lem}
\begin{proof}
The identity
\[
\prod^{\longra}_{b\in J'} W_{\br}^b  =  \prod^{\longra}_{b\in J'} \prod^{\longra}_{a\in J_{j_b}} \overline{F}_{d_{j_b}\hbar}(\Xb_a \otimes \Yb_b)
\]
 follows by induction using the quantum exponential property and the fact that for $a\in J_i$ we have that $\Xb_a$ and $\boxplus_{a'\in J_i,a'<a} \Xb_{a'}$  are $d_i\hbar$-commuting operators. We similarly get the identity
\[
\prod^{\longra}_{a\in J} {}^a W_{\br} =  \prod^{\longra}_{a\in J} \prod_{b\in J_{i_a}'}^{\longra}\overline{F}_{d_{i_a}\hbar}(\Xb_a \otimes \Yb_b).
\]
\end{proof}

\begin{Cor}
If $\mbi$ is a fully dualizable word, then $W$ as defined through Definition \ref{DefLongProd} coincides with the fundamental unitary associated to $\mbi$, as defined in Definition \ref{DefFundUnit}, after applying $\hat{\pi}$ from \eqref{EqHatPiFundUnit} to the second leg of the latter.
\end{Cor}

In the above setup, one can easily show that $W$ is a multiplicative unitary, whose definition we recall below:

\begin{Def}
If $\Hsp_1,\Hsp_2,\Hsp_3$ are Hilbert spaces, we say that operators $(x,y,z)$ with 
\[
x \in \Bc(\Hsp_1\otimes \Hsp_2),\quad y \in \Bc(\Hsp_1\otimes \Hsp_3),\qquad z \in \Bc(\Hsp_2\otimes \Hsp_3)
\]
form a \emph{pentagon triple} for $(\Hsp_1,\Hsp_2,\Hsp_3)$ if 
\[
x_{12}y_{13}z_{23} = z_{23}x_{12}.
\]
If $x,y,z$ are unitary, we call $(x,y,z)$ a \emph{unitary pentagon triple}.

If $x\in \Bc(\Hsp\otimes \Hsp)$, we say that $x$ satisfies the pentagon equation if $(x,x,x)$ is a pentagon triple for $(\Hsp,\Hsp,\Hsp)$.

If $x \in \Bc(\Hsp\otimes \Hsp)$ is a unitary satisfying the pentagon equation, we call $x$ a \emph{multiplicative unitary}. 
\end{Def}

Of course, we can also make sense of multiplicative unitaries $x \in M \bar{\otimes} M$ for any von Neumann algebra $M$. The proof of the following lemma is via a direct verification.

\begin{Lem}\label{LemPentEqIt}
\begin{enumerate}
\item 
If $(x,y,z)$ is a pentagon triple for $(\Hsp_1,\Hsp_2,\Hsp)$ and $(x,y',z')$ is a pentagon triple for $(\Hsp_1,\Hsp_2,\Gsp)$, then $(x,y_{12}y_{13}',z_{12}z_{13}')$ is a pentagon triple for $(\Hsp_1,\Hsp_2,\Hsp\otimes \Gsp)$. 
\item If $(x,y,z)$ is a pentagon triple for $(\Hsp,\Hsp_2,\Hsp_3)$ and $(x',y',z)$ is a pentagon triple for $(\Gsp,\Hsp_2,\Hsp_3)$, then $(x_{13}x_{23}',y_{13}y_{23}',z)$ is a pentagon triple for $(\Hsp\otimes \Gsp,\Hsp_2,\Hsp_3)$.
\end{enumerate}
\end{Lem} 

\begin{Theorem}\label{TheoMultUni}
If $\mbi$ is a fully dualizable word, then the fundamental unitary
\[
W \in L_{\hbar}(\mfd)\bar{\otimes}L_{\hbar}(\mfd)
\]
is a multiplicative unitary.
\end{Theorem} 
\begin{proof}
The pentagon identity
\[
W_{12}W_{13}W_{23} = W_{23}W_{12}
\]
can be made sense of as an identity in $L_{\hbar}(\mfb) \bar{\otimes}L_{\hbar}(\mfd) \bar{\otimes}L_{\hbar}(\mfb')$. It is hence sufficient to show that  $(W,W,W)$ is a pentagon triple for $(\Hsp(\mbi),\Hsp,\Hsp(\mbj))$, with $\Hsp$ any Hilbert space with a faithful normal $L_{\hbar}(\mfd)$-representation. 

By Lemma \ref{LemPentEqIt} and Lemma \ref{LemBraidPentAlt} (together with an easy pushing through of Cartan elements again), it is by iteration enough to show that $({}^aW,W^{(a,b)},W^b)$ forms a pentagon triple for $(L^2(\R),\Hsp,L^2(\R))$, where for $a\in J,b\in J'$ we put
\[
{}^aW = \overline{F}_{d_{i_a}\hbar}(\eb \otimes \Fb_{j_a})\overline{e}_{\hbar}(\hb\otimes d_{i_a}\Hb_{i_a}'),\qquad W^b = \overline{F}_{d_{j_b}\hbar}(\Eb_{j_b}\otimes \fb)\overline{e}_{\hbar}(d_{j_b}\Hb_{j_b}\otimes \hb)
\]
and as before
\[
W^{(a,b)} =  \overline{F}_{d_{i_a}\hbar}(\delta_{i_a,j_b}\eb \otimes \fb) \overline{e}_{\hbar}(\langle i_a,j_b\rangle \hb\otimes \hb').
\]
But once again moving the Cartan part outside (to the right), these pentagon identities split into the pentagon identity for the triple of operators
\[
\overline{e}_{\hbar}(\hb\otimes d_{i_a}\Hb_{i_a}'),\quad \overline{e}_{\hbar}(\langle i_a,j_b\rangle \hb\otimes \hb'),\quad \overline{e}_{\hbar}(d_{j_b}\Hb_{j_b}\otimes \hb'),
\]
which is easy, and a braided type of pentagon equation for the remaining parts, namely for all $a\in J, b\in J'$:
\[
 \overline{F}_{d_{i_a}\hbar}(\eb \otimes\Fb_{j_a}\otimes 1)\overline{F}_{d_{i_a}\hbar}(\delta_{i_a,j_b} \eb \otimes \Lb_{i_a}\otimes \fb) \overline{F}_{d_{j_b}\hbar}(1\otimes \Eb_{j_b}\otimes \fb) 
=  \overline{F}_{d_{j_b}\hbar}(1\otimes \Eb_{j_b}\otimes \fb)  \overline{F}_{d_{i_a}\hbar}(\eb \otimes \Fb_{j_a}\otimes 1).\]
But these identities either follow from strong commutation (if $i_a\neq j_b$), or by Theorem \ref{TheoHeisPent} (if $i_a = j_b$).
\end{proof}

\begin{Exa}
One can see explictly that in type $A_n$, the word 
\[
(1;2,1;3,2,1;\ldots;n,n-1,\ldots,2,1)
\]
is fully dualizable. This follows from the construction of \cite{DCSSV25}, but let us give a direct argument: write in this case 
\[
J = \{ij\mid 1\leq i\leq n, 1\leq j \leq i\},\qquad \mbi(ij) = j,
\]
with the lexicographic ordering, taking the normal ordering on the $i$-variable and the opposite ordering on the $j$-variable. Then we have 
\[
\Eb_k = \boxplus_{k\leq i \leq n} \Xb_{ik},\qquad 1\leq k \leq n.
\]
Let us now write 
\[
\Zb_{jj} = \Xb_{jj},\qquad \Zb_{ij} = \Xb_{ij}\star \Xb_{i-1,j}^{-1},\qquad 1\leq j \leq n, n \geq i>j, 
\]
so that $\Eb_k = \boxplus_{k\leq i \leq n} \bigstar_{j=k}^i \Zb_{jk}$ for $1\leq k \leq n$. Further write 
\[
Z_{k0} = \Lb_k \star \bigstar_{i=1}^k \Zb_{ki}^{-1},\qquad 1\leq k \leq n,
\]
as affiliated to $L_{\hbar}(\mfd)$, and put 
\[
\Yb_{ij} = \bigstar_{k=0}^{i-1} \Zb_{jk},\qquad 1\leq j\leq n, 1 \leq i \leq j. 
\]
Then with 
\begin{equation}\label{EqWordOpp}
J' = \{ij\mid 1\leq j \leq n, 1\leq i \leq j\},\qquad \mbj(ij) = j
\end{equation}
one verifies that:
\begin{enumerate}
\item The $\{\Yb_{ij}\mid ij\in J'\}$, together with the $\Lb_i$, realize a copy of $L_{\hbar}(\mfb'(\mbj))$, where the word $\mbj$ in \eqref{EqWordOpp} is determined by giving again $J'$ the total lexicographic order with the opposite order in the second component, so corresponding to the word 
\[
\mbj = (n,n-1,\ldots,1;n,n-1,\ldots,2;\ldots; n,n-1;n). 
\]
\item The 
\[
\Fb_i = \boxplus_{j=1}^i \Yb_{ji},\qquad 1\leq i \leq n 
\]
form dual simple root operators for $\mbi$.  
\end{enumerate}
The proof of the first statement is verified by a direct (combinatorial) argument. The second statement is checked by a well-organized iterated application of either strong commutation relations or the pentagon equation (see e.g.\ \cite{DCSSV25}*{Appendix G}). We do not spell out the details here, as it is explained in great detail in \cite{DCSSV25}. 
\end{Exa}

The above situation is special, as in type $A$ it is possible to start with a reduced word $\mbi$ for $w_0$ containing one of the simple roots exactly once. This is however impossible in all of the other types, which means that any set of dual simple root operators, with the first simple root operator  constructed as in the proof of Proposition \ref{PropExchDual}, can never be a Feigin realisation of $L_{\hbar}(\mfb')$. 

A global expression for the simple root operators $\Fb_i$ in all Lie types can be gleaned from \cite{Ip18a}*{Section 7}. However, one needs \emph{cluster mutations} to transform these expressions into Feigin form. We do not deal with this problem here in general, but simply illustrate the phenomenon in type $C_2$. 

\begin{Exa}
Consider the Cartan datum of type $C_2$ as in Section \ref{SecCartTypeC}, and consider the words
\[
\mbi = (1,2,1,2),\qquad \mbj = (2,1,2,1),
\]
with the quantum torus generators of $L_{\hbar}(\mfn(\mbi))$, resp.\ $L_{\hbar}(\mfn(\mbj))$ written as $\Xb_1,\ldots,\Xb_4$, resp.\ $\Zb_1,\ldots,\Zb_4$. Put 
\[
U = F_{\hbar}(\mbX_2\star \Xb_4^{-1})F_{2\hbar}(\Xb_1\star \Xb_3^{-1})F_{\hbar}(\Xb_2^{-1}\star \Xb_4),
\]
and put 
\[
\Xb_1' = \Xb_4,\qquad \Xb_2' = \Xb_2^2\star \Xb_3\star \Xb_4^{-2},\qquad \Xb_3' = \Xb_2^{-1}\star \Xb_4^2,\qquad \Xb_4' = \Xb_1\star \Xb_2^2 \star \Xb_4^{-2}.
\]
Then it follows from the results in Section \ref{Secmij4} that we can identify $L_{\hbar}(\mfb(\mbj)) \cong L_{\hbar}(\mfb(\mbi))$ in such a way that 
\[
\Zb_i = U \Xb_i' U^*,\qquad \Eb_1 = \Xb_1\boxplus \Xb_3 = \Zb_2\boxplus \Zb_4,\qquad \Eb_2 = \Xb_2\boxplus \Xb_4 = \Zb_1 \boxplus \Zb_3.  
\]
From the proof of Proposition \ref{PropExchDual}, we then know that 
\[
\Fb_1 = \Xb_1^{-1}\star \Lb_1,\qquad \Fb_2 = \Zb_1^{-1}\star \Lb_2
\]
will be dual simple root operators for $\mbi$, affiliated to $L_{\hbar}(\mfd(\mbi))$. To show that $\mbi$ is fully dualizable, we hence need to realize $\Fb_1$ and $\Fb_2$ in Feigin form. 

For this, put 
\[
V = F_{\hbar}(\mbX_2\star \Xb_4^{-1})F_{2\hbar}(\Xb_1\star \Xb_3^{-1}). 
\]
Then one computes from \eqref{EqFormSumIdAlt} that 
\[
\Zb_1^{-1}= U \Xb_4^{-1} U^* = V (\Xb_4^{-1}\boxplus \Xb_2^{-1})V^*,
\]
\[
V^* \Xb_1^{-1}V = F_{2\hbar}(\Xb_1\star \Xb_3^{-1})^* \Xb_1^{-1} F_{2\hbar}(\Xb_1\star \Xb_3^{-1}) = \Xb_1^{-1}\boxplus \Xb_3^{-1},
\]
so 
\[
V^* \Fb_1 V = \Yb_1 \boxplus \Yb_3,\qquad V^* \Fb_2V = \Yb_2\boxplus \Yb_4,\qquad  V^*\Lb_1V = \Lb_1,\qquad V^*\Lb_2V = \Lb_2,
\]
with 
\[
\Yb_1 = \Lb_1\star \Xb_1^{-1},\qquad \Yb_2 = \Lb_2\star \Xb_2^{-1},\qquad \Yb_3 = \Lb_1\star \Xb_3^{-1},\qquad \Yb_4 = \Lb_2\star \Xb_4^{-1}. 
\]
It is now easily checked that the $\Yb_a$ and $\Lb_i$ provide a Feigin realization of $L_{\hbar}(\mfb'(\mbi))$, with $\mbi$ the word $(1,2,1,2)$ as before. 
\end{Exa}

\section{Outlook}

In this paper, we have shown how locally compact quantum groups can be constructed starting from a general Cartan matrix and a word over its index set satisfying two conditions, \emph{full exchangeability} and \emph{integrability}. We have then shown that for Cartan matrices of classical type, these two conditions are indeed satisfied. We end with the following list of open problems.

\begin{enumerate}
\item Following \cite{Ip18a}*{Section 7}, any longest word associated to a  Cartan datum of Lie type should also be fully exchangeable, but the $G_2$-case needs a careful analytic treatment. 
\item At the moment, it is not clear how to prove the integrability for all the exceptional types, although a case-by-case study should be feasible with the methods in this paper. However, $F_4$ and $E_8$ might be more problematic due to the non-existence of a `nice word' for $w_0$ \cite{Lit98}. Algebraically, there is some more freedom using iterated quantum commutators or more elaborate representation-theoretic arguments to establish the algebraic analogue of integrability (stating that the simple root operators generate the fraction field of the quantum torus generators as a skew field). It is however not clear at the moment how these methods can be implemented rigorously on the analytic level. 
\item The following question is probably very hard in general: can one characterize those words which are fully exchangeable, resp.\ integrable? The former will happen for words which are not too short, while the latter happens for words which are not too long. It would be very interesting to see if these conditions are ever satisfied outside of the setting of semisimple Lie groups.
\item There are general methods available to construct multiplicative unitaries associated to any Lie type, using the machinery of \cites{FG09,GS19}. However, at the moment it is not clear to us what general methods would be needed to realize those multiplicative unitaries, within a rigorous analytic framework, in a manner suitable for the manipulations we need (e.g.\ modularity). It is also not clear to us how those techniques, or quantum cluster algebra techniques in general, can be used to establish integrability of a given word.
\item Finally, it would be interesting to provide a more complete picture of the regular (left or right) multiplicative unitary associated to the locally compact quantum groups we construct from fully exchangeable integrable words. As in \cite{DCSSV25}, this will however need again the more elaborate quantum cluster algebra formalism.
\end{enumerate}

\appendix

\section{Proof of Lemma \ref{LemDistr}} \label{SecLemm}

We first make some preliminary remarks. 

Fix $\hbar \in \R$. For any $\lambda>0$ and $x>0$ write 
\[
f_{\lambda}(x) = e^{-\lambda \ln(x)^2},\quad k_{\lambda}(x) = e^{-\lambda (\ln(x)+\pi i \hbar/2)^2},\quad h_{\lambda}(x) = e^{-\pi i \hbar/2}xf_{\lambda}(x),\quad g_{\lambda}(x) = xk_{\lambda}(x).
\] 
Then $f_{\lambda},g_{\lambda},h_{\lambda},k_{\lambda}$ all define $C_0$-functions on $(0,\infty)$. 

Let now $\Ab$ be a positive operator on a Hilbert space $\Hsp$. We  write
\[
\xi_{\lambda} = f_{\lambda}(\Ab^{1/2})\xi,\qquad \xi \in \Hsp. 
\]
Then $\xi_{\lambda} \in \msD(\Ab^t)$ for any $t\in \R$, and similarly $k_{\lambda}(\Ab^{1/2})\xi \in \msD(\Ab^t)$ for any $t\in \R$. 

Let $\Db$ be another positive operator on $\Hsp$, and assume that $(\Ab,\Db)$ forms an $\hbar$-commuting pair. The following statements can then be proven similarly as in  \cite{Wor00}*{Section 3}\footnote{The proofs in \cite{Wor00} are only formulated for $|\hbar|<1/2$, but can be generalized to arbitrary $\hbar$. Alternatively, one can simply iterate the lemma, starting from a sufficiently high root of $\Ab$.}. Alternatively, one can invoke the Stone--von Neumann theorem in case $\hbar\neq 0$, and prove the lemmas below directly for the pair of operators $\Ab =\kb^2,\Db=\eb$ given by \eqref{EqGenke}, using the description for the domain of (powers of) $\eb$ as given in \cite{Sch94}*{Lemma 1.1} (see also \cite{DCSSV25}*{Lemma F.10}). For $\hbar =0$, the lemmas below are proven directly.

\begin{Lem}\label{LemApp1}
If $\xi \in \msD(\Db^{1/2})$ and $s,t\in \R$, then $\Ab^t\xi_{\lambda} \in \msD(\Db^{1/2}\star \Ab^s)$, and
\[
(\Db^{1/2}\star\Ab^{s+t}) \xi_{\lambda} = e^{-\pi i t\hbar/2}(\Db^{1/2}\star \Ab^s)\Ab^t \xi_{\lambda} = e^{\pi i(s+t) \hbar/2} \Ab^{s+t} k_{\lambda}(\Ab^{1/2})\Db^{1/2}\xi. 
\]
\end{Lem} 

\begin{Lem}\label{LemApp2}
The linear span of all $\xi_{\lambda}$ for $\lambda>0$ and $\xi\in \msD(\Db^{1/2})$ forms a core for $\Db^{1/2}$. More precisely, if $\xi \in \msD(\Db^{1/2})$, then 
\begin{equation}\label{EqCoreDCond}
\lim_{\lambda \rightarrow 0} \xi_{\lambda} = \xi,\qquad \lim_{\lambda \rightarrow 0}  \Db^{1/2}\xi_{\lambda} = \Db^{1/2}\xi.
\end{equation}
\end{Lem}

Assume now the setting and notation of Lemma \ref{LemDistr}, and put
\[
\Db = \Bb\boxplus \Cb.
\]
The $\hbar$-commutation  \eqref{EqSkewComm} of $\Ab, \Bb$, resp.\ $\Ab,\Cb$ is then equivalent to 
\[
\Ab^{it}\Bb \Ab^{-it} = e^{2\pi  \hbar t} \Bb,\qquad \Ab^{it}\Cb \Ab^{-it} = e^{2\pi \hbar t}\Cb. 
\]
Using \eqref{EqPropConjFormSum}, we see that indeed $\Ab$ and $\Db$ $\hbar$-commute, and the left hand side of \eqref{EqCompDist} is meaningful. 

Take now
\[
\xi \in \msD(\Ab^{1/2}\star \Db^{1/2}).
\]
Applying Lemma \ref{LemApp1}, with $\Db^{1/2}$ in that lemma replaced by $\Ab^{1/2}\star\Db^{1/2}$ and taking $s = -1/2$ and $t = 1/2$, we find with respect to any $\lambda>0$ that
\[
\{\xi_{\lambda},\Ab^{1/2}\xi_{\lambda}\}\subseteq \msD(\Db^{1/2}) = \msD(\Bb^{1/2})\cap \msD(\Cb^{1/2}),
\]
with 
\begin{eqnarray*}
\|(\Ab^{1/2}\star \Db^{1/2})\xi_{\lambda}\|^2 &=& \|\Db^{1/2}\Ab^{1/2}\xi_{\lambda}\|^2\\
&=& \|\Bb^{1/2}\Ab^{1/2}\xi_{\lambda}\|^2+ \|\Cb^{1/2}\Ab^{1/2}\xi_{\lambda}\|^2 \\
&=&  \|(\Ab^{1/2}\star \Bb^{1/2})\xi_{\lambda}\|^2 + \|(\Ab^{1/2}\star \Cb^{1/2})\xi_{\lambda}\|^2,
\end{eqnarray*}
where in the last step we apply again Lemma \ref{LemApp1}, now with respect to $\xi_{\lambda}$ and $\Bb$, resp.\ $\Cb$, noting that $(\xi_{\lambda})_{\lambda} = \xi_{2\lambda}$, and using that $\lambda$ was arbitrary.  

Since the left hand side converges to $\|(\Ab^{1/2}\star \Db^{1/2})\xi\|^2$, it follows that the norms $\|(\Ab^{1/2}\star \Bb^{1/2})\xi_{\lambda}\|$ and $\|(\Ab^{1/2}\star \Cb^{1/2})\xi_{\lambda}\|$ stay bounded as $\lambda$ goes to zero, which (by functional calculus and Fatou's lemma) is sufficient to guarantee $\xi \in \msD(\Ab^{1/2}\star \Bb^{1/2})\cap \msD(\Ab^{1/2}\star \Cb^{1/2})$, with 
\begin{equation}\label{EqEqualFormDist}
\|(\Ab^{1/2}\star \Db^{1/2})\xi\|^2 \geq  \|(\Ab^{1/2}\star \Bb^{1/2})\xi\|^2 + \|(\Ab^{1/2}\star \Cb^{1/2})\xi\|^2.
\end{equation}

On the other hand, we can do the same computation starting from $\xi\in \msD(\Ab^{1/2}\star \Bb^{1/2})\cap \msD(\Ab^{1/2}\star \Cb^{1/2})$, and since the precise core condition \eqref{EqCoreDCond} for the core of Lemma \ref{LemApp2} applies to both $\Ab\star \Bb$ and $\Ab\star \Cb$ simultaneously, we can conclude that $\xi \in \msD(\Db^{1/2})$, with 
\begin{equation}\label{EqEqualFormDist2}
\|(\Ab^{1/2}\star \Db^{1/2})\xi\|^2 \leq  \|(\Ab^{1/2}\star \Bb^{1/2})\xi\|^2 + \|(\Ab^{1/2}\star \Cb^{1/2})\xi\|^2.
\end{equation}

Taken together, these two statements give that 
\[
\msD(\Ab^{1/2}\star \Db^{1/2}) = \msD(\Ab^{1/2}\star\Bb^{1/2})\cap \msD(\Ab^{1/2}\star \Cb^{1/2}),
\]
with 
\begin{equation}\label{EqEqualFormDist3}
\|(\Ab^{1/2}\star \Db^{1/2})\xi\|^2 =  \|(\Ab^{1/2}\star \Bb^{1/2})\xi\|^2 + \|(\Ab^{1/2}\star \Cb^{1/2})\xi\|^2.
\end{equation}
This is exactly the statement that \eqref{EqCompDist} holds.

\section{Proof of the pentagon equation in type \texorpdfstring{$B$}{}} \label{SecLemmTypeB}

In this section we prove the identity 
\eqref{EqPentTypeB}. As far as we know, this identity had only been shown in the formal setting (see e.g.\ \cite{KT91}, \cite{Ip18a}*{Appendix}). Below, we deduce the above result from the ordinary pentagon equation \eqref{EqPentKash}. 

It will be convenient to work with the quantum dilogarithm function 
\[
\varphi_{\hbar}(x) = F_{\hbar}(e^x),\qquad x\in \R. 
\]
In the following lemma we collect some of the properties of $\varphi_{\hbar}$ that we will need. We write 
\[
\epsilon = \sgn(\hbar).
\]
\begin{Lem}
\begin{enumerate}
\item $\varphi_{\hbar}$ extends to a meromorphic function on $\C$.
\item The zeroes of $\varphi_{\hbar}$ are located at 
\[
Z = \{- \epsilon\pi i \left((2k+1) + (2l+1)\theta^{-1}\right)\mid k,l\in \Z_{\geq0}\},
\]
while its poles are located at $-Z$. 
\item The function $\varphi_{\hbar}$ is bounded on $\{z \in \C\mid 0\leq  \epsilon\Imm(z)\leq \pi\}$.
\item We have the following functional identities, relating quantum dilogarithm functions at different parameters:
\begin{equation}\label{EqNegh}
\varphi_{-\hbar}(z) = \varphi_{\hbar}(z)^{-1},\qquad z\in \C,
\end{equation}
\begin{equation}\label{EqDoubleh}
\varphi_{2\hbar}(2z) = \varphi_{\hbar}(z+i\pi/2)\varphi_{\hbar}(z-i\pi/2),\qquad z\in \C. 
\end{equation}
\end{enumerate}
\end{Lem}
\begin{proof}
The first three properties and the identity \eqref{EqNegh} are well-known, see e.g.\ \cite{DCSSV25}*{Proposition A.2}. The identity \eqref{EqDoubleh} follows directly from (the elementary) \cite{Rui97}*{Equation (3.25)} (after applying the transformation formula for the hyperbolic gamma function in terms of the quantum dilogarithm, see e.g.\ \cite{DCSSV25}*{Example A.7}.
\end{proof} 

Now by the Stone--von Neumann theorem, it is sufficient to prove \eqref{EqPentTypeB} for 
\[
\Ab = \kb^2,\qquad \Bb = \eb^2,
\]
with $\kb,\eb$ as in \eqref{EqGenke}. Write 
\[
\xb = \ln(\kb),\qquad \pb = \ln(\eb),
\]
so $\xb,\pb$ are the ordinary position and momentum operator. Recall that for any $\alpha,\beta\in \R$, the operator $\alpha \xb + \beta \pb$ is essentially self-adjoint, and its closure $\alpha\xb \dotplus \beta\pb$ satisfies
\[
e^{\alpha\xb \dotplus \beta\pb} = \kb^{\alpha}\star \eb^{\beta}. 
\]
Hence \eqref{EqPentTypeB} is equivalent with proving 
\begin{equation}\label{EqPentTypeBAlt}
\varphi_{2\hbar}(2\xb)\varphi_{\hbar}(2\pb) = \varphi_{\hbar}(2\pb)\varphi_{2\hbar}(2\xb \dotplus 4\pb) \varphi_{\hbar}(2\xb\dotplus 2\pb) \varphi_{2\hbar}(2\xb).
\end{equation}
On the other hand, \eqref{EqPentKash} is equivalent to 
\begin{equation}\label{EqPentaLog}
\varphi_{\hbar}(\xb)\varphi_{\hbar}(2\pb) = \varphi_{\hbar}(2\pb)\varphi_{\hbar}(\xb\dotplus 2\pb)\varphi_{\hbar}(\xb). 
\end{equation}
Let us then show how, intuitively, \eqref{EqPentTypeBAlt} follows from \eqref{EqPentaLog} by translating $\xb$ and $\pb$ with an imaginary constant: using \eqref{EqDoubleh}, we have 
\begin{eqnarray*}
\varphi_{\hbar}(2x)\varphi_{\hbar}(2p) &=& \varphi_{\hbar}(x+i\pi/2) \varphi_{\hbar}(x-i\pi/2) \varphi_{\hbar}(2p)\\
&=& \varphi_{\hbar}(x+i\pi/2) \varphi_{\hbar}(2p)\varphi_{\hbar}(x + 2p-i\pi/2)\varphi_{\hbar}(x-i\pi/2) \\
&=& \varphi_{\hbar}(2p)\varphi_{\hbar}(x+2p+i\pi/2)\varphi_{\hbar}(x+i\pi/2) \varphi_{\hbar}(x + 2p-i\pi/2)\varphi_{\hbar}(x-i\pi/2)\\
&=& \varphi_{\hbar}(2p)\varphi_{\hbar}(x+2p+i\pi/2) \varphi_{\hbar}(x + 2p-i\pi/2)\varphi_{\hbar}(2x+2p)\varphi_{\hbar}(x+i\pi/2)\varphi_{\hbar}(x-i\pi/2)\\
&=& \varphi_{\hbar}(2p)\varphi_{2\hbar}(2x+4p)\varphi_{\hbar}(2x+2p)\varphi_{2\hbar}(2x). 
\end{eqnarray*} 

Below we justify this computation on the analytic level. Put 
\[
\varphi^{(+)}_{\hbar}(z) = \varphi_{\hbar}(z+\epsilon \pi i /2),\qquad \varphi^{(-)}_{\hbar}(z) = 1/\varphi_{\hbar}(z-\epsilon \pi i /2).
\]
Then $\varphi^{(\pm)}_{\hbar}$ are bounded functions, and \eqref{EqDoubleh} gives 
\begin{equation}\label{EqNewDoubl}
\varphi_{2\hbar}(2z) = \varphi_{\hbar}^{(+)}(z)/\varphi_{\hbar}^{(-)}(z).  
\end{equation}
Hence, for any self-adjoint operator $\hb$ on a Hilbert space $\Hsp$, we have that $\varphi_{\hbar}^{(\pm)}(\hb)$ is bounded, with densily defined (unbounded) inverse. In particular,
\begin{equation}\label{EqDomInvVarphi}
\msD(\varphi_{\hbar}^{(-)}(\hb)^{-1}) = \varphi_{\hbar}^{(+)}(\hb)\Hsp,\qquad \varphi_{\hbar}^{(-)}(\hb)^{-1}\varphi_{\hbar}^{(+)}(\hb)\xi = \varphi_{2\hbar}(2\hb)\xi,\qquad \xi \in \Hsp. 
\end{equation}

Now \eqref{EqPentaLog} entails also 
\begin{equation}\label{EqPentaLogScaled}
\varphi_{\hbar}(\xb+t)\varphi_{\hbar}(2\pb) = \varphi_{\hbar}(2\pb)\varphi_{\hbar}(\xb\dotplus 2\pb+t)\varphi_{\hbar}(\xb+t),\qquad t\in \R.
\end{equation}
Since both sides are the boundary values of holomorphic functions in $t$ when $0< \epsilon \Imm(t)<\pi$, we get that also 
\begin{equation}\label{EqPentaLogScaledIm}
\varphi_{\hbar}^{(+)}(\xb)\varphi_{\hbar}(2\pb) = \varphi_{\hbar}(2\pb)\varphi_{\hbar}^{(+)}(\xb\dotplus 2\pb)\varphi_{\hbar}^{(+)}(\xb),\qquad \varphi_{\hbar}^{(-)}(\xb) \varphi_{\hbar}(2\pb) = \varphi_{\hbar}(2\pb) \varphi_{\hbar}^{(-)}(\xb)\varphi_{\hbar}^{(-)}(\xb\dotplus 2\pb). 
\end{equation}

We now need the following lemma: 
\begin{Lem}
If $\xi \in L^2(\R)$, then 
\[
\varphi_{\hbar}^{(+)}(\xb +2\pb)\varphi_{\hbar}^{(+)}(\xb)\xi \in \msD(\varphi_{\hbar}^{(-)}(\xb)^{-1}),
\]
and 
\begin{equation}\label{EqInvForm}
\varphi_{\hbar}^{(-)}(\xb)^{-1}\varphi_{\hbar}^{(+)}(\xb +2\pb)\varphi_{\hbar}^{(+)}(\xb)\xi  = \varphi_{\hbar}^{(+)}(\xb\dotplus 2\pb)\varphi_{\hbar}(2\xb\dotplus 2\pb)\varphi_{2\hbar}(2\xb)\xi.
\end{equation}
\end{Lem} 
\begin{proof}
We need to show the identity 
\[
\varphi_{\hbar}^{(+)}(\xb +2\pb)\varphi_{\hbar}^{(+)}(\xb) = \varphi_{\hbar}^{(-)}(\xb)\varphi_{\hbar}^{(+)}(\xb\dotplus 2\pb)\varphi_{\hbar}(2\xb\dotplus 2\pb)\varphi_{2\hbar}(2\xb),
\]
which by \eqref{EqNewDoubl} will follow from 
\begin{equation}\label{EqHolExtArg}
\varphi_{\hbar}^{(+)}(\xb +2\pb)\varphi_{\hbar}^{(-)}(\xb) = \varphi_{\hbar}^{(-)}(\xb)\varphi_{\hbar}^{(+)}(\xb\dotplus 2\pb)\varphi_{\hbar}(2\xb\dotplus 2\pb).
\end{equation}
But both sides are the continuous boundary values of holomorphic functions 
\begin{equation}\label{EqHolMorph}
z \mapsto \varphi_{\hbar}^{(+)}(\xb +2\pb+z)\varphi_{\hbar}^{(-)}(\xb+z),\qquad z \mapsto \varphi_{\hbar}^{(-)}(\xb+z)\varphi_{\hbar}^{(+)}(\xb\dotplus 2\pb+z)\varphi_{\hbar}(2\xb\dotplus 2\pb+2z)
\end{equation}
on the strip $\{z\in \C\mid 0< \epsilon \Imm(z) < \pi/2\}$. However, the functions in \eqref{EqHolMorph} also extend continuously to the other boundary of this strip, where with respect to the parameter $z = t+\epsilon \pi i/2$, with $t\in \R$, they become the functions 
\begin{equation}\label{EqHolMorphReal}
t \mapsto \varphi_{\hbar}(\xb +2\pb+ t + \pi i \epsilon )\varphi_{\hbar}(\xb+t)^{-1},\qquad t \mapsto \varphi_{\hbar}(\xb+t)^{-1}\varphi_{\hbar}(\xb\dotplus 2\pb+ t + \pi i \epsilon)\varphi_{\hbar}(2\xb\dotplus 2\pb+ 2t + \pi i \epsilon).
\end{equation}
Hence \eqref{EqHolExtArg} will follow once we show the identity 
\begin{equation}\label{EqNewId}
\varphi_{\hbar}(\xb+t) \varphi_{\hbar}(\xb +2\pb+ t + \pi i \epsilon )= \varphi_{\hbar}(\xb\dotplus 2\pb+ t + \pi i \epsilon)\varphi_{\hbar}(2\xb\dotplus 2\pb+ 2t + \pi i \epsilon)\varphi_{\hbar}(\xb+t),\qquad t \in \R.
\end{equation}
But in turn, both sides are boundary values of the holomorphic functions 
\[
z \mapsto \varphi_{\hbar}(\xb+t) \varphi_{\hbar}(\xb +2\pb+ t + z),\qquad z \mapsto \varphi_{\hbar}(\xb\dotplus 2\pb+ t +z)\varphi_{\hbar}(2\xb\dotplus 2\pb+ 2t+z)\varphi_{\hbar}(\xb+t)
\]
on the strip $\{z\in \C\mid 0< \epsilon \Imm(z) < \pi/2\}$. As the functions coincide on $\R$ by \eqref{EqPentaLog}, they coincide everywhere, proving \eqref{EqNewId} and hence establishing the lemma. 
\end{proof}

We can now prove \eqref{EqPentTypeBAlt}: using juxtaposition to denote composition of unbounded operators, we have 
\begin{eqnarray*}
\varphi_{2\hbar}(2\xb)\varphi_{\hbar}(2\pb) 
&\underset{\eqref{EqNewDoubl}}{=}& \varphi_{\hbar}^{(-)}(\xb)^{-1}\varphi_{\hbar}^{(+)}(\xb) \varphi_{\hbar}(2\pb) \\
&\underset{\eqref{EqPentaLogScaledIm}}{=}& \varphi_{\hbar}(2\pb) (\varphi_{\hbar}^{(-)}(\xb) \varphi_{\hbar}^{(-)}(\xb \dotplus 2\pb))^{-1}\varphi_{\hbar}(\xb\dotplus 2\pb)\varphi_{\hbar}^{(+)}(\xb)\\
&\underset{\eqref{EqInvForm}}{=}&
\varphi_{\hbar}(2\pb)\varphi_{\hbar}^{(-)}(\xb \dotplus 2\pb)^{-1}\varphi_{\hbar}^{(+)}(\xb\dotplus 2\pb)\varphi_{\hbar}(2\xb \dotplus 2\pb)\varphi_{2\hbar}(2\xb)\\
&\underset{\eqref{EqNewDoubl}}{=}& 
\varphi_{\hbar}(2\pb)\varphi_{2\hbar}(2\xb\dotplus 4\pb)\varphi_{\hbar}(2\xb \dotplus 2\pb) \varphi_{2\hbar}(2\xb).
\end{eqnarray*}

\begin{Rem}
As stated in \cite{Ip18a}*{Appendix}, the identity \eqref{EqPentTypeB} also leads to the following identities: if  $\Ab,\Bb$ $2\hbar$-commute (so algebraically $AB = q^4BA$), then 
\begin{equation}\label{EqPentTypeBAdd}
F_{\hbar}(\Ab\boxplus \Bb) = F_{\hbar}(\Ab) F_{2\hbar}(\Ab\star\Bb) F_{\hbar}(\Bb),
\end{equation} 
and
\begin{equation}\label{EqPentTypeBAdd2}
F_{\hbar}((\Ab^{1/2} \boxplus \Bb^{1/2})^2) = F_{\hbar}(\Ab)F_{\hbar/2}(\Ab^{1/2}\star\Bb^{1/2}) F_{\hbar}(\Bb).
\end{equation} 
This follows from applying the respective identities in \eqref{EqFormSumIdAlt} and rescaling.
\end{Rem}


\begin{thebibliography}{00}
\bibitem[Ar07]{Ar07} T. Arakawa, Representation theory of $\mathscr{W}$-algebras, \emph{Invent. Math.} \textbf{169} (2007), 219--320.
\bibitem[BS93]{BS93} S. Baaj and G. Skandalis, Unitaires multiplicatifs et dualit\'{e} pour les produits crois\'{e}s de C$^*$-alg\`{e}bres, \emph{Ann. Sci. École Norm. Sup.} (4) \textbf{26} (4), 425--488.
\bibitem[BV05]{BV05} S. Baaj and S. Vaes,  Double crossed products of locally compact quantum groups, \emph{J. Inst. Math. Jussieu} \textbf{4} (1) (2005), 135--173.
\bibitem[Ber96]{Ber96} A. Berenstein, Group-like elements in quantum groups and Feigin's conjecture, \emph{preprint} (1996), arxiv:q-alg/9605016v1.
\bibitem[BZ05]{BZ05} A. Berenstein and A. Zelevinsky, Quantum cluster algebras, \emph{Adv. Math.} \textbf{195} (2) (2005), 405--455.
\bibitem[Bu00]{Bu00} J.M. Burns, An elementary proof of the `strange formula' of Freudenthal and de Vries,
\emph{Quart. J. Math.} \textbf{51} (2000), 295--297.
\bibitem[BT03]{BT03} A. G. Bytsko and J. Teschner, $R$-Operator, Co-Product and Haar-Measure for the Modular Double
of $U_q(\mathfrak{sl}(2,\R))$, \emph{Commun. Math. Phys.} \textbf{240} (1) (2003), 171--196.
\bibitem[CDN26]{CDN26} T. Creutzig, G. Dhillon and S. Nakatsuka, The Kazhdan-Lusztig category of $W$-algebras of simply-laced Lie algebras at irrational levels, \emph{preprint} (2026), arXiv:2603.04667.
\bibitem[CL26]{CL26} T. Creutzig, S. D. Lentner, Affine $\mathfrak{sl}_2$ at admissible levels and quantum $\mathfrak{sl}_{2\mid 1}$, \emph{preprint} (2026), arXiv:2609.12895.
\bibitem[DCDV10]{DCDV10} K. De Commer and A. Van Daele, Multiplier Hopf Algebras Imbedded in Locally Compact Quantum Groups, \emph{Rocky Mountain J. Math.} \textbf{40} (4) (2010), 1149--1182.
\bibitem[DCSSV25]{DCSSV25} K. De Commer, G. Schrader, A. Shapiro and C. Voigt, Quantum $SL^+(N,\mathbb{R})$ as a locally compact quantum group, \emph{preprint}, arXiv:2512.21579.
\bibitem[DK94]{DK94} M. S. Dijkhuizen and T.H. Koornwinder, CQG algebras: A direct algebraic approach to compact quantum groups, \emph{Lett. Math. Phys.} \textbf{32} (1994), 315 -- 330.
\bibitem[ER94]{ER94} E.G. Effros and  Z.-J. Ruan, Discrete quantum groups. I. The Haar measure, \emph{Internat. J. Math.} \textbf{5} (1994), 681--723.
\bibitem[ES98]{ES98} P. Etingof and O. Schiffmann, Lectures on Quantum Groups, \emph{Lectures in Math. Phys., International Press} (1998).
\bibitem[Fad00]{Fad00} L.D. Faddeev, Modular double of quantum group, \emph{Math. Phys. Stud.} \textbf{21} (2000), 149--156.
\bibitem[FF91]{FF91} B. Feigin and E. Frenkel, Duality in $W$-algebras, \emph{Internat. Math. Res. Not.} \textbf{1991} (6) (1991), 75--82.
\bibitem[FG06]{FG06} V. Fock and A. Goncharov, Moduli spaces of local systems and higher Teichm\"{u}ller theory, \emph{Publications Math\'{e}matiques de l’IH\'{E}S} \textbf{103} (2006), 1--211.
\bibitem[FG09]{FG09}  V. V. Fock and A. B. Goncharov, The quantum dilogarithm and representations of quantum cluster varieties, \emph{Invent. math.} \textbf{175} (2) (2009), 223--286.
\bibitem[FZ02]{FZ02} S. Fomin and A. Zelevinsky, Cluster Algebras I: Foundations, \emph{J. Amer. Math. Soc.} \textbf{15} (2) (2002), 497--529. 
\bibitem[FI14]{FI14} I. Frenkel and  I. Ip, Positive representation of split real quantum groups and future perspectives, \emph{Int. Math. Res. Not.} \textbf{2014} (8) (2014), 2126--2164.
\bibitem[FdV69]{FdV69} H. Freudenthal and H. de Vries, Linear Lie Groups, \emph{Academic Press, New York} (1969).
\bibitem[GS19]{GS19} A.B.  Goncharov and L. Shen, Quantum geometry of moduli spaces of local systems and representation theory, \emph{preprint}, arxiv:1904.10491v4.
\bibitem[Ip13a]{Ip13a} I.\,Ip, Representation of the quantum plane, its quantum double and harmonic analysis on $GL_q^+(2,\R)$, \emph{Selecta Math. (N.S.)} \textbf{19} (4) (2013), 987--1082.
\bibitem[Ip13b]{Ip13b} I. Ip, The classical limit of representation theory of the quantum plane, \emph{Int. J. Math.} \textbf{24} (4) (2013), 1350031.
\bibitem[Ip15a]{Ip15a} I. Ip, Positive representations of non-simply-laced split real quantum groups, \emph{J. Algebra} \textbf{425} (2015), 245--276.
\bibitem[Ip15b]{Ip15b} I. Ip, Positive representations of split real quantum groups: the universal $R$ Operator,
\emph{Int. Math. Res. Not.} \textbf{2015} (1) (2015), 240--287.
\bibitem[Ip15c]{Ip15c} I. Ip, Gauss-Lusztig decomposition of $GL_q^+(N,\mathbb{R})$ and representations by $q$-tori, \emph{J. Pure Appl. Algebra} \textbf{219} (12) (2015), 5650--5672.
\bibitem[Ip18a]{Ip18a} I. Ip, Cluster realization of $U_q(\mathfrak{g})$ and factorizations of the universal $R$-matrix, \emph{Selecta Math. (N.S.)} \textbf{24} (2018), 4461--4553. 
\bibitem[Ip18b]{Ip18b} I. Ip, On tensor products of positive representations of split real quantum Borel subalgebra $U_{qq'}(\mfb_{\R})$,
\emph{Trans. Amer. Math. Soc.} \textbf{370} (6) (2018), 4177--4200.
\bibitem[Ip19]{Ip19} I. Ip, Cluster Realization of positive representations of split real quantum Borel subalgebra, \emph{
Theoret. and Math. Phys.} \textbf{198} (2) (2019), 246--272.
\bibitem[Ip20]{Ip20} I. Ip, Positive representations of split real simply-laced quantum groups, \emph{Publ. Res. Inst. Math. Sci.} \textbf{56} (3) (2020), 603--646.
\bibitem[Jan96]{Jan96} J.C. Jantzen,  Lectures on quantum groups, \emph{Graduate Studies in Mathematics}, vol. \textbf{6}, American Mathematical Society, Providence, RI (1996).
\bibitem[Jos95]{Jos95} A. Joseph, Sur une conjecture de Feigin, \emph{C. R. Acd. Sci. Paris S\'{e}rie I} \textbf{320} (1995), 1441--1444.
\bibitem[KLS02]{KLS02} S. Kharchev, D. Lebedev, and M. Semenov-Tian-Shansky, Unitary representations of $U_q(\mathfrak{sl}(2,\R))$,
the modular double, and the multiparticle $q$-deformed Toda chains, \emph{Commun. Math. Phys.} \textbf{225} (2002),
573--609.
\bibitem[Ka96]{Ka96}  R. M. Kashaev, The Heisenberg double and the pentagon relation, \emph{Algebra i Analiz} \textbf{8} (4) (1996), 63--74; \emph{St. Petersburg Math. J.} \textbf{8} (4) (1997), 585--592.
\bibitem[KT91]{KT91} S.M.\ Khoroshkin and  V.N.\ Tolstoy, Universal $R$-matrix for quantized (super) algebras, \emph{Commun. Math.
Phys.} \textbf{141} (1991), 599--617.
\bibitem[KV00]{KV00} J.\,Kustermans and S.\,Vaes, Locally compact quantum groups, \emph{Ann. Sci. \'{E}c. Norm. Sup\'{e}r.} \textbf{33} (6) (2000), 837--934.
\bibitem[KV03]{KV03} J.\,Kustermans and S.\,Vaes, Locally compact quantum groups in the von Neumann algebraic setting, \emph{Math. Scand.} \textbf{92} (1) (2003), 68--92.
\bibitem[LS91]{LS91} S. Levendorskii and Y. Soibelman, Algebras of Functions on Compact Quantum Groups, Schubert Cells and Quantum Tori, \emph{Commun. Math. Phys.} \textbf{139} (1991), 141--170.
\bibitem[Lit98]{Lit98} P. Littelmann, Cones, crystals and patterns, \emph{Transform. Groups}  \textbf{3} (2) (1998), 145--179.
\bibitem[Lus94a]{Lus94a} 
G. Lusztig, Total positivity in reductive groups, in: \emph{Lie theory and geometry, in
honor of Bertram Kostant, Progress in Mathematics} \textbf{123}, Birkh\"{a}user (1994).
\bibitem[Lus94b]{Lus94b} G. Lusztig, Introduction to quantum groups, Birkh\"{a}user, Boston (1994).
\bibitem[MRW14]{MRW14} R. Meyer, S. Roy and S.L. Woronowicz, Quantum group-twisted tensor products of C$^*$-algebras, \emph{Internat. J. Math.} \textbf{25} (2) (2014), 1450019, 37 pp.
\bibitem[RS81]{RS81} M. Reed and B. Simon, Methods of Modern Mathematical Physics I: Functional Analysis, \emph{Academic Press} (1981).
\bibitem[Rui97]{Rui97} S. N. M. Ruijsenaars, First order analytic difference equations and integrable quantum systems, \emph{J. Math. Phys.} \textbf{38} (2) (1997), 1069--1146.
\bibitem[Sch94]{Sch94} K. Schm\"{u}dgen, Integrable Operator Representations of $\R_q^2, X_{q,\gamma}$ and $SL_q(2,R)$, \emph{Comm. Math. Phys.} \textbf{159} (2) (1994), 217--237.
\bibitem[SS17]{SS17} G. Schrader and A. Shapiro, Continuous tensor categories from quantum groups I: algebraic aspects,
 \emph{preprint} (2017), arXiv:1708.08107.
\bibitem[SW01]{SW01} P.\,M.\,So\l tan and S.\,L.\,Woronowicz, A remark on manageable multiplicative unitaries, \emph{Lett. Math. Phys.} \textbf{57} (2001), 239--252.
\bibitem[SW07]{SW07} P.\,M.\,So\l tan and S.\,L.\,Woronowicz, From multiplicative unitaries to quantum groups II, \emph{J. Funct. Anal.} \textbf{252} (1) (2007), 42--67.
\bibitem[Tak02]{Tak02} M. Takesaki, Theory of operator algebras I,  \emph{Encyclopaedia of Mathematical Sciences} \textbf{24},
\emph{Springer–Verlag, Berlin} (2002).
\bibitem[Tak03]{Tak03} M. Takesaki, Theory of operator algebras II,  \emph{Encyclopaedia of Mathematical Sciences} \textbf{125},
\emph{Springer–Verlag, Berlin} (2003).
\bibitem[Tes11]{Tes11} J. Teschner, Quantization of the Hitchin moduli
spaces, Liouville theory and the
geometric Langlands
correspondence $I$, \emph{Adv. Theor. Math. Phys.}  \textbf{15} (2011), 471--564.
\bibitem[VVD03]{VVD03} S. Vaes and A. Van Daele, The Heisenberg commutation relations, commuting squares and the
Haar measure on locally compact quantum groups, \emph{In: Operator algebras and mathematical physics: conference Proceedings, Constanta (Romania), July 2-7, 2001, Theta Foundation, Bucarest} (2003), 379--400.
\bibitem[VD96]{VD96} [17] A. Van Daele, Discrete quantum groups, \emph{J. Algebra} \textbf{180} (1996), 431--444.
\bibitem[VD98]{VD98} A. Van Daele, An Algebraic Framework for Group Duality, \emph{Adv. Math.}  \textbf{140} (2)
(1998), 323--366.
\bibitem[VY20]{VY20} C. Voigt and R. Yuncken, Complex Semisimple Quantum Groups and Representation Theory, \emph{Lecture Notes in Mathematics} \textbf{2264}, Springer International Publishing, 376+X pages.  
\bibitem[Wor87]{Wor87} S. L. Woronowicz, Compact matrix pseudogroups, \emph{Comm. Math. Phys.} \textbf{111} (1987), 613--665.
\bibitem[Wor98]{Wor98} S. L. Woronowicz, Compact quantum groups, In \emph{Sym\'{e}tries quantiques (Les Houches, 1995)}, 845--884. North-Holland, Amsterdam, 1998.
\bibitem[Wor00]{Wor00} S.\,L.\,Woronowicz, Quantum exponential function, \emph{Rev. Math. Phys.} \textbf{12} (6) (2000), 873--920.
\bibitem[WZ02]{WZ02} S.\,L.\,Woronowicz and S.\,Zakrzewski, Quantum `$ax+b$' group, 
\emph{Rev. Math. Phys.} \textbf{14} (2002), 797--828.
\end{thebibliography}
\end{document}